\documentclass{amsart}
\usepackage{amssymb,amsmath,exscale,latexsym,amsthm,graphics,overpic}
\usepackage{epsfig}    
\usepackage{epic}      
\usepackage{eepic}     
\usepackage[matrix,arrow,curve,cmtip]{xy} 
\usepackage{enumerate} 
\usepackage{paralist}

\theoremstyle{plain}
        \newtheorem{theorem}{Theorem}[section]
        \newtheorem*{theorem*}{Theorem}
        \newtheorem*{conj*}{Conjecture}
        \newtheorem{lemma}[theorem]{Lemma}
        \newtheorem{cor}[theorem]{Corollary}
        \newtheorem{prop}[theorem]{Proposition}

\theoremstyle{definition}
        \newtheorem{definition}[theorem]{Definition}
        
        \newtheorem{open}{Open Problem}

\theoremstyle{remark}
        \newtheorem*{remark}{Remark}
        
        \newtheorem*{notation}{Notation}

        \newtheorem*{claim}{Claim}
        \newtheorem*{claim1}{Claim 1}
        \newtheorem*{claim2}{Claim 2}
        \newtheorem*{claim3}{Claim 3}
                
        \newtheorem*{case}{Case}
        \newtheorem*{convention}{Convention}

\numberwithin{equation}{section}

\providecommand{\defn}[1]{\emph{#1}}

\newcounter{mylistnum}
\newcounter{mylistnum2}

\newcommand{\mesh}{\operatorname{mesh}}

\newcommand{\diam}  {\operatorname{diam}}

\newcommand{\inte}  {\operatorname{int}}

\newcommand{\id} {\operatorname{id}}
\newcommand{\R}{\mathbb{R}}      
\newcommand{\C}{\mathbb{C}}      
\newcommand{\N}{\mathbb{N}}      
\newcommand{\Z}{\mathbb{Z}}      
\newcommand{\Q}{\mathbb{Q}}      
\newcommand{\T}{\mathbb{T}}      
\newcommand{\CDach}{\widehat{\mathbb{C}}}
\newcommand{\RDach}{\widehat{\mathbb{R}}}
\newcommand{\D}{\mathbb{D}}      
\newcommand{\Dbar}{\overline{\mathbb{D}}}      

\providecommand{\abs}[1]{\lvert#1\rvert}

\providecommand{\norm}[1]{\lVert#1\rVert}

\renewcommand{\:}{\colon}   
\newcommand{\ra}{\rightarrow}

\newcommand{\crit}{\operatorname{crit}}
\newcommand{\post}{\operatorname{post}}

\newcommand{\CC}{\mathcal{C}}

\newcommand{\EC}{\mathcal{E}}

\newcommand{\X} {\mathbf{X}}

\newcommand{\E} {\mathbf{E}}

\newcommand{\V} {\mathbf{V}}

\newcommand{\W} {\mathbf{W}}

\newcommand{\odd}{\operatorname{odd}}
\newcommand{\even}{\operatorname{even}}

\newcommand{\Sim}[1]{\stackrel{#1}{\sim}}

\newcommand{\U}{\mathcal{U}}
\newcommand{\I}{\mathcal{I}}
\newcommand{\J}{\mathcal{J}}

\begin{document}

\title{Invariant Peano curves of expanding Thurston maps}

\date{\today} 

\author{Daniel Meyer}
\thanks{The author was partially supported by an NSF postdoctoral
  fellowship, the Swiss National Science Foundation, and the Academy of
  Finland (projects SA-11842 and SA-118634).} 
\address{Jacobs University, Campus Ring 1, 28759 Bremen, Germany}  
\email{dmeyermail@gmail.com}  

 

\begin{abstract}
  We consider \emph{Thurston maps}, i.e., branched covering maps
  $f\colon S^2\to S^2$ that are \emph{postcritically finite}. 
  In addition, we assume that $f$ is \emph{expanding} in a suitable
  sense. 
  It is shown that each sufficiently high iterate $F=f^n$ of $f$ is
  \emph{semi-conjugate} to $z^d\colon S^1\to
  S^1$, where $d=\deg F$. More precisely, for such an $F$ we construct
  a \emph{Peano curve} $\gamma\colon S^1\to S^2$ (onto), such that $F\circ
  \gamma(z) = \gamma(z^d)$ (for all $z\in S^1$).    
\end{abstract}

\maketitle

\tableofcontents
\section{Introduction}
\label{sec:introduction}
A \defn{Thurston map} is a branched covering of the sphere $f\colon
S^2\to S^2$ that is \defn{postcritically finite}. 
A celebrated theorem of Thurston gives a \emph{topological
  characterization} of rational maps among Thurston maps (see
\cite{DouHubThurs}).  
In this paper we consider such
maps that are \defn{expanding} (see Section
\ref{sec:thurston-maps-as} for precise definitions). In the case when
$f$ is a rational map this means that the Julia set of $f$ is the
whole sphere. 

The main theorem is the following.

\begin{theorem}
  \label{thm:main}
  Let $f$ be an expanding Thurston map. Then for each sufficiently
  high iterate $F=f^n$ there is a 
  \emph{Peano curve} $\gamma\colon S^1 \to S^2$ (onto) such that
  $F(\gamma(z))= \gamma(z^d)$ (for all $z\in S^1$). Here $d=\deg
  F$. This means that 
  the following diagram commutes.
  \begin{equation*}
    \xymatrix{
      S^1 \ar[r]^{z^d} \ar[d]_{\gamma}
      &
      S^1 \ar[d]^{\gamma}
      \\
      S^2 \ar[r]_F & S^2
    }
  \end{equation*}
  Furthermore, we can approximate the Peano curve $\gamma$ as follows.
  There is a \emph{homotopy}
  $\Gamma\colon S^2\times [0,1]\to S^2$, with $\Gamma(z,0)=z$, such that
  \begin{equation*}
    \Gamma(z,1)=\gamma(z) \text{ for all } z\in S^1.
  \end{equation*}
  Here we view $S^1\subset S^2$ as the equator. 
\end{theorem}
In fact $\Gamma$ may be chosen to be a \emph{pseudo-isotopy}, meaning
it is an isotopy on $[0,1)$. 

\smallskip
The result may be
paraphrased as follows. Via $\gamma$ we can view the sphere $S^2$ as
a parametrized circle $S^1$. Wrapping this parametrized circle (which
is $S^2$) around itself $d$ times yields the map $F$. 

The existence of such a \defn{semi-conjugacy} $\gamma$ as above
follows for many rational maps $F$ of degree $2$ by work of Tan Lei,
M.~Rees, and M.~Shishikura (see \cite{MR1182664}, \cite{MR1149864},
and \cite{MR1765095}); the relevant construction of \emph{mating} is
reviewed  in Section~\ref{sec:cons-theor}. Milnor 
constructs such a Peano curve $\gamma$ (i.e., semi-conjugacy) for
one specific example $F$ (see \cite{MilnorMating}) in this setting. 
Kameyama gives a
sufficient criterion for the existence of $\gamma$ (in
\cite[Theorem~3.5]{MR1961296}). 

Note that the result is purely topological, i.e., does not depend
on $F$ being (equivalent to) a rational map or not.

\smallskip
We also prove the following converse statement to Theorem~\ref{thm:main}.
\begin{theorem}
  \label{thm:peano_implies_exp}
  Let $f\colon S^2\to S^2$ be a Thurston map such that for some
  iterate $F=f^n$ there exists a Peano curve $\gamma\colon S^1\to S^2$
  (onto) satisfying $F(\gamma(z))= \gamma(z^d)$ for all $z\in
  S^1$. Then $f$ is expanding.  
\end{theorem}

\smallskip
According to \emph{Sullivan's dictionary} there is a close
correspondence between the dynamics of rational maps and of Kleinian
groups \cite{MR806415}. Cannon-Thurston construct (in
\cite{MR2326947}) an invariant Peano
curve $\gamma\colon S^1\to S^2$ for the fundamental group of a
(hyperbolic) 
$3$-manifold $M^3$ that \defn{fibers over the circle}. 
Theorem~\ref{thm:main} may be viewed as the 
corresponding result in the case of rational maps. 
Thus it provides
another entry in Sullivan's dictionary.

\subsection{Group invariant Peano curves}
\label{sec:group-invar-peano}

We review the Cannon-Thurston
construction from \cite{MR2326947}. The purpose is to put
Theorem~\ref{thm:main} into perspective. 

Let $\Sigma$ be a compact hyperbolic $2$-manifold, and $\varphi\colon
\Sigma \to \Sigma$ be a \emph{pseudo-Anosov} homeomorphism. Consider
the equivalence relation on the product $\Sigma\times [0,1]$ given by 
$(x,0)\sim (\varphi(x),1)$. Then the $3$-manifold $M^3:= \Sigma\times
[0,1]/\sim$ is called a \emph{manifold that fibers over the
  circle}. Thurston has proved that $M^3$ admits a \emph{hyperbolic}
metric, see \cite{MR1855976}. 

\smallskip
The fundamental groups $\pi_1(\Sigma), \pi_1(M^3)$ are \emph{Gromov hyperbolic},
see \cite{MR919829} as well as \cite{MR1086648}. Thus they have 
\emph{boundaries at infinity}, which in this case are
$\partial_{\infty}\pi_1(\Sigma)= S^1$ and $\partial_{\infty}\pi_1(M^3)=S^2$.

This is seen by noting that $\pi_1(\Sigma)$ and 
hyperbolic $2$-space $\mathbb{H}^2$, as well
as $\pi_1(M^3)$ and hyperbolic $3$-space $\mathbb{H}^3$, are
\emph{quasi-isometric}. The 
boundary at infinity of $\mathbb{H}^2$ is $S^1$, the boundary at
infinity of $\mathbb{H}^3$ is $S^2$, the boundary of the disk,
respectively the
unit ball, in the Poincar\'{e} model of hyperbolic space.

\smallskip
The inclusion $\Sigma \to \Sigma\times\{0\} \to M^3$ induces an
inclusion of the fundamental groups $\imath \colon\pi_1(\Sigma)\to
\pi_1(M^3)$, which is a group homomorphism. In fact
$\imath(\pi_1(\Sigma))$ is a normal subgroup of $\pi_1(M^3)$. 
The map $\imath$ extends to the
boundaries at infinity $S^1=\partial_{\infty} \pi_1(\Sigma)$,
$S^2=\partial_{\infty}\pi_1(M^3)$ to a continuous map $\sigma\colon
S^1\to S^2$. 

It is well-known (and not very hard to show), that a non-trivial
normal subgroup $N\vartriangleleft G$ of a Gromov hyperbolic group $G$
has
the same boundary at infinity as $G$. Thus $\partial_{\infty}
\imath(\pi_1(\Sigma))=\partial_\infty (\pi_1(M^3))=S^2$. It follows
that the map $\sigma$ is \emph{onto}, i.e., a Peano curve. 

\smallskip
Each element $g\in \pi_1(\Sigma)$ acts (by left-multiplication) on
$\pi_1(\Sigma)$; this action extends to $S^1=\partial_{\infty}
\pi_1(\Sigma)$. Similarly each element $g\in \pi_1(M^3)$ acts on
$\pi_1(M^3)$ and this action extends to $S^2
= \partial_{\infty}\pi_1(M^3)$. The map $\sigma$ is \emph{invariant}
with respect to this group action, meaning that for every $g\in
\pi_1(\Sigma)$ it holds that $\imath(g) (\sigma(t))=
\sigma(g(t))$ for all $t\in S^1$. Thus the following diagram
commutes. 

\begin{equation*}
  \xymatrix{
    S^1 \ar[r]^{g} \ar[d]_{\sigma}
    &
    S^1 \ar[d]^{\sigma}
    \\
    S^2 \ar[r]_{\imath(g)} & S^2
  }
\end{equation*}

The invariant Peano curve $\gamma$ from Theorem~\ref{thm:main} is the
corresponding object to the group invariant Peano curve $\sigma$
according to Sullivan's dictionary. 

The Cannon-Thurston construction has been extended by Minsky in
\cite{MR1257060} and McMullen in \cite{MR1859018} to (some) cases
where $\Sigma$ is not compact.

\smallskip
In \cite{MR648524} Thurston asked whether (in
a sense) all hyperbolic $3$-manifolds arise as manifolds that fiber
over the circle. This has now become known as the \emph{virtual
  fibering conjecture}. It stipulates that every hyperbolic
$3$-manifold has a finite cover which fibers over the circle. This
would mean that we can understand every hyperbolic $3$-manifold in
terms of $2$-manifolds. See \cite{MR903863} for more background on
this conjecture, \cite{MR2399130} for recent progress. 

Theorem~\ref{thm:main} may be viewed as the solution of the problem
corresponding to the virtual fibering conjecture according to
Sullivan's dictionary. 


\subsection{Consequences of Theorem  \ref{thm:main}}
\label{sec:cons-theor}

To not further increase the size of the present paper, we will develop 
the implications of the main theorem in a follow-up paper
\cite{exp_quotients}.  
They are outlined here briefly to put the result into perspective. 

\smallskip
Using the invariant Peano curve $\gamma \colon S^1\to S^2$ from
Theorem \ref{thm:main}, an
equivalence relation on $S^1$ is defined by
\begin{equation}
  \label{eq:eq_rel}
  s\sim t \Leftrightarrow \gamma(s)=\gamma(t),
\end{equation}
for all $s,t\in S^1$. Elementary topology yields that $S^1/\!\sim$ is
homeomorphic to $S^2$ and that $z^d/\!\sim\colon S^1/\!\sim \,\to
S^1/\!\sim$ 
is topologically conjugate to the map $F$. 
\begin{theorem}
  \label{thm:S1simS2}
  The following diagram commutes,
  \begin{equation*}
    \xymatrix{
      S^1/\!\sim \ar[r]^{z^d/\sim} \ar[d]_{h}
      &
      S^1/\!\sim \ar[d]^{h}
      \\
      S^2 \ar[r]_F & S^2.
    }
  \end{equation*}
  Here the homeomorphism $h\colon S^1/\!\sim \,\to S^2$ is given by $h\colon
  [s]\mapsto \gamma(s)$, for all $s\in S^1$. 
\end{theorem}

The equivalence relation (\ref{eq:eq_rel}) may be constructed from
\emph{finite data}, more precisely from two finite families of finite
sets of rational numbers. 

The proper setting is as follows. 
For each $n\in \N$ two equivalence relations $\Sim{n,w}, \Sim{n,b}$ are
defined. The equivalence relation $\sim$ defined in (\ref{eq:eq_rel})
is the \emph{closure} of the union of all $\Sim{n,w}, \Sim{n,b}$. Each
$\Sim{n,w}$ is the \emph{pullback} of $\Sim{n-1,w}$ by $z^d$ (similarly
$\Sim{n,b}$ is the pullback of $\Sim{n-1,b}$). Thus $F$ can be
recovered (up to topological conjugacy) from the equivalence
relations $\Sim{1,w},\Sim{1,b}$. 


This provides a way to \defn{describe} expanding Thurston maps
effectively.

\smallskip
The description above may be viewed as a two-sided version of the
viewpoint introduced by Douady-Hubbard and Thurston
(\cite{DHOrsayI}, \cite{DHOrsayII}, \cite{ThurstonCombinatorics},
\cite{MR2508255}, see also \cite{MR1149864} and \cite{MR1761576}),
namely the combinatorial description of Julia sets in terms of
\emph{external rays}.  

\smallskip
Recently (analogously defined) \emph{random laminations} have been used
to study the scaling 
limits of planar maps (see \cite{MR2336042}, \cite{MR2438999}).

\smallskip
The description of $F$ as above yields in addition that $F$ arises as
a \defn{mating} of two polynomials. Mating of polynomials was
introduced by Douady and Hubbard \cite{MR728980} as a way to
geometrically combine two polynomials to form a rational map. We recall
the construction briefly.

Consider two monic polynomials $p_1$ and $p_2$ of the same degree with
connected and locally connected Julia sets. 
Let $K_1$ and $K_2$ be their filled-in Julia sets. For $j=1,2$ 
let 
\begin{equation*}
  \phi_{j}\colon \CDach\setminus \Dbar\to \CDach\setminus K_{j} 
\end{equation*}
be the Riemann maps, normalized by
$\phi_{j}(\infty)=\infty$ and 
\begin{equation*}
  \phi'_{j}(\infty)= \lim_{z\to \infty}z/\phi_{j}(z) >0
\end{equation*}
(in fact then 
$\phi'_{j}(\infty)=1$). 
By \emph{Carath\'{e}odory's theorem}
$\phi_{j}$ extends continuously 
to 
\begin{equation*}
  \sigma_{j}\colon S^1=\partial \Dbar\to\partial K_{j}. 
\end{equation*}
The \defn{topological mating} of $K_1$ and $K_2$ is obtained by identifying
$\sigma_1(z)\in \partial K_1$ with $\sigma_2(\bar{z})\in \partial
K_2$. More precisely, we consider the 
disjoint union of $K_1$ and $K_2$ and let
$K_1\amalg K_2$ be the quotient obtained from the equivalence
relation generated by $\sigma_1(z)\sim \sigma_2(\bar{z})$ (for all $z\in
S^1=\partial \D$). The map
\begin{align*}
  &p_1\amalg p_2 \colon K_1\amalg K_2 \to K_1\amalg K_2,
  \intertext{given by}
  &(p_1\amalg p_2)|_{K_j}=p_j, \quad \text{for } j=1,2, 
\end{align*}
is well defined. If a map $f$ is topologically conjugate
to $p_1\amalg p_2$, we say that $f$ is obtained as a (topological)
mating. If both $K_1$ and $K_2$ have empty interior each of the maps
$\sigma_1$ and $\sigma_2$
descends to a Peano curve $\gamma\colon S^1 \to K_1\amalg K_2$ which
provides a semi-conjugacy of $z^d \colon S^1\to S^1$ to $p_1\amalg
p_2$ (here $d=\deg p_1=\deg p_2$).  

\smallskip
In particular it is known (see \cite{MR1182664}, \cite{MR1765095}, and
\cite{MR1149864}) that the \defn{mating} of two quadratic
polynomials $p_1=z^2+c_1$, $p_2=z^2+c_2$, where $c_1,c_2$ are
\emph{Misiurewicz points} (i.e., the critical point $0$ is strictly
preperiodic for $p_i$) not contained in conjugate limbs of the
Mandelbrot set, results in a map that is topologically conjugate to a
rational map $F$. The filled-in Julia sets of $p_1,p_2$ have empty
interior. The Julia set of $F$ is the whole sphere, hence $F$
is expanding.  
Thus a Peano curve $\gamma$ as in Theorem
\ref{thm:main} exists for such a map $F$. 

\smallskip
Recall that a \defn{periodic critical
  point} (of 
a Thurston map $f$) is a critical point $c$, such that $f^k(c)=c$ for
some $k\geq 1$. 
\begin{theorem}[\cite{exp_quotients}]
  \label{thm:mating1}
  Let $f\colon S^2\to S^2$ be an expanding Thurston map without
  periodic critical points. Then every sufficiently high iterate
  $F=f^n$ is obtained as a \defn{topological mating} of two
  polynomials.   
\end{theorem}

If at least one of the filled-in Julia sets $K_1,K_2$ has non-empty
interior, we can take a further quotient of $K_1\amalg K_2$ by
identifying the points of the closure of each bounded Fatou
component. Technically we take the \emph{closure} of the equivalence
relation (on the disjoint union of $K_1,K_2$) obtained from
$\sigma_1(z)\sim \sigma_2(\bar{z})$ (for all $z\in S^1=\partial \D$)
as well as  
$x\sim y$ if $x,y$ are in the closure of the \emph{same} bounded Fatou
component of $p_1$ or $p_2$.  

The maps $p_1,p_2$ descend to the quotient map
$p_1\widehat{\amalg} \,p_2$.  

\begin{theorem}[\cite{exp_quotients}]
  \label{thm:mating2}
  Let $f\colon S^2\to S^2$ be an expanding Thurston map with (at least
  one) periodic critical point. Then every sufficiently high iterate
  $F=f^n$ is topologically conjugate to a map $p_1\widehat{\amalg}\,
  p_2$ as above. 
\end{theorem}

The next theorem investigates the \emph{measure theoretic} mapping
properties of $\gamma$. 

\begin{theorem}[\cite{exp_quotients}]
  \label{thm:FmapsLebesgue}
  The Peano curve $\gamma$ maps \emph{Lebesgue measure} of $S^1$ to the
  \emph{measure of maximal entropy} (with respect to $F$) on $S^2$. 
\end{theorem}

The polynomials into which $F$ \emph{unmates}, i.e., the polynomials
$p_1, p_2$ from Theorem~\ref{thm:mating1} and
Theorem~\ref{thm:mating2} can be found by a simple explicit
combinatorial algorithm. This is explained in \cite{unmating}.

As another application of Theorem \ref{thm:main} one obtains
\emph{fractal tilings}. Namely divide the circle $S^1=\R/\Z$ into $d$
intervals $[j/d, (j+1)/d]$ ($j=0,\dots, d-1$). It follows from Theorem
\ref{thm:main} that $F$ maps each set $\gamma([j/d,(j+1)/d])$ to the
whole sphere. The tiling lifts to the \emph{orbifold covering}, which
is either the Euclidean or the hyperbolic plane. 

\subsection{Outline}
\label{sec:outline}
The construction of the invariant Peano curve, i.e., the
proof of Theorem \ref{thm:main}, forms the core
of this work.  

In Section \ref{sec:example} an example is introduced that serves
to illustrate the construction throughout the paper. 

Section \ref{sec:thurston-maps-as} gives precise definitions of
expanding Thurston maps, as well as gathers facts from
\cite{expThurMarkov} relevant here. 

We will fix a Jordan curve $\CC$ containing the set of all postcritical
points ($=\post(F)$). We construct
\emph{approximations} $\gamma^n\colon S^1\to S^2$, that
will go through $F^{-n}(\CC)$. The limit $\gamma=\lim_n \gamma^n$ will
be the desired Peano curve. 

The construction of $\gamma$ consists of two parts. 
In the first part (which is logically the second)
we assume that we can deform $\CC$ by a \emph{pseudo-isotopy
  rel.\ $\post(F)$} to $\gamma^1=F^{-1}(\CC)$. The
approximations $\gamma^n$ can then be constructed inductively by
repeated lifts. This is done 
in Section \ref{sec:appr-gn}. 

The correct \emph{parametrization} of $\gamma^n$ is done in Section
\ref{sec:constr-g}.

\medskip
The second part is the construction of the
pseudo-isotopy $H^0$ rel.\ $\post(F)$, which deforms the Jordan curve
$\CC$ to the first approximation $\gamma^1$. 

We \emph{color} one component of
$S^2\setminus \CC$ white, the other black. Preimages of these Jordan
domains by $F$ then form the \defn{black/white $1$-tiles}. 

At each vertex (of $1$-tiles) we will declare which white/black
$1$-tiles are \defn{connected}. These connections will be described by
\defn{complementary non-crossing partitions}. 

Connections at all vertices will be defined in
such a way that the \defn{white tile graph} forms a \defn{spanning
  tree}. The ``outline'' of this spanning tree forms the first
approximation $\gamma^1$. The main work consists of making sure that
$\gamma^1$ lies in the right homotopy class (that $\CC$ can be
deformed to $\gamma^1$ by a pseudo-isotopy rel.\ $\post(F)$).   

\smallskip
Section \ref{sec:some-topol-lemm} assembles some standard topological
lemmas needed in the following.  

In Section \ref{sec:connections} the necessary background about
connections and com\-ple\-men\-ta\-ry
non-crossing partitions is developed.  

The desired pseudo-isotopy $H^0$ (equivalently the spanning tree of
white $1$-tiles) is
constructed in 
Section \ref{sec:construction-h0}. It is here that we (possibly) need
to take an iterate $F=f^n$ (in order to be in the right homotopy
class). 

In Section \ref{sec:comb-constr-gamm} an alternative
\emph{combinatorial} way to construct the approximations
$\gamma^n$ is presented. An \emph{$n$-tile} is the
preimage of a component of $S^2\setminus \CC$ by $F^n$. At each
\emph{$n$-vertex} of such an $n$-tile we define which $n$-tiles are
connected. Following the ``outline'' of one connected component as
before yields the approximation $\gamma^n$. These \emph{connections of 
  $n$-tiles} are constructed inductively in a purely combinatorial
fashion. 

\smallskip
Theorem~\ref{thm:peano_implies_exp} (existence of a Peano
curve which semi-conjugates $z^d$ to $F$ implies expansion) is proved
in Section~\ref{sec:invar-peano-curve}. 

\smallskip
The question arises whether it is necessary to take an iterate $F=f^n$
in Theorem~\ref{thm:main}.  While we do not have a definite answer, we
give an example in Section \ref{sec:an-example} which shows (in the
opinion of the author) that the answer is likely yes. More precisely,
for the considered example $h$ there exists no pseudo-isotopy $H^0$ as
required (there is one for the second iterate $h^2$).

We finish with some open problems in Section
\ref{sec:open-probl-concl}.  

\subsection{Acknowledgments}
\label{sec:acknolegements}
The author wishes to thank Juan Rivera-Letelier for many fruitful
discussions; Stanislav Smirnov, Mario Bonk, and Kari Astala for their
hospitality. Kevin Pilgrim and Tan Lei pointed out that Theorem
\ref{thm:main} should have a converse, i.e., that
Theorem~\ref{thm:peano_implies_exp} should hold.

\subsection{Example}
\label{sec:example}

We illustrate the proof using the following map $g$. 
It is a \defn{Latt\`{e}s map} (see \cite{Lattes},
\cite{milnor06:_lattes}). 

Map the square $[0,\frac{1}{2}]^2\subset \C$ to the upper half plane by a
Riemann map, normalized by mapping 
the vertices $0,\frac{1}{2}, \frac{1}{2} + \frac{1}{2}i,\frac{1}{2}i$ 
to $0,1,\infty,-1$. By Schwarz reflection this map can be extended 
to a meromorphic function $\wp\colon
\C\to \CDach$. This is the \defn{Weierstra\ss\ $\wp$-function} (up to a
M\"{o}bius transformation), it is (doubly) periodic with respect to
the lattice $L :=\Z^2$. Thus we may view $\wp$ as a (double) branched
covering map of the sphere by the torus $\T^2:=\C/L$.  

Color preimages of the upper half
plane by $\wp$ white, preimages of the lower half plane by $\wp$
black. The plane is then colored in a \defn{checkerboard} fashion. 
Consider the map 
\begin{align*}
  \psi\colon & \C\to \C,
  \\
  & z \mapsto 2 z.   
\end{align*}
We may view $\psi$ as a self-map of the torus $\T^2$. 
One checks that there is a (unique/well defined) map $g\colon \CDach\to
\CDach$ such that the diagram 
\begin{equation*}
  \xymatrix{
    \C \ar[r]^\psi \ar[d]_{\wp} &
    \C \ar[d]^{\wp}
    \\
    \CDach \ar[r]_g & \CDach
  }
\end{equation*}
commutes. The map $g$ is \defn{rational}, in fact
$g=4\frac{z(1-z^2)}{(z^2+1)^2}$.  
The Julia set of $g$ is the whole
sphere. 

\begin{figure}
  \centering
  \includegraphics[scale=0.5]{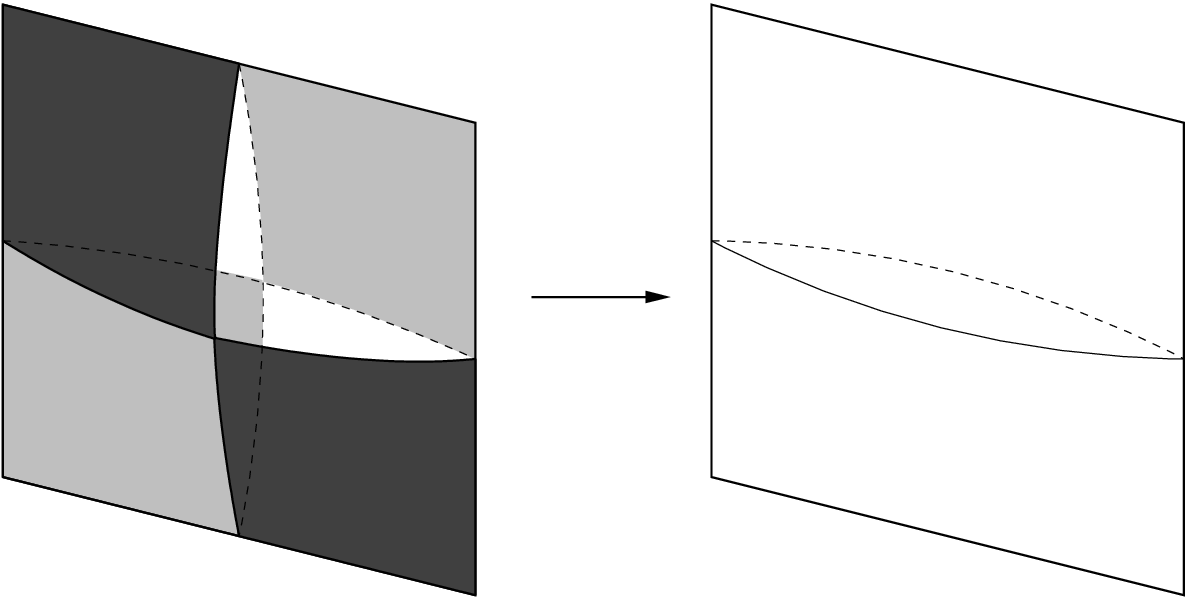}
  \begin{picture}(10,10)
    \put(-122,19){$\scriptstyle 0$}
    \put(0,-5){$\scriptstyle 1$}
    \put(-132,140){$\scriptstyle -1$}
    \put(-5,119){$\scriptstyle \infty$}
    \put(-172,-3){$\scriptstyle 1\mapsto 0$}
    \put(-240,5){$\scriptstyle \mapsto 1$}
    \put(-300,18){$\scriptstyle 0\mapsto 0$}
    \put(-313,80){$\scriptstyle \mapsto -1$}
    \put(-300,148){$\scriptstyle -1\mapsto 0$}
    \put(-233,135){$\scriptstyle \mapsto 1$}
    \put(-177,120){$\scriptstyle \infty \mapsto 0$}
    \put(-170,55){$\scriptstyle \mapsto -1$}
    \put(-148,80){$\scriptstyle g$}
    \put(-230, 65){$\scriptstyle \mapsto \infty$}
  \end{picture}
  \caption{The Latt\`{e}s map $g$.}
  \label{fig:mapg}
\end{figure}
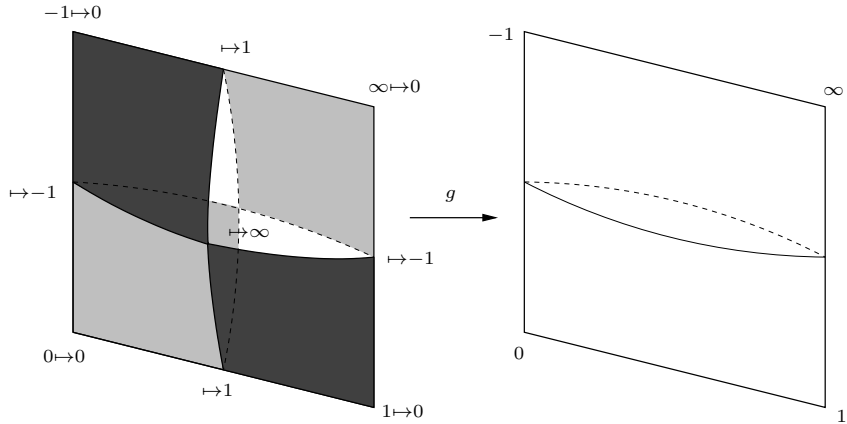

One may describe $g$ as follows. 
Push the Euclidean metric of $\C$ to the (Riemann) sphere
$\CDach$ by $\wp$. 
In this metric
the sphere looks like a \defn{pillow}
(technically this is an \defn{orbifold}, see for example 
\cite[Appendix~E]{MR2193309} and \cite[Appendix~A]{McM}).  
Indeed by construction the upper and lower half plane are then both  
isometric to the square
$[0,\frac{1}{2}]^2$.
Two such squares glued along their boundary form the sphere. 
We \defn{color} one of these squares (say the upper half plane)
\defn{white}, the other square (the lower half plane) \defn{black}.
The map $g$ is now given as follows. Divide each of the two squares into
$4$ small squares (of side-length $\frac{1}{4}$). Color these $8$ small
squares in a checkerboard 
fashion white and black. Map one such small white square to the big white
square. This 
extends by reflection to the whole pillow, which yields the map
$g$. There are obviously many different ways to color and map the small
squares. 
The ``right'' way to do so (in order to obtain
$g$) is indicated in Figure \ref{fig:mapg}.
 
The $6$ vertices of the small squares at which $4$ small squares
intersect are the \defn{critical points} of
$g$. They are mapped by $g$ to $\{1,\infty,-1\}$; these points in turn   
are mapped to $0$, which is a fixed point. The set
$\{0,1,\infty,-1\}=\post(g)$ is the set of all 
\defn{postcritical points}.
  
The map $\wp$ is the \defn{orbifold covering map}. The pictures
explaining our construction will all be in the \defn{orbifold
  covering}, i.e., in $\C$. For example the Peano curve will be
constructed by certain approximating curves. These are more easily
visualized when lifted to $\C$.  

\subsection{The construction for the example}
\label{sec:outline-construction}

The construction is explained using the example $g$ defined in the
last section. 

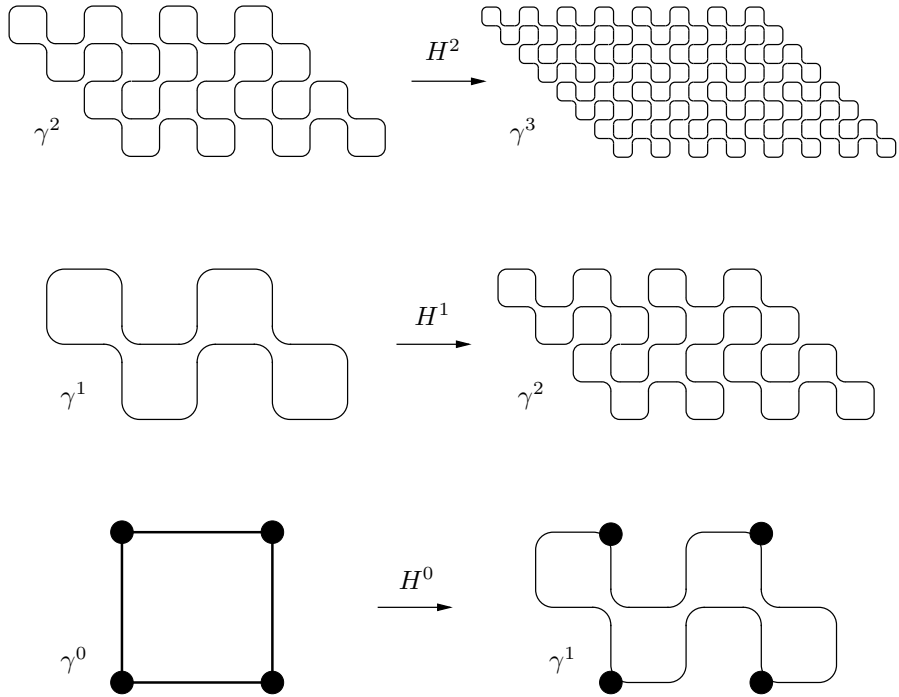
\begin{figure}
  \centering
  \includegraphics[scale=0.5]{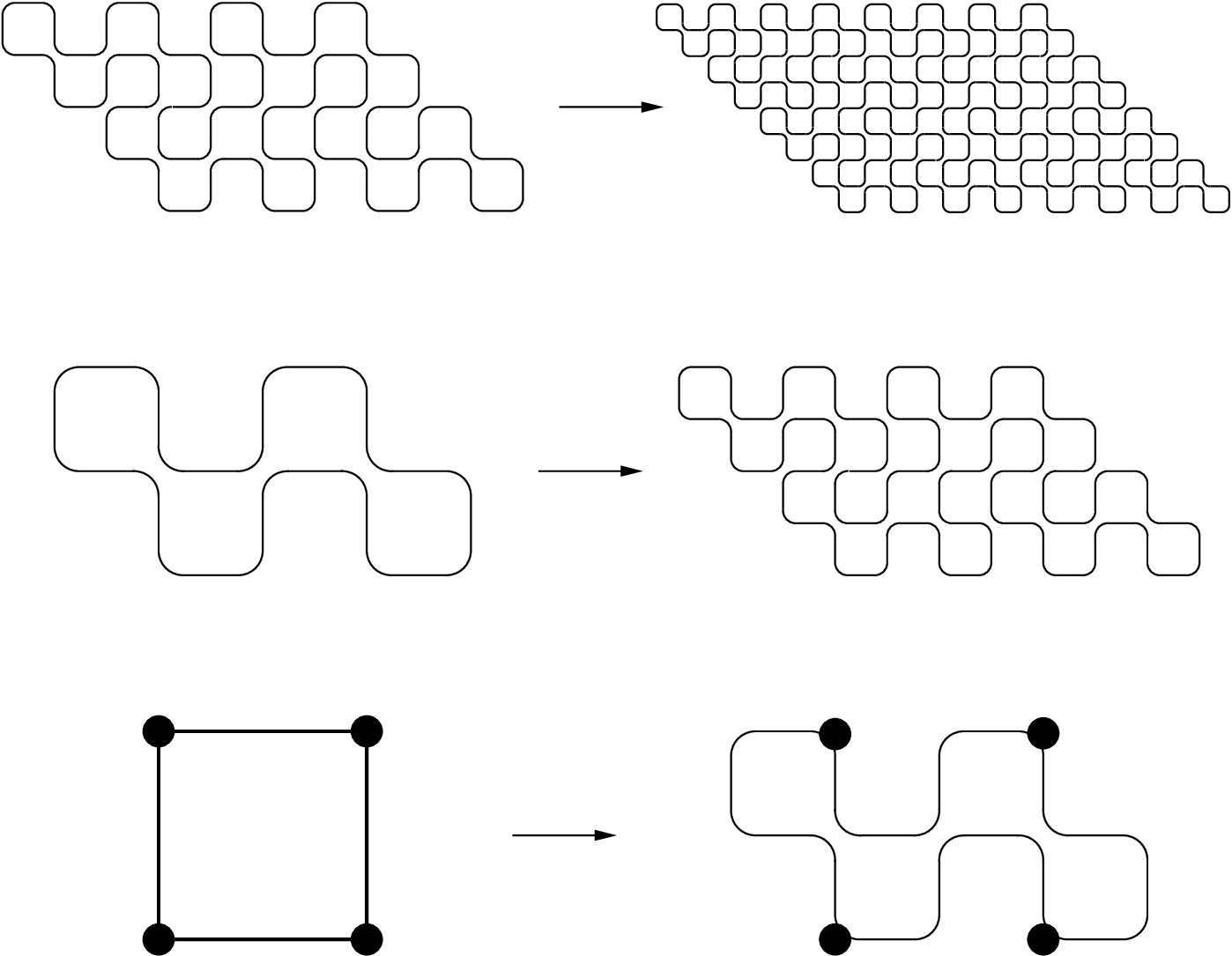}
  \begin{picture}(10,10)
    \put(-192,40){$ H^0$}
    \put(-186,140){$H^1$}
    \put(-182,240){$H^2$}
    \put(-320,10){$\gamma^0$}
    \put(-135,10){$\gamma^1$}
    \put(-320,110){$\gamma^1$}
    \put(-147,110){$\gamma^2$}
    \put(-330,210){$\gamma^2$}
    \put(-150,210){$\gamma^3$}
  \end{picture}
  \caption{Construction of $\gamma$ for the map $g$.}
  \label{fig:H0H1H2g}
\end{figure}

The \emph{$0$-th approximation} $\gamma^0$ of the Peano curve is the extended
real line 
$\RDach=\R\cup\{\infty\}\subset \CDach$. Note that $\RDach$ contains
all postcritical points of $g$. In the ``pillow'' model $\RDach$ is the
common boundary of the two squares. The picture in the orbifold
covering is shown in Figure \ref{fig:H0H1H2g} in the lower left. The
(lifts of the) postcritical points are the dots at the vertices. 

\smallskip
The upper and lower half planes (the two 
squares from which the ``pillow'' was constructed) are called the
\emph{$0$-tiles}. Their preimages  
 by $g$ (the small squares to the left in Figure \ref{fig:mapg}) are
 called the \emph{$1$-tiles}. We \emph{color} them 
white if they are preimages of the upper half plane,
otherwise black. There are four white as well as four black
$1$-tiles. The white 
$1$-tiles intersect at the \defn{critical points}, of which there are
six. At each critical point  
(\defn{$1$-vertex}) we define a \defn{connection}. This is an
assignment of which
$1$-tiles are connected and which are disconnected at this
$1$-vertex. Connections are defined in such a way that the resulting
\defn{white tile graph} is a spanning tree. This means it contains all
white $1$-tiles and no loops. In our example the white $1$-tiles are
connected at the three critical points labeled by
``$\scriptstyle{\mapsto -1}$'', 
``$\scriptstyle{\mapsto \infty}$'' in Figure \ref{fig:mapg}, and
disconnected at the others. The corresponding picture in the orbifold
covering is shown in the lower right of Figure \ref{fig:H0H1H2g}.

Following the boundary of this spanning tree gives the \defn{first
  approximation} of the Peano curve $\gamma^1$ (again indicated in
the lower right of Figure \ref{fig:H0H1H2g}). To obtain the curve
$\gamma^1$ on the pillow, one needs to ``fold the two squares that are
overlapping to the left and right on the back'' (where they intersect
in a critical point). 

We will need the following additional assumption on the
spanning tree. We have to be able to deform $\gamma^0$ to $\gamma^1$
by a \defn{pseudo-isotopy} $H^0$ that keeps the postcritical
points fixed. Recall that a pseudo-isotopy $H^0\colon S^2\times
[0,1]\to S^2$ is a homotopy that ceases to be an isotopy only at
$t=1$.   

The pseudo-isotopy is \defn{lifted} to (pseudo-isotopies)
$H^n$ by iterates $g^n$. The approximations of the Peano curve are
constructed inductively. Namely $\gamma^{n+1}$ is obtained as the
deformation of $\gamma^n$ by $H^n$. Each curve $\gamma^n$ goes through
$g^{-n} (\post)$. The limiting curve $\gamma$ is the desired Peano
curve.

\subsection{Notation}
\label{sec:notation}

The Riemann sphere is denoted by $\CDach=\C\cup\{{\infty}\}$. We denote
the $2$-sphere by $S^2$, when it is not assumed to be equipped
with a conformal structure. By $\inte U$ we denote the \defn{interior}
of a set $U$. The cardinality of a (finite) set $S$ is denoted by
$\#S$. The circle $S^1$ will often be identified with $\R/\Z$ whenever
convenient.  

For two non-negative expressions $A,B$ we write $A\lesssim
B$ if there is a constant $C>0$ such that $A\leq C B$. We refer to $C$
as $C(\lesssim)$. Similarly we write $A\asymp B$ if
$A/C \leq B\leq C A$ for a constant $C\geq 1$.

\begin{asparaitem}[$\centerdot$]
\item The \emph{$n$-iterate} of a map $f$ is denoted by $f^n$,
  $f^{-n}(A)$ denotes the preimage of a set $A$ by the iterate $f^n$. 
\item \emph{Upper indices} indicate the \defn{order} of an object,
  meaning $U^n$ is 
the preimage of some object $U^0$ by $f^n$ or $F^n$.
\item By $\crit=\crit(f)$, $\post=\post(f)$ we denote the \defn{set of
critical/postcritical points} (see next section).
\item The \defn{degree} of $F$ is denoted by $d$, the \defn{number of
    postcritical points} by $k$.
\item The \defn{local degree} of the map $F$ at $v\in S^2$ is denoted
  by $\deg_F(v)$ (see Definition \ref{def:f} (\ref{def:fbranch})). 
\item $\CC$ is a Jordan curve containing all postcritical points. 
\item \emph{Lower indices} $w,b$ denote whether objects are colored
  \emph{white} or \emph{black}.
\item $X^0_w,X^0_b$ denote the white and black \emph{$0$-tiles}
  (Section \ref{sec:thurston-maps-as}). 
\item The \emph{sets of all $n$-tiles, -edges, -vertices} are denoted by
  $\X^n,\E^n,\V^n$ (Section \ref{sec:thurston-maps-as}).
\item The \emph{expansion factor} of a fixed \emph{visual metric} for
  $F$ is denoted by $\Lambda$, see (\ref{eq:def_visuald}).
\item $\gamma^n$ is the \defn{$n$-th approximation} of the invariant
  Peano curve (Section \ref{sec:appr-gn}). 
\item $H^0$ is the \emph{pseudo-isotopy} that deforms $\CC$ to
  $\gamma^1$. $H^n$ is the \emph{lift} of $H^0$ by $F^n$, it is a
  pseudo-isotopy that deforms $\gamma^n$ to $\gamma^{n+1}$ (Definition
  \ref{def:pseudo-isotopy-h0}, Lemma
  \ref{lem:lift_degenerate_isotopies}).
\item $\alpha^n_j\subset \R/\Z$ is a point that is mapped by
  $\gamma^n$ (and subsequently by $\gamma$) to an $n$-vertex (Section
  \ref{sec:parametrizing-gn}). 
\item $\pi_w\cup\pi_b$ is a \emph{complementary non-crossing partition}. It
  describes which white/black $1$-tiles are connected at some
  $1$-vertex (Section \ref{sec:non-cross-part}).
\item A lower index ``$\epsilon$'' indicates a geometric realization
  of an object, where in a small neighborhood of each $1$-vertex we
  change tiles to ``geometrically represent the connection'' (Definition
  \ref{def:conn_geom_repres}).   
\end{asparaitem}

\section{Expanding Thurston maps as Subdivisions}
\label{sec:thurston-maps-as}

\begin{definition}\label{def:f}
  A \defn{Thurston map} is an orientation-preserving, postcritically
  finite, branched covering of the sphere, 
  \begin{equation*}
    f\colon S^2\to S^2.
  \end{equation*}
  To elaborate
  \begin{enumerate}
  \item\label{def:fbranch} $f$ is a \defn{branched cover} of the
    sphere $S^2$, meaning that locally we can write $f$ as $z\mapsto
    z^q$ after orientation-preserving homeomorphic changes of
    coordinates in domain and range. 

    More precisely for each point
    $v\in S^2$ there exists a $q\in \N$, (open) neighborhoods $V, W$
    of $v, w=f(v)$ and orientation-preserving homeomorphisms
    $\varphi\colon V\to \D$,  $\psi\colon W\to \D$ with $\varphi(v)=0$,
    $\psi(w)=0$ satisfying 
    \begin{equation*}
      \psi\circ f\circ \varphi^{-1}(z)= z^q,
    \end{equation*}
    for all $z\in \D$. 
    The integer $q=\deg_f(v)\geq 1$
    is called the \defn{local degree} of the map at $v$. A point $c$
    at which the local degree $\deg_f(c)\geq 2$ is called a
    \defn{critical point}. The set of all critical points is denoted
    by $\crit=\crit(f)$. There are only finitely many critical points since
    $S^2$ is compact. Note that no assumptions about the smoothness of
    $f$ are made.  
    %
    \setcounter{mylistnum}{\value{enumi}}
  \end{enumerate}

  \begin{enumerate}
    \setcounter{enumi}{\value{mylistnum}}
  \item\label{def:pcf} The map $f$ is \defn{postcritically finite},
    meaning that the set of \defn{postcritical points} 
    \begin{equation*}
      \post=\post(f):=\bigcup_{n\geq 1} \{f^n(c):c\in \crit(f)\}
    \end{equation*}
    is finite. As usual $f^n$ denotes the $n$-th iterate. We are only
    interested in the case when $\#\post(f)\geq 3$. 
    %
    \setcounter{mylistnum}{\value{enumi}}
  \end{enumerate}

  Consider a Jordan curve $\CC\supset \post$. The Thurston map $f$ is
  called \defn{expanding} if 
  \begin{enumerate}
    \setcounter{enumi}{\value{mylistnum}}
  \item 
    \label{def:fexpanding}
    \begin{equation*}
      \mesh f^{-n}(\CC) \to 0 \text{ as } n\to \infty.
    \end{equation*}
  \end{enumerate}
  Here $\mesh f^{-n}(\CC)$ is the maximal diameter of a
  component of $S^2\setminus f^{-n}(\CC)$. In
  \cite[Lemma~6.1]{expThurMarkov} it   
  was shown that this definition is independent of the chosen curve
  $\CC$. This notion of ``expansion'' agrees with the one by
  Ha\"{i}ssinsky-Pilgrim in \cite{HaiPil} (see
  \cite[Proposition~6.2]{expThurMarkov}).   
\end{definition}



Fix a Jordan curve $\CC\supset \post$. Here and in the following, we
always assume that such a curve $\CC$ is \defn{oriented}. 
Let $U_w,U_b$ be the two components of $S^2\setminus \CC$, where $\CC$
is positively oriented as boundary of $U_w$.  
The closures of $U_w,U_b$ are denoted by
$X^0_w,X^0_b$. 
We \emph{color} $X^0_w$ \emph{white}, $X^0_b$ \emph{black}. We refer to
$X^0_w$ $(X^0_b)$ as the white (black) $0$\defn{-tile}. 

The closure
of one component of $f^{-n}(U_{w})$ or of $f^{-n}(U_b)$ is called an
$n$\defn{-tile}. In \cite[Proposition~5.17]{expThurMarkov} it was
shown that for such an $n$-tile $X$ the map
\begin{equation}
  \label{eq:fnXntoXhomeo}
  f^n\colon X\to X^0_{w,b}\quad \text{is a homeomorphism.}
\end{equation}
This means in particular that each $n$-tile is a closed Jordan
domain. The set of all $n$-tiles is denoted by $\X^n$. The definition
of ``expansion'' implies that $n$-tiles become arbitrarily small, this
is the (only) reason we require expansion.

\smallskip
In \cite[Theorem~14.2]{expThurMarkov} (see also\cite{CFPsubdiv_rat}) it
was shown that if $f$ is expanding, then for every sufficiently high
iterate $F=f^n$ we can choose $\CC$ to be 
\defn{invariant} with respect to $F$. This means that
$F(\CC)\subset \CC$ ($\Leftrightarrow \CC\subset F^{-1}(\CC)$). It
implies that each $n$-tile is contained in exactly one
$(n-1)$-tile. Furthermore, $F$ may be represented as a
\defn{subdivision} (see \cite[Chapter~12]{expThurMarkov} as well as
the ongoing work of Cannon, Floyd, and Parry 
\cite{CFPfinSub}, \cite{CFPexpcomplex}). We will require $\CC$ to be
$F$-invariant only in 
Section~\ref{sec:construction-h0}. This is clearly a convenience in
the proof, the author however feels that this assumption is not
strictly necessary.


The set of all $n$-vertices is defined as 
\begin{equation}
  \label{eq:defVn}
  \V^n=f^{-n}(\post).
\end{equation}
Note that $\post=\V^0\subset \V^1 \subset \dots$ . Each point $v\in
\V^n$ is called an \emph{$n$-vertex}.

\smallskip
The postcritical points (or $0$-vertices) divide the curve $\CC$ into
$k=\#\post(f)$ closed Jordan arcs called \defn{$0$-edges}. 
The closure of one component of $f^{-n}(\CC)\setminus \V^n$ is called a
\defn{$n$-edge}. For each $n$-edge $E^n$ there is a $0$-edge $E^0$
such that $f^n(E^n)=E^0$. Furthermore the map $f^n\colon E^n\to E^0$
is a homeomorphism (\cite[Proposition~5.17]{expThurMarkov}). 
The set of all
$n$-edges is denoted by $\E^n$, so that $f^{-n}(\CC)=\bigcup
\E^n$. There are $\#\E^n= k(\deg(f))^n$ $n$-edges.

Each $n$-edge will have an \defn{orientation}, meaning it has an 
\emph{initial} and a \emph{terminal} point. A
$0$-edge is 
\defn{positively oriented} if its orientation agrees with the one of
the Jordan curve $\CC$. Similarly, an $n$-edge $E^n$ is called positively
oriented if $f^n$ maps the initial/terminal point of $E^n$ to the
initial/terminal 
point of (the $0$-edge) $f^n (E^n)$. 

Each $n$-tile contains exactly $k=\#\post$ $n$-edges and $k$
$n$-vertices in its boundary.




\smallskip
The $n$-tiles, $n$-edges, $n$-vertices form a \defn{cell complex} when
viewed as $2$-, $1$-, and $0$-cells (see
\cite[Chapter~5]{expThurMarkov}). 

The $n$-edges and $n$-vertices form a \emph{graph} in the natural
way. Note that this graph may have multiple edges, but no loops.

\smallskip
We \defn{color} the $n$-tiles \defn{white} if they are preimages of
$X^0_w$, \defn{black} if 
they are preimages of $X^0_b$. Each $n$-edge is shared by two $n$-tiles
of different color. Thus $n$-tiles are colored in a ``checkerboard
fashion''. An oriented $n$-edge is positively oriented if and only if
it is positively oriented as boundary of the white $n$-tile it is
contained in (and negatively oriented as boundary of the black
$n$-tile it is contained in). 
The set of white $n$-tiles is denoted by $\X^n_w$, the set
of black $n$-tiles by $\X^n_b$.

\begin{lemma}
  \label{lem:whiteXconn}
  The $n$-tiles of each color are connected, meaning 
  \begin{equation*}
    \bigcup\X^n_w,\; \bigcup\X^n_b \quad \text{are connected sets}.
  \end{equation*}
\end{lemma}

\begin{proof}
  Note that $\bigcup \X^n_w$ (or $\bigcup\X^n_b$) is connected if and
  only if $\bigcup \E^n$ is connected. 

  If $\bigcup\E^n$ is not connected, one component of $S^2\setminus
  \bigcup\E^n$ is not simply connected. This contradicts the fact that each
  such component is the interior of an $n$-tile, thus simply connected.
\end{proof}

In \cite[Chapter~8]{expThurMarkov} \emph{visual metrics} for an
expanding Thurston map $f$ were considered. If $n$-tiles have been
defined (in terms of a Jordan curve $\CC\supset \post$), we define
$m=m_{f,\CC}$ by
\begin{equation*}
  m(x,y) := \max\{n\in \N | \text{ there exist non-disjoint
    $n$-tiles } X\ni x, Y\ni y\},
\end{equation*}
for all $x,y\in S^2$, $x\neq y$. We set $m(x,x)= \infty$. 
A metric $\varrho$ on $S^2$ is called a \emph{visual metric} for $f$ if
there is a constant $\lambda>1$ (called the \emph{expansion factor} of
$\varrho$), such that
\begin{equation}
  \label{eq:def_visuald}
  \varrho(x,y) \asymp \lambda^{-m(x,y)},
\end{equation}
for all $x,y\in S^2$ and a constant $C=C(\asymp)$ independent of
$x,y$. Here it is understood that $\lambda^{-\infty}=0$. 

Visual metrics always exist, see \cite[Theorem~15.1]{expThurMarkov}, as
well as \cite{HaiPil}. In fact $\varrho$ can be chosen such that $f$ is an 
\defn{expanding local similarity} with respect to $\varrho$. More precisely,
for each $x\in S^2$ there exists a neighborhood $U_x\ni x$,
such that
\begin{equation}
  \label{eq:expmetric}
  \frac{\varrho(f(x),f(y))}{\varrho(x,y)} = {\lambda}, 
\end{equation}
for all $y\in U_x\setminus\{x\}$. We do however not need this stronger
form. 

\bigskip
We fix a curve $\CC\supset \post(f)$ as well as an iterate $F=f^n$ for
now, assuming they have certain properties (more precisely, there is a
pseudo-isotopy $H^0$ as in the next section). In Section
\ref{sec:construction-h0} they will be chosen properly. Note that the
postcritical set of $F$ equals the postcritical set of $f$, which is
thus just denoted by ``$\post$''. Throughout the construction we
denote by
\begin{equation*}
  \boxed{\phantom{x}d := \deg F = (\deg f)^n, \quad k:=\#\post.\phantom{x}   }
\end{equation*}

From now on $m$-tiles, $m$-edges,
$m$-vertices are understood to be with respect to $(F,\CC)$, meaning they are 
$mn$-tiles, $mn$-edges, $mn$-vertices with respect to $(f,\CC)$.  


\smallskip
Clearly expansion of $f$ implies expansion of $F$. A visual metric for
$f$ with expansion factor $\lambda$ is a visual metric for $F$ with
expansion factor $\Lambda=\lambda^n$. Expression
(\ref{eq:expmetric}) continues to hold, where we have to
replace $\lambda$ by $\Lambda:=\lambda^n>1$. 

\begin{lemma}
  \label{lem:Vepspre}
  Let $\varrho$ be a visual metric for $F$ with expansion factor
  $\Lambda$. Then there are an $\epsilon_0>0$ and a constant $K\geq 1$
  such that the following holds. For any 
  $\epsilon\in (0,\epsilon_0)$ let $\mathcal{N}(\V^1,\epsilon)$ be the
  $\epsilon$-neighborhood of $\V^1$ (defined in terms of $\varrho$).
  Then there is a neighborhood $V^1_\epsilon$ of $\V^1$ such that
  \begin{align*}
    &\mathcal{N}(\V^1,\epsilon/K) \subset V^1_\epsilon \subset
    \mathcal{N}(\V^1,\epsilon)
    \intertext{and for all $n\in \N$ the set $V = V^{n+1}_{\Lambda{^{-n}}\epsilon}:=
      F^{-n}(V^1_{\epsilon})$ satisfies}
    &\mathcal{N}(\V^{n+1},\Lambda^{-n}\epsilon/K) \subset V \subset
    \mathcal{N}(\V^{n+1},\Lambda^{-n}\epsilon).    
  \end{align*}
\end{lemma}

The proof of this lemma follows immediately from
\cite[Lemmas~8.9 and 8.10]{expThurMarkov}.

\section{The approximations  $\gamma^n$}
\label{sec:appr-gn}

We begin the proof of Theorem~\ref{thm:main}. We assume (until the end
of Section~\ref{sec:construction-h0}) that $F$ ($=f^n$, the index
``$n$''  
however will be ``recycled'')  is an expanding Thurston map, and
$\CC\supset \post$ is a fixed Jordan curve. The $n$-tiles and
$n$-edges are defined in terms of $(F,\CC)$; see the 
previous section. Furthermore we fix a visual metric $\varrho$
for $F$ with expansion factor $\Lambda>1$; see
(\ref{eq:def_visuald}). Metrical properties and objects, such as the 
diameter and neighborhoods, will always be defined in terms of this
metric.  

\smallskip
The desired invariant Peano curve $\gamma$ will be constructed as the
limit of 
\defn{approximations} $\gamma^n$. 
Here $\gamma^0$ is the Jordan curve $\CC\supset\post$. The first
approximation $\gamma^1$ will be constructed in Section
\ref{sec:construction-h0}, 
more precisely a \defn{pseudo-isotopy}
$H^0$ (rel.\ $\post$) that deforms $\gamma^0$ to $\gamma^1$ will be
constructed.  

In this section the approximations $\gamma^n$ of the invariant Peano
curve will be constructed by repeated \defn{lifts} of $H^0$. These curves
are however not 
yet parametrized, they are \defn{Eulerian circuits}.  

\subsection{Pseudo-isotopies }
\label{sec:degen-isot}

\begin{definition}[Pseudo-isotopies]
  \label{def:degen-isot}
  A homotopy 
  \begin{equation*}
    H\colon S^2\times [0,1]\to S^2
  \end{equation*}
  is called a \defn{pseudo-isotopy} if it is an isotopy on $S^2\times
  [0,1)$. We always require that $H(x,0)= x$ on $S^2$. 
  If $H(\cdot,t)$ is 
  constant on a set $A\subset S^2$ it is an
  \defn{pseudo-isotopy rel.\ 
    $A$}; alternatively we then say that $H$ is \defn{supported} on
    $S^2\setminus A$. We interchangeably write $H_t(x)=H(x,t)$ to
    unclutter notation.    
\end{definition}

\begin{remark}
  Given a pseudo-isotopy $H_t$ as above it follows that $H_1$ is
  \emph{surjective} ($S^2\setminus \{\text{point}\}$ has different
  homotopy type than $S^2$) and \emph{closed} (since we are dealing with
  compact Hausdorff spaces). A pseudo-isotopy on a general space $S$
  is required to end in a surjective, closed map.  
\end{remark}

Our starting point is a
pseudo-isotopy $H^0=H^0(x,t)$ as follows. This is the central object
of the whole construction. In this and 
the following section we show that such a $H^0$ is sufficient to
construct the invariant Peano curve as desired. The construction of
$H^0$ itself will be done in Section~\ref{sec:construction-h0}. In
Lemma~\ref{lem:H0epsH0} an equivalent condition for the existence of
$H^0$ will be given.  
\begin{definition}[Pseudo-isotopy $H^0$]
  \label{def:pseudo-isotopy-h0}
  We consider a pseudo-isotopy $H^0$ with the following
  properties.
  \begin{enumerate}[($H^0$ 1)]
  \item 
    \label{item:H0_1}
    $H^0$ is a pseudo-isotopy rel.\ $\V^0=\post$ (the set of all
    postcritical points). 
  \item
    \label{item:H0_2}
    The set of all $0$-edges $\bigcup \E^0=\CC$ is deformed by $H^0$
    to $\bigcup\E^{1}$, 
    \begin{equation*}
      H^0_1\left(\bigcup\E^0\right)=\bigcup \E^{1}.
    \end{equation*}
    %
    \setcounter{mylistnum}{\value{enumi}}
  \end{enumerate}

  To simplify the discussion we require that $H^0$ deforms the
  $0$-edges to $1$-edges as ``nicely as possible'' (see Lemma
  \ref{lem:H0maps_edges} below). The
  construction would still work however, without imposing the
  following two properties.

  \begin{enumerate}[($H^0$ 1)]
    \setcounter{enumi}{\value{mylistnum}}
  \item 
    \label{item:H0_3}
    Let $\epsilon_0>0$ be the constant from Lemma~\ref{lem:Vepspre}, 
    $0<\epsilon<\min \{ \epsilon_0, 1/2\}$, and
    $V^1_{\epsilon}$ be a neighborhood of
    $\V^{1}$ as in Lemma~\ref{lem:Vepspre}, 
    we require that 
    \begin{align*}
      & H^0\colon S^2 \times [1-\epsilon,1]\to S^2 \text{ is supported
        on } V^1_{\epsilon}.            
    \end{align*}
    So $H^0$ ``freezes'' on $S^2\setminus V^1_{\epsilon}$.
  \item
    \label{item:H0_4}
    Consider a $1$-vertex $v$. Only finitely many points of
    $\CC=\bigcup \E^0$ are deformed by $H^0$ to $v$. In other words, we
    require that
    \begin{equation*}
      \left\{x\in \bigcup \E^0 \bigm\vert
        H^0_1(x)=v\right\} \text{ is a finite set.} 
    \end{equation*}    
    %
    \setcounter{mylistnum}{\value{enumi}}
  \end{enumerate}
  One final assumption will be made on $H^0$. However the precise
  meaning will only be explained in Section
  \ref{sec:gamman+1-d-fold}. 
  \begin{enumerate}[($H^0$ 1)] 
    \setcounter{enumi}{\value{mylistnum}}
  \item
    \label{item:H0_5}
    View $\gamma^0=\CC$ as a circuit of $0$-edges. Let $\gamma^1$
    be the Eulerian circuit obtained from $H^0$, 
    see Definition \ref{def:gamma_n} (\ref{item:gamma_n_4}). Then
    \begin{equation*}
      F\colon \gamma^1 \to \gamma^0,     
    \end{equation*}
    is a $d$-fold cover, see Definition \ref{def:dfoldcover}. 
  \end{enumerate}
\end{definition}

Consider  $\{x_j\}:= (H^0_1)^{-1}(\V^{1}) \cap\CC$, the
set of points on $\CC= \bigcup \E^0$ 
that are mapped by $H^0_1$ to some $1$-vertex (each $x_j$ possibly
to a different one). 
Note that $\{x_j\}$  is finite by ($H^0$ \ref{item:H0_4}) and
$\{x_j\}\supset \post=\V^0$ by ($H^0$ \ref{item:H0_1}). 
Thus the points
$\{x_j\}$ divide $\CC$ (and each $0$-edge) into closed arcs
$A_j$. Recall that $d=\deg F, k=\#\post$.

\begin{lemma}
  \label{lem:H0maps_edges}
  There are $kd$ arcs $A_j$ as above.
  Furthermore 
  \begin{align*}
    & E^1_j:=H^0_1(A_j) \text{ is a $1$-edge and}
    \\
    & H^0_1\colon A_j \to E^1_j \text{ is a homeomorphism,}
    \intertext{for each $j$. On the other hand}
    &\text{ each $1$-edge $E^1$ is the image of one such $A_j$ by
      $H^0_1$.} 
  \end{align*}
\end{lemma}

\begin{proof}
  Consider one arc $A_j$ as in the statement with endpoints
  $x_j,x_{j+1}$.  
  Note that $\bigcup \E^{1} \setminus \V^{1}$ is disconnected,
  each component is the interior of a $1$-edge. Thus
  \begin{equation*}
    H^0_1(\inte A_j)\subset \inte E^1_j,
  \end{equation*}
  for some $1$-edge $E^1_j$. Assume $H^0_1\colon A_j\to E^1_j$ is not
  a homeomorphism. 

  Assume first that $H^0_1(A_j)\ne E^1_j$. Then
  $H^0_1(x_j)=H^0_1(x_{j+1})$ and there are distinct points $x,y\in \inte A_j$
  mapped to the same point $z$ by $H^0_1$. But $z\in S^2\setminus
  V^1_\epsilon$ for sufficiently small $\epsilon$. Then
  \begin{equation*}
    H^0_{1-\epsilon}(x)=H^0_1(x)=H^0_1(y)=H^0_{1-\epsilon}(y), 
  \end{equation*}
  which is a contradiction ($H^0_{1-\epsilon}$ is a
  homeomorphism). Thus $H^0_1(A_j)=E^1_j$. Exactly the same argument
  shows that $H^0_1\colon A_j\to E^1_j$ is bijective, hence a
  homeomorphism.  

  \smallskip
  Using the previous argument again shows that distinct arcs $A_i,A_j$
  map to distinct $1$-edges $E^1_i,E^1_j$.  

  \smallskip
  Finally, since $H^0_1(\bigcup \E^0)=\bigcup \E^1$ (by ($H^0_1$
  \ref{item:H0_2})) each
  $1$-edge $E^1$ is the image of one such arc $A_j$ by $H^0_1$.

  Thus there is exactly one $A_j$ for each $1$-edge, meaning there are
  $kd$ such arcs. 
\end{proof}

\subsection{Lifts of pseudo-isotopies}
\label{sec:lifts-degen-isot}

\begin{lemma}[Lift of pseudo-isotopy]
  \label{lem:lift_degenerate_isotopies}
  Let $H\colon S^2\times [0,1]\to S^2$ be a pseudo-isotopy rel.\
  $\post=\V^0$. Then $H$ can be \defn{lifted uniquely} by $F$ to a
  pseudo-isotopy $\widetilde{H}$ rel.\ $\V^1$. This means that
  $F(\widetilde{H}(x,t))= H(F(x),t)$ for all $x\in S^2, t\in [0,1]$,
  i.e., the following diagram commutes.

  \begin{equation*}
    \xymatrix{
      S^2 \ar[r]^{\widetilde{H}} \ar[d]_F & S^2 \ar[d]^F
      \\
      S^2 \ar[r]_{H} & S^2
    }
  \end{equation*}
    
  Furthermore 
  \begin{enumerate}
  \item
    \label{item:lift1}
    if $H$ is a pseudo-isotopy rel.\ a set $S\subset S^2$,
    then the lift $\widetilde{H}$ is a pseudo-isotopy rel.\
    $F^{-1} (S)$.
  \item
    \label{item:lift2}
    Let $H^n$ be the lift of $H$ by an iterate $F^n$. Then
    \begin{equation*}
      \diam H^n:= \max_{x\in S^2} \diam \{H^n(x,t) \mid t\in [0,1]\}  
      \lesssim
      \Lambda^{-n}.
    \end{equation*}
    Here the diameter is measured with respect to the fixed visual metric
    with expansion factor $\Lambda>1$. The constant
    $C(\lesssim)$ is independent of $n$. 
  \end{enumerate}
\end{lemma}

The proof follows from the standard lifting of paths, see
\cite[Proposition 10.1]{expThurMarkov}. For
property~(\ref{item:lift2}) see \cite[Lemma 10.3]{expThurMarkov}.

\smallskip
We now lift the pseudo-isotopy from the last subsection. Lifts
retain the properties of $H^0$.   
\begin{lemma}[Properties of $H^n$]
  \label{lem:Hn}
  Let $H^0$ be a pseudo-isotopy as in the 
  last subsection. Let $H^n$ be the lift of $H^0$ by $F^n$
  (equivalently the lift of $H^{n-1}$ by $F$). The lifts satisfy the
  following.

  \begin{enumerate}[\upshape($H^n$ 1)]
  \item 
    \label{item:Hn_1}
    $H^n$ is a pseudo-isotopy rel.\ $\V^n$ (the set of all
    $n$-vertices). 
  \item
    \label{item:Hn_2}
    The set of all $n$-edges $\bigcup \E^n$ is deformed by $H^n$
    to $\bigcup\E^{n+1}$, 
    \begin{equation*}
      H^n_1\left(\bigcup\E^n\right)=\bigcup \E^{n+1}.
    \end{equation*}
  \item 
    \label{item:Hn_3}
    Let $V^1_\epsilon$ be the neighborhood of $\V^1$ as in {\upshape($H^0$
    \ref{item:H0_3})}, see also Lemma~\ref{lem:Vepspre}. The set 
    $V=V^{n+1}_{\Lambda^{-n}\epsilon}:= F^{-n}(V^1_\epsilon)$, which
    is a neighborhood of $\V^{n+1}$, is such that
    \begin{align*}
      & H^n\colon S^2 \times [1-\epsilon,1]\to S^2 \text{ is supported
        on }V.            
    \end{align*}
    So $H^n$ ``freezes'' on $S^2\setminus V$.
  \item
    \label{item:Hn_4}
    Consider an $(n+1)$-vertex $v$. Only finitely many points of
    $\bigcup \E^n$ are deformed by $H^n$ to $v$. In other words,
    \begin{equation*}
      \left\{x\in \bigcup \E^n \bigm\vert H^n_1(x)=v\right\} \text{ is a finite set.} 
    \end{equation*}    
    %
    \setcounter{mylistnum2}{\value{enumi}}
  \end{enumerate}
  We list the final property here. Again it will be explained and
  proved only in Section~\ref{sec:gamman+1-d-fold}.
  \begin{enumerate}[\upshape($H^n$ 1)]
    \setcounter{enumi}{\value{mylistnum2}}
  \item
    \label{item:Hn_5}
    Let $\gamma^n,\gamma^{n+1}$ be the Eulerian circuits from
    Definition \ref{def:gamma_n} (\ref{item:gamma_n_4}). Then 
    \begin{equation*}
      F\colon \gamma^{n+1}\to \gamma^n
    \end{equation*}
    is a $d$-fold cover in the sense of Definition
    \ref{def:dfoldcover}. 
  \end{enumerate}

\end{lemma}

\begin{proof}
  ($H^n$ \ref{item:Hn_1}) is clear from Lemma 
  \ref{lem:lift_degenerate_isotopies} (\ref{item:lift1}). 

  \medskip
  ($H^n$ \ref{item:Hn_3}) follows directly from Lemma
  \ref{lem:Vepspre} and Lemma \ref{lem:lift_degenerate_isotopies}
  (\ref{item:lift1}). 

  \medskip
  ($H^n$ \ref{item:Hn_2}) Since $H^n$ is the lift of $H^0$ by $F^n$ we
  have
  \begin{align*}
    &F^n\left(H_1^n\left(\bigcup \E^n\right)\right) 
    = H_1^0\left(F^n\left(\bigcup \E^n\right)\right) 
    = H_1^0\left(\bigcup\E^0\right) 
    = \bigcup \E^1. 
    \intertext{Thus}
    &H_1^n\left(\bigcup \E^n\right) \subset \bigcup \E^{n+1}.
  \end{align*}
  To prove equality in the last expression consider $\inte E^1$, the
  interior of a $1$-edge. Let $U^0=\inte A^0=(H_1^0)^{-1}(\inte E^1)\cap \bigcup
  \E^0$ be the set in $\bigcup \E^0$ that is deformed by $H^0_1$ to
  $\inte E^1$. 
  This is an arc that does not contain a 
  postcritical point (see Lemma~\ref{lem:H0maps_edges}).
  
  Consider $U_1^n, \dots U^n_{d^n}\subset \bigcup \E^n$,
  the preimages of $U^0$ by $F^n$; they are disjoint arcs. Each
  $U_j^n$ is deformed by $H_1^n$ 
  to (the interior of) a $(n+1)$-edge (since
  $F^n(H^n_1(U_j^n))=H^0_1(F^n(U^n_j))=H^0_1(U^0)=\inte E^1$).   

  We remind the reader of the following elementary fact about
  lifts. Let $\sigma\colon 
  [0,1]\to S^2\setminus\post(F)$ be a path and
  $\widetilde{\sigma}_1,\widetilde{\sigma}_2$ two lifts by $F^n$
  with distinct initial points. Then the endpoints of
  $\widetilde{\sigma}_1,\widetilde{\sigma}_2$ are distinct. Indeed
  otherwise the lift of the reversed path $\sigma(1-t)$ would fail to
  be unique. 

  Therefore the $U_j^n$ are deformed by $H^n$ to (the interior of)
  $d^n$ \emph{distinct} $(n+1)$-edges. It follows that $\bigcup \E^n$
  is deformed by $H^n$ to $kd^{n+1}$ $(n+1)$-edges, meaning all of them.

  \medskip
  ($H^n$ \ref{item:Hn_4}) 
  Assume distinct points $\{x^n_j\}_{j\in \N}\subset \bigcup \E^n$ are
  deformed to some $(n+1)$-vertex $v^{n+1}$ by $H^n_1$. Then the
  (infinitely many different)
  points $x^0_j:=F^n(x^n_j)\in \bigcup \E^0$ are deformed by $H^0_1$
  to the $1$-vertex 
  $v^1:=F^n(v^{n+1})$, contradicting Property ($H^0$
  \ref{item:H0_4}). 
\end{proof}

From now on we assume that the pseudo-isotopies $H^n$ are given as
above. 

Consider  $\{x_j\}:= (H^n_1)^{-1}(\V^{n+1}) \cap \bigcup \E^n$, the
set of points on $\bigcup \E^n$ 
that are mapped by $H^n_1$ to some $(n+1)$-vertex (each $x_j$ possibly
to a different one). 
Note that $\{x_j\}$  is finite by ($H^n$ \ref{item:Hn_4}) and
$\{x_j\}\supset \V^n$ by ($H^n$ \ref{item:Hn_1}). 
Thus the points
$\{x_j\}$ divide $\bigcup \E^n$ (and each $n$-edge) into closed arcs
$A_j$.   

\begin{lemma}
  \label{lem:Hn_maps_edges}
  There are $kd^{n+1}$ such arcs $A_j$ as above. 
  Furthermore 
  \begin{align*}
    & E'_j:=H^n_1(A_j) \text{ is an } (n+1)\text{-edge and}
    \\
    & H^n_1\colon A_j \to E'_j \text{ is a homeomorphism,}
     \intertext{for each $j$. On the other hand}
     & \text{each } (n+1)\text{-edge } E' \text{ is the image of
       one such } A_j \text{ by } H^n_1.  
  \end{align*}
\end{lemma}

\begin{proof}
  This follows exactly as in Lemma \ref{lem:H0maps_edges}.
\end{proof}

\subsection{Eulerian circuits $\gamma^n$}
\label{sec:euler-circ-gamm}

We construct $\gamma^n$, the
$n$-th approximation of the invariant Peano curve, from the
pseudo-isotopies $H^n$.  
The curves $\gamma^n$ however do not yet have the
``right'' parametrization. Thus $\gamma^n$ will for now be an
\defn{Eulerian circuit} in $\bigcup \E^n$. However the
\defn{parametrization} of this Eulerian circuit will later still be
denoted by $\gamma^n(t)$. 

\begin{definition}
  \label{def:euler_circuit}
  An \defn{Eulerian circuit} is a closed edge path that traverses each
  edge exactly once.

  Consider now the graph of $n$-edges $\bigcup \E^n$, containing $k
  d^n$ $n$-edges. In this graph an
  Eulerian circuit is a finite sequence of oriented $n$-edges
  \begin{equation*}
    \gamma^n=E_0,\dots, E_{kd^n-1},
  \end{equation*}
  such that the following holds (indices are taken $\bmod
  \,kd^n$). Each $n$-edge appears exactly once, 
  and the terminal point of $E_j$ is the initial point of
  $E_{j+1}$. In particular, the terminal point of $E_{kd^n-1}$ is the
  initial point of $E_0$. 
  If $v$ is the terminal point of $E_{j}$/the
  initial point of $E_{j+1}$, we say that $E_{j+1}$ \defn{succeeds}
  $E_j$ in $\gamma^n$ at $v$. 
 
  Cyclical permutations of indices are not considered 
  to change $\gamma^n$, but orientation reversing does.  
\end{definition}

The approximations $\gamma^n$ of the invariant Peano curve are defined  
as follows. 
\begin{definition}[Eulerian circuits $\gamma^n$]
  \label{def:gamma_n}
  Recall that the Jordan curve $\CC=\bigcup \E^0$ is positively
  oriented as boundary of the white $0$-tile $X^0_w$. Let 
  \begin{equation*}
    \gamma^0=S^1 \to \CC
  \end{equation*}
  be an orientation-preserving homeomorphism. We define
  inductively 
  \begin{align*}
    \gamma^{n+1}\colon S^1 \to \bigcup \E^{n+1} \text{ by }
    \\
    \gamma^{n+1}(t):= H^n_1(\gamma^n(t)),
  \end{align*}
  for all $n\geq 0$. 
  Let us note the following properties.
  \begin{enumerate}[\upshape (i)]
  \item    
    \label{item:gamma_n_1}
    The map is surjective by ($H^n$ \ref{item:Hn_2}). 
 \item
   \label{item:gamma_n_2}
   The set $\W^n:=(\gamma^n)^{-1}(\V^n)\subset S^1$ is finite by ($H^n$
   \ref{item:Hn_4}). 
  \item
    \label{item:gamma_n_3}
    For each $n$-edge $E$ there is exactly one closed arc
    $[w_j,w_{j+1}]\subset \R/\Z=S^1$,
    formed by consecutive points $w_j,w_{j+1}\in \W^n$, such that
    \begin{equation*}
      \gamma^n\colon [w_j,w_{j+1}] \to E \text{ is a
        homeomorphism}. 
    \end{equation*}
    This follows directly from Lemma \ref{lem:Hn_maps_edges}.
  \item
    \label{item:gamma_n_4}
    The map $\gamma^n$ induces an Eulerian circuit (still denoted by
    $\gamma^n$) on $\bigcup
    \E^n$ in the obvious way, namely the $n$-edges are given the
    orientation and ordering induced by $\gamma^n$. 
  \end{enumerate}
\end{definition}

We record how the Eulerian circuit $\gamma^n$ is related to the Eulerian
circuit $\gamma^{n+1}$. Consider an $n$-edge $E$, which is subdivided
into arcs $A_0,\dots,A_m$ as in Lemma \ref{lem:Hn_maps_edges}. An
orientation of $E$ induces an orientation of the arcs $A_j$. As
before we say that $A_j$ succeeds $A_i$ in $E$ if the terminal point
of $A_i$ is the initial point of $A_j$.      

\begin{lemma}
  \label{lem:gamma_n_rel}
  Let $D',E'$ be two $(n+1)$-edges. Let $A',B'\subset \bigcup \E^n$ be
  the two arcs that are mapped (homeomorphically) to $D',E'$ by
  $H^n_1$. 
  Then $E'$ succeeds $D'$ in $\gamma^{n+1}$ if and only if
  \begin{align*}
    &A',B' \text{ are contained in the same } n\text{-edge } E, 
    \\
    &\text{ and } B'
    \text{ succeeds } A' \text{ in } E \text{ (oriented by } \gamma^n).
    \intertext{ or}
    &A',B' \text{ are contained in different } n\text{-edges }
    E(A'),E(B') \text{ and}
    \\
    &\text{ the terminal point of } A' \text{ is the terminal point of }
    E(A'),
    \\ 
    &\text{ the initial point of } B' \text{ is the initial point of } E(B'),
    \\
    &\text{ and }
    E(B') \text{ succeeds } E(A') \text{ (in } \gamma^n).
  \end{align*}
\end{lemma}

\begin{proof}
  This is again obvious from the construction.
\end{proof}

\subsection{$\gamma^{n+1}$ is a $d$-fold cover of $\gamma^n$}
\label{sec:gamman+1-d-fold}

We are now ready to give the definition of properties ($H^0$
\ref{item:H0_5}) and ($H^n$ \ref{item:Hn_5}).
\begin{definition}[Cover of Eulerian circuits]
  \label{def:dfoldcover}
  Let $\gamma^{n+1},\gamma^n$ be the Eulerian circuits constructed in 
  Definition \ref{def:gamma_n} (\ref{item:gamma_n_4}).
  We call
  \begin{equation*}
    F\colon \gamma^{n+1}\to \gamma^n \text{ a $d$\defn{-fold cover}},       
  \end{equation*}
  if $F$ maps succeeding
  $(n+1)$-edges 
  (in $\gamma^{n+1}$) to succeeding $n$-edges (in $\gamma^n$). 
  An equivalent definition is as follows.
  Let 
  \begin{align*}
    &\gamma^n=E_0, \dots, E_{d^n-1},
    \\
    &\gamma^{n+1}=E'_0,\dots, E'_{d^{n+1}-1}
  \end{align*}
  be two Eulerian circuits. Here each $E_j$ is an (oriented) $n$-edge,
  each $E'_j$ an (oriented) $(n+1)$-edge. 
  Let $m$ be the index such that $F(E'_0)=E_m$. 
  Then $\gamma^{n+1}$ is
  a $d$-fold cover of $\gamma^n$ by $F$ if 
  \begin{equation*}
    F(E'_j)=E_{m+ j},
  \end{equation*}
for all $j=0,\dots, d^{n+1}-1$. 
\end{definition}

\begin{convention}
   Indices of $n$-edges (and
  $n$-vertices) are taken $\bmod \,kd^n$ in here and the following.
\end{convention}

Property ($H^0$ \ref{item:H0_5}) is equivalent to the following
(seemingly weaker) condition.
Recall that each $0$-edge $E_j\subset \CC$ is \defn{positively
  oriented} if its orientation agrees with the one induced by
$\CC$. Similarly each $n$-edge $E^n$ is positively oriented if
$F^n\colon E^n\to E_j$ preserves orientation.  
Recall furthermore that $n$-tiles are colored white/black if
they are preimages of the $0$-tiles $X^0_w,X^0_b$ by $F^n$. 
Each $n$-edge $E^n$ is contained in the boundary of exactly one white
and one 
black $n$-tile.
Then $E^n$ is \defn{positively oriented} if it is positively oriented as
boundary arc of the white $n$-tile in $X^n\supset E^n$.

\begin{lemma}
  \label{lem:orientation-d-fold-cover}
  Let $\gamma^1$ be a Eulerian circuit in $\bigcup \E^1$. Then the
  following conditions are equivalent:
 \begin{align*}
   &\text{{\upshape ($H^0$ \ref{item:H0_5})} }\quad 
  F\colon \gamma^1\to \gamma^0 
  \text{ is a $d$-fold cover;} 
  \\
   & \text{{\upshape ($H^0$ \ref{item:H0_5}')} \quad Each $1$-edge in $\gamma^1$ is
     positively oriented.} 
 \end{align*}
\end{lemma}

\begin{proof}
  Let
  $p_0,\dots, p_{k-1}\subset \CC$ be the postcritical points, labeled
  mathematically positively on $\CC$.
  Consider an oriented $1$-edge $E^1$ with initial point $v\in \V^1$
  and terminal point $v'\in \V^1$. It is positively
  oriented if and only if $F(v')$ succeeds $F(v)$, i.e., if
  $F(v)=p_{j}$, $F(v')=p_{j+1}$ for some $j$ (indices are taken $\bmod
  k$).    

  Let $\gamma^1$ go through $1$-vertices $v_0,\dots,v_{kd^n-1}$ in
  this order. Then $F\colon \gamma^1\to \gamma^0$ is a $d$-fold cover
  if and only if $F(v_{i+1})$ succeeds $F(v_i)$ (for all $i$, indices
  are taken $\bmod kd^n$), if and
  only if each 
  edge in $\gamma^1$ is positively oriented.
\end{proof}

\begin{remark}
  It is not very hard to show that if $\gamma^1$ is obtained as in
  Definition \ref{def:gamma_n} (without assuming ($H^0$
  \ref{item:H0_5})), then either all $1$-edges are
  positively oriented, or all $1$-edges are negatively oriented
  in $\gamma^1$ (see \cite[Lemma~6.7]{unmating}). In the latter case
  our construction would result in a 
  semi-conjugacy of $F$ to $z^{-d}$. Indeed a Peano curve
  $\gamma\colon S^1\to S^2$ that semi-conjugates $F=f^n$ to $z^{-d}$
  exists by a slight variation of the construction presented here. 
  Namely in Section
  \ref{sec:construction-h0} the role of the white and black $1$-tiles
  has to be reversed.  
\end{remark}

We now show how property ($H^0$ \ref{item:H0_5}) implies ($H^n$
\ref{item:Hn_5}), i.e., finish the proof of Lemma~\ref{lem:Hn}.

\begin{lemma}
  \label{lem:covergn}
  Let $H^0$ be a pseudo-isotopy as in Definition
  \ref{def:degen-isot}, $H^n$ the lifts of
  $H^0$ by $F^n$. The Eulerian circuits
  $\gamma^n$ are the ones from Definition \ref{def:gamma_n}. Then

  \begin{equation*}
    \text{{\upshape($H^n$ \ref{item:Hn_5})}\quad $F\colon \gamma^{n+1}\to \gamma^n$
      is a $d$-fold cover. }
  \end{equation*}
\end{lemma}

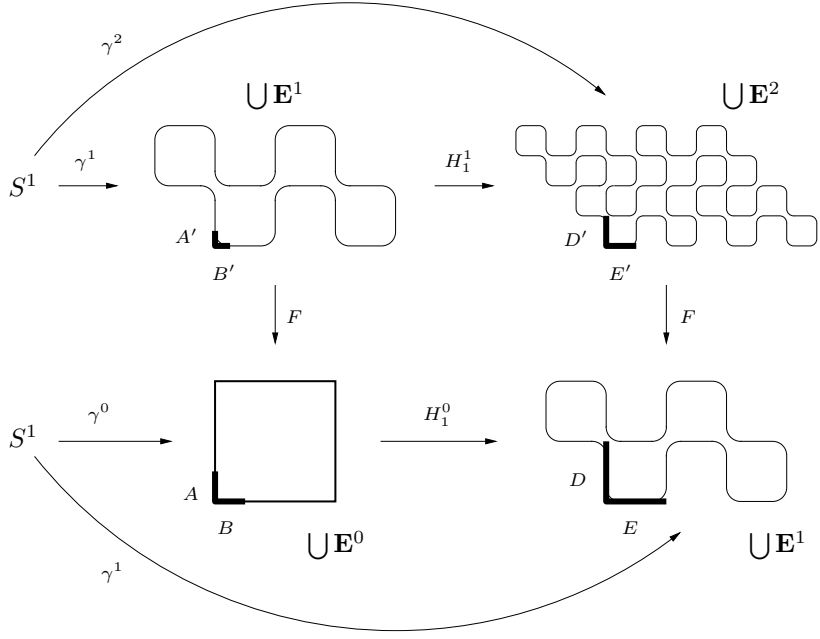
\begin{figure}
  \centering
  \includegraphics[scale=0.4]{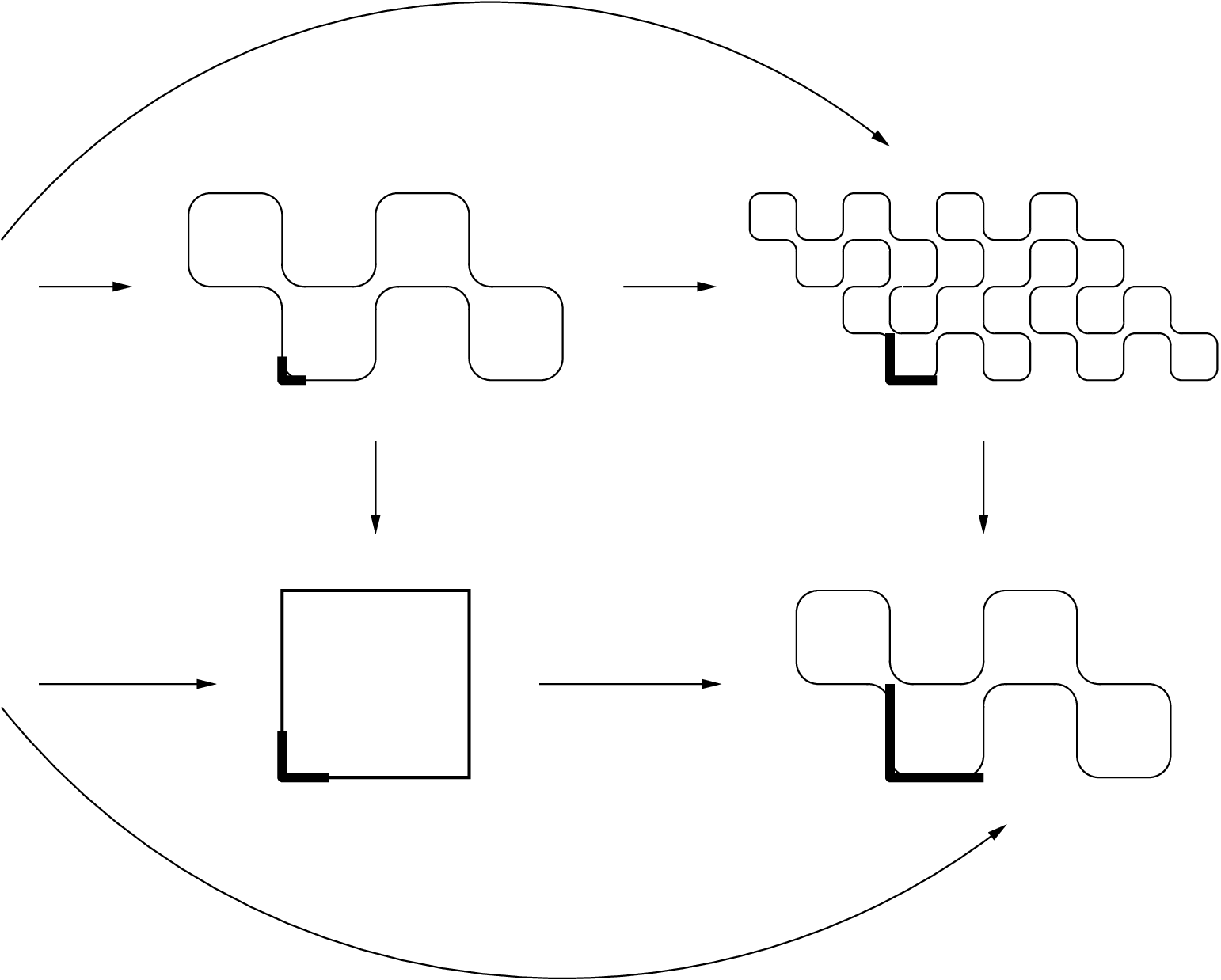}
  \begin{picture}(10,10)
    \put(-310,165){$S^1$}
    \put(-310,70){$S^1$}
    \put(-285,176){$\scriptstyle{\gamma^1}$}
    \put(-280,80){$\scriptstyle{\gamma^0}$}
    \put(-145,176){$\scriptstyle{H^1_1}$}
    \put(-153,80){$\scriptstyle{H^0_1}$}
    \put(-205,117){$\scriptstyle{F}$}
    \put(-56,117){$\scriptstyle{F}$}
    \put(-275,220){$\scriptstyle{\gamma^2}$}
    \put(-275,20){$\scriptstyle{\gamma^1}$}
    \put(-220,200){${\bigcup \E^1}$}
    \put(-40,200){${\bigcup \E^2}$}
    \put(-197,30){${\bigcup \E^0}$}
    \put(-30,30){$\bigcup \E^1$}
    \put(-247,147){$\scriptstyle{A'}$}
    \put(-233,133){$\scriptstyle{B'}$}
    \put(-100,146){$\scriptstyle{D'}$}
    \put(-83,133){$\scriptstyle{E'}$}
    \put(-244,50){$\scriptstyle{A}$}
    \put(-231,37){$\scriptstyle{B}$}
    \put(-98,55){$\scriptstyle{D}$}
    \put(-78,37){$\scriptstyle{E}$}
  \end{picture}
  \caption{Commutative diagram for Lemma \ref{lem:covergn}.}
  \label{fig:H0H1}
\end{figure}

\begin{proof}
  The reader is advised to consult Figure \ref{fig:H0H1} for
  reference. Roughly speaking by deforming $\bigcup \E^0$ via $H^0$ and
  $\bigcup \E^1$ via $H^1$, one can push the $d$-fold cover $F\colon
  \gamma^1\to \gamma^0$ to a $d$-fold cover $F\colon \gamma^2\to
  \gamma^1$. We give however a more pedestrian (combinatorial) proof.

  \medskip
  The proof is by induction. Thus assume that $F\colon \gamma^n\to
  \gamma^{n-1}$ is a $d$-fold cover.

  Assume the $(n+1)$-edge $E'$ succeeds the $(n+1)$-edge $D'$ in
  $\gamma^{n+1}$. We need to show that the $n$-edge $E:=F(E')$ succeeds
  the $n$-edge $D:=F(D')$ in $\gamma^n$.  

  Let $A',B'\subset \bigcup \E^n$ be the two arcs that are mapped by
  $H^n_1$ to $D',E'$, see Lemma~\ref{lem:Hn_maps_edges}. Let
  $A:=F(A'), B:= F(B')\subset\bigcup 
  \E^{n-1}$. Since $H^{n}$ is the 
  lift of $H^{n-1}$ by $F$ (the diagram commutes)
  \begin{equation*}
    H^{n-1}_1(A)=D, \quad H^{n-1}_1(B)=E.
  \end{equation*}

  There are two cases to consider by Lemma \ref{lem:gamma_n_rel}.
  \begin{case}[1]
    $A',B'$ are contained in the same $n$-edge $E^n$, and $B'$
    succeeds $A'$ (given the orientation of $E^n$ by $\gamma^n$). 

    Note that since $F\colon \gamma^n\to \gamma^{n-1}$ is a $d$-fold
    cover, $F$ maps $n$-edges oriented by $\gamma^n$ to $(n-1)$-edges
    oriented by $\gamma^{n-1}$. 
 
    Therefore
    $A,B$ are contained in the same $(n-1)$-edge $E^{n-1}=F(E^n)$, and
    $B$ succeeds $A$ (given the orientation of $E^{n-1}$ by
    $\gamma^{n-1}$). Thus $E$ succeeds $D$ in $\gamma^n$.  
  \end{case}

  \begin{case}[2]
    $A',B'$ are contained in different $n$-edges $E(A'),E(B')$, such
    that $A', E(A')$ have the same terminal points, $B',E(B')$ have
    the same initial points, and $E(A')$, $E(B')$ are succeeding in
    $\gamma^n$.  

    Thus the $(n-1)$-edge $F(E(B'))\supset B$ succeeds
    $F(E(A'))\supset A$ in $\gamma^{n-1}$, 
    since $F\colon \gamma^n\to \gamma^{n-1}$ is a $d$-fold
    cover. Furthermore the terminal point of $A$ is the terminal point
    of $F(E(A'))$, which is the initial point of both $B, F(E(B'))$. 
    Thus $E$ succeeds $D$ in $\gamma^n$ by Lemma
    \ref{lem:gamma_n_rel}.   
  \end{case}
\end{proof}

By repeating the argument in Lemma \ref{lem:orientation-d-fold-cover}
we obtain inductively the following.

\begin{cor}
  \label{cor:gamman_po}
  All $n$-edges in the Eulerian circuit
  $\gamma^n$ are positively oriented (for each $n$).  
\end{cor}

\section{Construction of $\gamma$}
\label{sec:constr-g}


In this section we complete the construction of $\gamma$, i.e., the
proof of Theorem~\ref{thm:main}, under the assumption of
the existence 
of a pseudo-isotopy $H^0$ as in Definition~\ref{def:pseudo-isotopy-h0}.

\begin{lemma}
  \label{lem:theorem1reduction}
  To construct $\gamma\colon S^1\to S^2$ as in Theorem \ref{thm:main}
  it is enough to show the following. There is a Peano curve
  $\tilde{\gamma}\colon S^1\to S^2$ such that the diagram
  \begin{equation*}
    \xymatrix{
      S^1 \ar[r]^{\widetilde{\varphi}} \ar[d]_{\tilde{\gamma}}
      &
      S^1 \ar[d]^{\tilde{\gamma}}
      \\
      S^2 \ar[r]_F & S^2      
    }
  \end{equation*}
  commutes, where $\widetilde{\varphi}(z)= e^{2\pi i \theta_0} z^d$.
\end{lemma}

\begin{proof}
  Let $\mu:= e^{\frac{2\pi i \theta_0}{1-d}}$, this means that
    \begin{equation*}
      e^{2\pi i \theta_0}\mu^d=e^{2\pi i\theta_0} \mu^{d-1}\mu=\mu.
    \end{equation*}
    Consider $\gamma(z):= \tilde{\gamma}(\mu z)$. Then
    \begin{align*}
      F({\gamma}(z)) & =F(\tilde{\gamma}(\mu z))= \tilde{\gamma}(e^{2\pi i
        \theta_0}\mu^dz^d) = \tilde{\gamma}(\mu z^d)
      \\
      & = \gamma(z^d).
    \end{align*}
\end{proof}

In this section however we will drop the
``$\widetilde{\phantom{\varphi}}$'' from the notation. This means we
will write $\gamma,\gamma^n$, and so on; when in fact we mean
$\tilde{\gamma},\tilde{\gamma}^n$, which become our desired objects
by composing with a rotation as above.   

\subsection{The length of $n$-arcs}
\label{sec:length-edges-s1}

The circle $S^1$ will be divided into \defn{$n$-arcs}, each of which
will be
mapped by $\gamma^n$ to an $n$-edge. 
We first need to find the right ``length'' of such $n$-arcs. 
It will be convenient to parametrize those lengths by the
corresponding $n$-edges. Thus 
 $l(E)$ will be the length of the $n$-arc  (in $S^1$)
that is mapped by 
$\gamma^n$ to the $n$-edge $E$. We require
the following properties.
\begin{enumerate}[($l$ 1)]
\item
  \label{item:propl_1}
  $l(E)>0$ for every $n$-edge $E$.
\item 
  \label{item:propl_2}
  For all $n$,
  \begin{equation*}
    \sum_{E\in \E^n} l(E)=1.
  \end{equation*}
\item
  \label{item:propl_3}
  Given an $(n+1)$-edge $E'$ let $E=F(E')\in \E^n$. Then
  \begin{equation*}
    l(E)= d\, l(E').
  \end{equation*}
\item
  \label{item:propl_4}
  Let $E$ be an $n$-edge. Then $H^n_1(E)$ is a chain $E'_1,\dots,
  E'_N$ of $(n+1)$-edges. We require that
  \begin{equation*}
    l(E)=\sum_{m=1}^N l(E'_i).
  \end{equation*}
\end{enumerate}

To this end consider (all) $0$-edges $E_0,\dots, E_{k-1}$ ordered by the first
approximation $\gamma^0$ (mathematically positively on $\CC$). We say
an $n$-edge $E^n$ is of \defn{type} 
$j$ if $F^n (E^n)=E_j$. Recall that $H^0$ deforms each $0$-edge to
several $1$-edges. We define a matrix $M=(m_{ij})$, which keeps track of those
deformations, by
\begin{align*}
  m_{ij} \text{ is the number of } 1\text{-edges in } H^0_1(E_i)
  \text{ that are of type } j. 
\end{align*}

\begin{lemma}
  \label{lem:liftM}
  Consider an $n$-edge $E^n_i$ of type $i$. Let $\widetilde{m}_{ij}$
  be 
  the number of $(n+1)$-edges of type $j$ in $H^n_1(E^n_i)$. Then
  \begin{equation*}
    \widetilde{m}_{ij}=m_{ij}.
  \end{equation*}
  Furthermore, let $m^n_{ij}$ be the number of $n$-edges of type $j$
  contained in 
  $H^{n-1}_1\circ H^{n-2}_1\circ \dots \circ H^0_1(E_i)$. Then
  \begin{equation*}
    (m^n_{ij})=M^n.
  \end{equation*}
\end{lemma}

\begin{proof}
  Let $E^{n+1}_1,\dots,E^{n+1}_m$ be the $(n+1)$-edges in
  $H^n_1(E^n_i)$. Since $H^n$ is the lift of $H^0$ by $F^n$ it follows
  that $H^0$ deforms (the $0$-edge) $E_i=F^n(E^n_i)$ to the $1$-edges
  $E^1_1=F^n(E^{n+1}_1),\dots, E^1_m=F^n(E^{n+1}_m)$. The first
  statement follows, since $F^n$ preserves the type of edges.

  \medskip
  The second statement follows immediately from the first. 
\end{proof}

\begin{lemma}
  \label{lem:Mprimitive}
  The matrix $M$ is \emph{primitive}, i.e., $M^n>0$ for some $n$. 
\end{lemma}

\begin{proof}
  Recall from Section \ref{sec:gamman+1-d-fold} that $F\colon
  \gamma^{n+1}\to \gamma^{n}$ is a $d$-fold cover. Thus by induction
  $F^n\colon \gamma^n\to \gamma^0$ is a $d^n$-fold cover. Therefore
  along $\gamma^n$ the type of $n$-edges varies cyclically, in
  $\gamma^n$ an $n$-edge of type $j$ is succeeded by one of type
  $j+1$. This means that every chain of $k$ $n$-edges in $\gamma^n$
  contains exactly one $n$-edge of each type. 

  \smallskip
  Fix a $0$-edge $E_i$ connecting two postcritical points $p,q$.
  Consider $H^{n-1}_1\circ H^{n-2}_1\circ \dots \circ
  H^0_1(E_i)$. This is a chain of $n$-edges in $\gamma^n$ that
  connects the points $p,q$. 
  Since $F$ is expanding (see Definition \ref{def:f}~\eqref{def:fexpanding}), 
  the diameter of $n$-edges goes to $0$ (uniformly) with $n$. 
  Thus by
  choosing $n$ large enough, our chain contains at least $k$
  $n$-edges, therefore at least one $n$-edge of each type.

  \smallskip
  With this choice of $n$ the claim follows from Lemma \ref{lem:liftM}.
\end{proof}

Note that there are $d$ $1$-edges of each type, thus $\sum_i
m_{ij}=d$. 
The Perron-Frobenius theorem (see for example 
\cite[Theorem 8.2.11 and
Theorem 8.1.21]{0704.15002}) implies that $d$ is a simple eigenvalue of $M$ (in
fact its spectral radius). Furthermore there is unique eigenvector
$l=(l_j)$ to $d$, such that $l_j>0$ (for all $j=0,\dots, k-1$) and
$\sum_j l_j=1$. We note that $l_j\subset \Q$ for all $j=0,\dots, k-1$. 
The \defn{length} of (an $n$-arc in $S^1$ corresponding to)
an $n$-edge $E^n_j$ of 
type $j$ is now defined as
\begin{equation}
  \label{eq:de_lEn}
  l(E^n_j) := d^{-n}l_j. 
\end{equation}

\begin{lemma}
  \label{lem:l_prop}
  The \defn{length} defined above satisfies Properties
  \upshape{($l$ \ref{item:propl_1})--($l$ \ref{item:propl_4})}. 
\end{lemma}

\begin{proof}
  ($l$ \ref{item:propl_1}) follows immediately, since $l_j>0$ for all
  $j$.  

  There are $d^n$ $n$-edges of each type. Thus
  \begin{equation*}
    \sum_{E\in \E^n} l(E)=\sum_j l_j=1,
  \end{equation*}
  which is property ($l$ \ref{item:propl_2}).

  ($l$ \ref{item:propl_3}) is again clear, since $F$ maps
  $(n+1)$-edges to $n$-edges of the same type. 

  Property ($l$ \ref{item:propl_4}) follows from $Ml=dl$. Let $E^n_i$
  be an $n$-edge of type $i$, and $E^{n+1}_1,\dots, E^{n+1}_N$ be the
  $(n+1)$-edges contained in $H^n_1(E^n_i)$. Then by Lemma
  \ref{lem:liftM}
  \begin{equation*}
    \sum_m l(E^{n+1}_m)= d^{-n-1}\sum_j m_{ij} l_j=d^{-n} l_i=l(E^n_i).  
  \end{equation*}
\end{proof}

Note that the lengths depend on the particular pseudo-isotopy
$H^0$ chosen, it is not a property of the edges alone. 

\subsection{Parametrizing $\gamma^n$}
\label{sec:parametrizing-gn}

Fix a postcritical point $p_0$. Consider the Eulerian circuit
$\gamma^0=\CC=\bigcup \E^0$ 
\begin{equation*}
  \gamma^0=E_0, \dots ,E_{k-1}, \quad (E_j\in \E^0).   
\end{equation*}
It is labeled such that the
initial point of $E_0$ is $p_0$. Recall that 
we want to parametrize $\gamma$ such that $\varphi=e^{2\pi i
  \theta_0}z^d$ is semi-conjugate to $F$ (see Lemma
\ref{lem:theorem1reduction}).  
We now define $\theta_0$. If $p_0$ is a fixed point of $F$ set
$\theta_0:= 0$. Otherwise 
let $E_0,\dots ,E_{m^0-1}$ be the
(unique) positively oriented chain in $\gamma^0$ from $p_0$ to
$F(p_0)$. Then  
\begin{equation}
  \label{eq:deftheta0}
  \theta_0:= l(E_0)+ \dots + l(E_{m^0-1}).  
\end{equation}

Label $\gamma^1=E^1_0,\dots, E^1_{kd-1}$ such that $E^1_0$ is the
initial $1$-edge of the chain $H^0_1(E_0)$ in $\gamma^1$.  
In the same fashion label (the Eulerian circuit)
\begin{equation*}
  \gamma^{n}=E^{n}_0,\dots, E^{n}_{kd^{n}-1}, \quad (E^n_j\in \E^n)  
\end{equation*}
such that $E^n_0$ is the initial $n$-edge in
$H^{n-1}_1(E^{n-1}_0)$ (for each $n$). Thus the initial point of each
$E^n_0$ is 
$p_0$. Note however, that $\gamma^n$ may go through $p_0$ several
times. 

\smallskip
It will be convenient to identify $S^1$ with $\R/\Z$. Divide the
circle $\R/\Z$ into $k$ arcs $a_j$ as follows. Let
\begin{align}
  \label{eq:defalphaj}
  &\alpha_0:=0
  \\
  \notag
  &\alpha_j:= l(E_0)+ \dots + l(E_{j-1}), 
\end{align}
for $j=1,\dots,k-1$. Then $a_j:=[\alpha_j,\alpha_{j+1}]$ (where
indices are taken $\bmod k$). 

\begin{convention}
  When writing $[\alpha,\beta]\subset \R/\Z$ for an arc on the circle,
  we always mean the \emph{positively oriented} arc from $\alpha$ to
  $\beta$. In particular $a_{k-1}=[\alpha_{k-1},
  0]=[\alpha_{k-1},1]$. 
\end{convention}

In the same fashion we divide the circle $\R/\Z$ into $kd^n$
\defn{$n$-arcs} $a^n_j$ (for each
$n$) by
\begin{align*}
  &\alpha^n_0:=0
  \\
  &\alpha^n_j:= l(E^n_0)+ \dots + l(E^n_{j-1}), 
\end{align*}
for $j=1,\dots,kd^n-1$. Then $a^n_j:=[\alpha^ n_j,\alpha^n_{j+1}]$.

\begin{convention}
  The (lower) indices of points $\alpha^n_j$, $n$-arcs $a^n_j$, and $n$-edges
  $E^n_j$ are
  always taken $\bmod \,kd^n$. In particular $\alpha^n_{kd^n}=\alpha^n_0$,
  and $a^n_{kd^n-1}= [\alpha^n_{kd^n-1},0]=[\alpha^n_{kd^n-1},1]$.  
\end{convention}

We now define  
the approximations $\gamma^n$ on each
$n$-arc $a^n_j\subset \R/\Z$ by
\begin{align*}
  \gamma^n\colon  a^n_j \to E^n_j \text{ is (any) orientation-preserving
    homeomorphism,} 
\end{align*}
as \emph{parametrized curves}.
Thus initial/terminal points are mapped onto each other by
$\gamma^n$. 
Note that $\gamma^n(0)=p_0$ for all $n$. 

\smallskip
In $\R/\Z$ the map $\varphi(z)=e^{2\pi i \theta_0}z^d$ is
given by
\begin{equation*}
  \phi\colon \R/\Z \to \R/\Z, \quad \phi(t)= d t + \theta_0 \bmod 1.
\end{equation*}

\begin{lemma}
  \label{lem:paragn}
  The parametrized curves $\gamma^n$ satisfy the following.
  \begin{enumerate}
  \item
    \label{item:paragn_1}
    Let $m\geq n$, then each point $\alpha^n_j$ is a point
    $\alpha^m_i$. Furthermore
    \begin{equation*}
      \gamma^m(\alpha^n_j)=\gamma^n(\alpha^n_j),
    \end{equation*}
    for all $j=0,\dots ,kd^n-1$. Note that
    $\{\alpha^n_j\}=(\gamma^n)^{-1} (\V^n)$. So the $n$-th approximation
    determines the preimages (on the circle) of the 
    $n$-vertices.
  \item
    \label{item:paragn_2}
    The map $\phi$ maps each point $\alpha^{n+1}_j$ to a point
    $\alpha^n_i$. 
    For any point $\alpha^{n+1}_j\in \R/\Z$
    \begin{equation*}
      F(\gamma^{n+1}(\alpha^{n+1}_j)) = \gamma^n(\phi(\alpha^{n+1}_j)).
    \end{equation*}
    Thus we have the following commutative diagram,
    \begin{equation*}
      \xymatrix{
        \{\alpha^{n+1}_j\}\subset\R/\Z \ar[r]^\phi \ar[d]^{\gamma^{n+1}}
        & \{\alpha^n_j\}\subset\R/\Z \ar[d]^{\gamma^n}
        \\
        \V^{n+1}\subset S^2 \ar[r]_F
        & \V^n\subset S^2. 
      }
    \end{equation*}
    This will imply the desired semi-conjugacy.
  \item
    \label{item:paragn_3}
    The supremum norm is given in terms of the visual metric
    (\ref{eq:def_visuald}). Then
    \begin{equation*}
      \norm{\gamma^{n+1}-\gamma^n}_{\infty}\lesssim \Lambda^{-n},
    \end{equation*}
    for all $n$. Here $C(\lesssim)$ does not depend on $n$. 
  \end{enumerate}
\end{lemma}

\begin{proof}
  (\ref{item:paragn_1}) Consider $E_0$, the first $0$-edge in
  $\gamma^0$. Then $H^0_1(E_0)$ is the chain $E^1_0,\dots, E^1_{m-1}$ of
  $1$-edges in $\gamma^1$. Note that the terminal point of $E_0$ is
  the terminal point of $E^1_{m-1}$. By Property ($l$ \ref{item:propl_4}) 
  \begin{equation*}
    \alpha_1=l(E_0)=l(E^1_0)+ \dots + l(E^1_{m-1})=\alpha^1_m.
  \end{equation*}
  Thus 
  \begin{align*}
    \gamma^1(\alpha_1)
    &=\gamma^1(\alpha^1_m)
    = \text{terminal point of } E^1_m
    \\
    &= \text{terminal point of } E_0
    =\gamma^0(\alpha_1). 
  \end{align*}
  In the
  same fashion one shows that each $\alpha_j$ is a point $\alpha^1_i$, 
  and $\gamma^1(\alpha_j)=\gamma^0(\alpha_j)$
  for all $j=0,\dots, k-1$. The general statement follows by
  induction (see Lemma \ref{lem:liftM}). 


  \medskip
  (\ref{item:paragn_2}) 
  Recall from the definitions of $\theta_0$ (\ref{eq:deftheta0}) and
  the $\{\alpha_j\}$ (\ref{eq:defalphaj}) that  
  $\alpha_{m^0}=\theta_0$. Then 
  by (\ref{item:paragn_1}) and the definition of
  $\theta_0$ we have
  \begin{equation*}
    \gamma^n(\theta_0)=\gamma^0(\theta_0)= F(p_0). 
  \end{equation*}
  Let $m^n=m^n(\theta_0)$ be the index such that
  $\alpha^n_{m^n}=\theta_0$. 
  
  \smallskip
  Consider $E^{n+1}_0$, the initial $(n+1)$-edge in $\gamma^{n+1}$. It
  is clear that $F(E^{n+1}_0)$ is an $n$-edge with initial point
  $F(p_0)$ (by Corollary \ref{cor:gamman_po}). There may be several
  such $n$-edges in general however. We 
  next show that $F(E^{n+1}_0)$ is in fact the ``right'' $n$-edge,
  namely the image (by $\gamma^n$) of the $n$-arc (on $\R/\Z$) with initial
  point $\theta_0$.  
  \begin{claim1}
    $F(E^{n+1}_0) = \gamma^n(a^n_{m^n})=E^n_{m^n}$.
  \end{claim1}
  This is clear for $n=0$, since there is only one $0$-edge with
  initial point
  $F(p_0)$. To prove the claim by induction, we assume it is true for
  $n-1$. 

  Consider $E^n_0$, 
  by assumption
  $F(E^n_0)=\gamma^{n-1}\left(a^{n-1}_{m^{n-1}}\right)=E^{n-1}_{m^{n-1}}$.   
  Let $A^n\subset E^n_0$ be the (initial) $n$-arc
  that is deformed by $H^{n}$ to $E^{n+1}_0$. 
  Let $A^{n-1}:=F(A^n)\subset E^{n-1}_{m^{n-1}}$, it is an $n$-arc that is
  deformed by $H^{n-1}$ 
  to an $n$-edge $E^n_j$ (since $H^n$ is the lift
  of $H^{n-1}$ by $F$). 
  \begin{equation*}
    \xymatrix{
      A^n\subset E^n_0 \ar[rr]^{H^n_1} \ar[d]_F
      & & E^{n+1}_0 \ar[d]^F
      \\
      A^{n-1}\subset E^{n-1}_{m^{n-1}} \ar[rr]_>>>>>>>>>>{H^{n-1}_1} 
      & & E^n_j
    }
  \end{equation*}

  The crucial property is that by construction
  $j=m^n$. This is seen as follows. By ($l$ \ref{item:propl_4}) the
  total length of the $(n-1)$-edges preceding $E^{n-1}_{m^{n-1}}$
  (which is $\theta_0$) is
  the same as the total length of all $n$-edges preceding
  $E^n_j$,
  \begin{align*}
    \theta_0 & =l(E^{n-1}_0)+\dots +l(E^{n-1}_{m^{n-1}-1})
    \\
    &=l(E^n_0)+\dots+l(E^n_{j-1}),
    \\
    & \text{thus }  j=m^n.
  \end{align*}
  Hence $F(E^{n+1}_0)=E^n_{m^n}$, since the diagram above
  commutes. This proves Claim 1. 

  \begin{claim2}
    $F(E^{n+1}_j) = E^n_{m^n + j}$, for $j=0,\dots,
    kd^{n+1}-1$.  
  \end{claim2}
  This follows from Claim 1, and the fact that $F\colon
  \gamma^{n+1}\to \gamma^n$ is a $d$-fold covering in the sense of     
  Definition \ref{def:dfoldcover}. The reader is reminded (for the
  last time) that the index $m^n+j$ is taken $\bmod \, kd^n$. 

  \begin{claim3}
    The map $\phi$ maps points $\alpha^{n+1}_j$ to points $\alpha^{n}_i$,
    in fact
    \begin{equation*}
      \phi(\alpha^{n+1}_j) = \alpha^n_{m^n+j}.
    \end{equation*}
  \end{claim3}
  To prove this claim note first that
  \begin{equation*}
    \phi(\alpha^{n+1}_0) = \phi(0) = \theta_0 = \alpha^n_{m^n}
  \end{equation*}
  by definition. 
  In the following we write $\alpha\equiv\beta$ if $\alpha, \beta$
  represent the same point on the circle $\R/\Z$, i.e., if
  $\alpha-\beta\in \Z$. 

  By the previous claim
  $F(E^{n+1}_j)= E^n_{m^n+j}$, thus 
  \begin{equation*}
    l(E^n_{m^n+j})=d\,l(E^{n+1}_j)
  \end{equation*}
  by Property ($l$
  \ref{item:propl_3}). Therefore
  \begin{align*}
    \alpha^n_{m^n+j} 
    & \equiv \alpha^n_{m^n} + l(E^n_{m^n }) + l(E^n_{m^n+1})
     + \dots + l(E^n_{m^n+j-1})
    \\
    & = \theta_0 + d\,(l(E^{n+1}_0) + \dots + l(E^{n+1}_{j-1}))
    \\
    & = \theta_0 + d \alpha^{n+1}_j \equiv \phi(\alpha^{n+1}_j),
  \end{align*}
  for $j=0,\dots, kd^{n+1}-1$. Thus Claim 3 is proved.

  \medskip
  It remains to show the semi-conjugacy.
  Note that by construction $\gamma^n$ maps
  $\alpha^n_j$ to the initial point of $E^n_j$. Thus
  \begin{align*}
    F(\gamma^{n+1}&(\alpha^{n+1}_j))
    = F(\text{initial point of } E^{n+1}_j)
    \\
    &= \text{ initial point of } E^{n}_{m^n+j} 
    &&\text{ by Claim 2}
    \\
    &= \gamma^n(\alpha^{n}_{m^n+j}) = \gamma^n(\phi(\alpha^{n+1}_j)) 
    &&\text{ by Claim 3.}
  \end{align*}
  This finishes the proof of property (\ref{item:paragn_2}).

  \medskip
  (\ref{item:paragn_3})
  The diameter of each $n$-edge $E^n$ in the visual metric
  (\ref{eq:def_visuald}) is given by
  \begin{equation*}
    \diam E^n\asymp \Lambda^{-n},
  \end{equation*}
  see \cite[Lemma~8.4]{expThurMarkov}. 
  
  Consider one $n$-arc $a^n_j=[\alpha^n_j,
  \alpha^n_{j+1}]$. Then $\gamma^n(a^n_j)=
  E^n_j$. The pseudo-isotopy $H^n$ deforms $E^n_j$ to a
  $(n+1)$-chain $E^{n+1}_{i},\dots, E^{n+1}_{i+m-1}$.  
  The number $m$ (of $(n+1)$-edges in this chain) is uniformly
  bounded by Lemma \ref{lem:liftM}.
  By (the proof of) property (\ref{item:paragn_1}) it holds
  $\alpha^n_j=\alpha^{n+1}_{i}$ and 
  $\alpha^n_{j+1}=\alpha^{n+1}_{i+m}$, and so
  \begin{align*}
    &a^n_j
    =
    a^{n+1}_i
    \cup \dots \cup
    a^{n+1}_{i+m-1}, \text{ where}
    \\ 
    &\gamma^{n+1}(a^{n+1}_i)=E^{n+1}_i,\dots,
    \gamma^{n+1}(a^{n+1}_{i+m-1})=E^{n+1}_{i+m-1}. 
  \end{align*}
  Furthermore the $(n+1)$-chain $E^{n+1}_i,\dots, E^{n+1}_{i+m-1}$ and
  the $n$-edge $E^n_j$ intersect in (the endpoints of $E^n_j$)
  $\gamma^n(\alpha^n_j)=\gamma^{n+1}(\alpha^{n+1}_{i})$ and 
  $\gamma^n(\alpha^n_{j+1})=\gamma^{n+1}(\alpha^{n+1}_{i+m})$, again by
  property (\ref{item:paragn_1}). Thus on $a^n_j$
  \begin{align*}
    \norm{\gamma^n-\gamma^{n+1}}_{\infty}
    &\leq 
    \diam E^n_j + \diam E^{n+1}_i + \dots + \diam E^{n+1}_{i+m-1}
    \\
    &\lesssim
    \Lambda^{-n} + m \Lambda^{-n-1} \lesssim \Lambda^{-n},
  \end{align*}
  as desired.

\end{proof}

\subsection{Construction of the invariant Peano curve $\gamma$}
\label{sec:proof-theorem1}
We now come to the proof of the main result, assuming the existence of  
a pseudo-isotopy $H^0$ as in Definition \ref{def:pseudo-isotopy-h0}. 

\smallskip
Define
\begin{equation*}
  \gamma\colon \R/\Z \to S^2, 
  \quad
  \gamma(t):= \lim_n \gamma^n(t).
\end{equation*}
Since the sequence $(\gamma^n)$ converges uniformly by Lemma
\ref{lem:paragn} 
(\ref{item:paragn_3}) this is a parametrized curve.

\begin{claim1}
  $\gamma$ is a Peano curve (onto).
\end{claim1}

This is clear since the curve $\gamma$ contains by construction
$\bigcup_n \V^n$ (all $n$-vertices). This set is dense in $S^2$.

\begin{claim2}
  $F(\gamma(t))=\gamma(\phi(t))$, for all $t\in \R/\Z$.
\end{claim2}

Note that by properties (\ref{item:paragn_1}),(\ref{item:paragn_2}) of
Lemma \ref{lem:paragn} 
this is true for all $t=\alpha^n_j$. The claim follows, since the
set of all such points 
$\alpha^n_j$ is
dense in the circle $\R/\Z$. 

\medskip
Thus we ``just'' need to construct the pseudo-isotopy $H^0$ (with
Properties ($H^0$~\ref{item:H0_1})--($H^0$~\ref{item:H0_5})) to finish
the proof of Theorem~\ref{thm:main}.  

\subsection{$\gamma$ is the end of a pseudo-isotopy}
\label{sec:glimit_jordan}

The homotopy $\Gamma\colon S^2\times [0,1]\to S^2$ from
Theorem~\ref{thm:main} is constructed as follows. Roughly speaking we 
concatenate the homotopies $H^n$. 
The precise
definition is as follows.
Break up the unit interval into intervals
\begin{equation*}
  I=[0,1]= \left[0,\frac{1}{2}\right]\cup 
  \left[\frac{1}{2},\frac{3}{4}\right] \cup \dots \cup 
  \left[1-2^{-n},1 -2^{-n-1}\right]
  \cup \dots \cup \{1\}. 
\end{equation*}
The $n$-th interval in this union is denoted by $I^n=[1-2^{-n},1
-2^{-n-1}]$. Let $s_n\colon I^n \to I$, $s_n(t)= 2^{n+1}(t- (1
-2^{-n}))$, for $n\in \N_0$. We define  
$\Gamma\: S^2\times I\ra S^2$ by $\Gamma(x,t)= H^0(x,s_0(t))$
for $t\in I^0$, $\Gamma(x,t)= H^1(H^0_1(x), s_1(t))$ for $t\in
I^1$. In general
\begin{align*}
  \Gamma(x,t) := H^{n}(H^{n-1}_1\circ \dots \circ H^0_1(x),s_n(t))
\end{align*}
if $t\in I^n$ (for some  $n\in \N_0,$) and all $x\in S^2$. Since the
diameters of $H^n$ tend to $0$ exponentially (see
Lemma~\ref{lem:lift_degenerate_isotopies}~\eqref{item:lift2}), it
follows that $\Gamma$ extends to $t=1$ by $\Gamma(x,1):= \lim_{t\to 1}
\Gamma(x,t)$ continuously. This is the desired homotopy.   

\smallskip
It is possible to choose $\Gamma$ to be a pseudo-isotopy. This can be
done explicitly by slightly altering the above construction. 
We do not work out the details here. It is however a direct
consequence of the general theory of decomposition spaces. 
Namely it follows from the fact that
every \emph{cell-like upper semicontinuous decomposition} of a
$2$-manifold is \emph{shrinkable} \cite[Theorem 25.1]{MR872468}.    

\section{Some topological Lemmas}
\label{sec:some-topol-lemm}

Here we collect some topological theorems/lemmas for future reference.
We first note the following form of the Jordan-Sch\"{o}nflies theorem.
\begin{theorem}[Isotopic Sch\"{o}nflies theorem]
  \label{thm:Schoenflies}
  Let $\gamma,\sigma\subset \D$ be two Jordan arcs with common
  endpoints $p,q\in \Dbar$. Then there is an isotopy of $\Dbar$ rel. $\partial
  \D\cup \{p,q\}$ that deforms $\gamma$ to $\sigma$.
\end{theorem}

We give a quick outline how this form can be obtained from the
standard Sch\"{o}nflies theorem.

\begin{theorem}[Sch\"{o}nflies theorem, see {\cite[Theorem 10.4]{MoiseTop23}}]
  \label{thm:Schoenflies_standard}
  Let $h\colon J\subset \R^2 \to \widetilde{J}\subset \R^2$ be a
  homeomorphism, where $J$ is a Jordan curve. Then $h$ may be extended
  to a homeomorphism $h\colon\R^2\to \R^2$. 
\end{theorem}

We remind the reader of the \emph{Alexander trick}. 
\begin{theorem}[Alexander, see {\cite[Theorem 11.1]{MoiseTop23}}]
  \label{thm:Alexander}
  Let $h\colon \Dbar\to \Dbar$ be a homeo\-morphism, such that
  $h|_{S^1}=\id_{S^1}$. Then the map $\phi \colon \Dbar\times [0,1]$
  defined by
  \begin{align*}
    \phi(x,t):=
    \begin{cases}
      t h(x/t),  & 0\leq \abs{x}\leq t,
      \\  
      x, & t\leq \abs{x}\leq 1;
    \end{cases}
  \end{align*}
  is an isotopy with $\phi(\cdot, 0)=\id_{\Dbar}$, $\phi(\cdot,
  1)=h$. 
\end{theorem}

\begin{proof}[Proof of Theorem \ref{thm:Schoenflies}, outline]
  Consider first $p,q\in S^1=\partial \Dbar$. Let $C_1,C_2\subset S^1$
  be the two arcs bounded by $p,q$. Let $h_i\colon \gamma\cup C_i\to
  \sigma\cup C_i$ be homeomorphisms constant on $S^1$ ($i=1,2$). Using
  Theorem \ref{thm:Schoenflies_standard} they can be extended to a
  homeomorphism of $\Dbar$. Theorem \ref{thm:Alexander} gives the
  desired isotopy. 

  If $p=0,q\in S^1$ extend $\gamma,\sigma$ to arcs with common
  endpoints $\tilde{p},q\in S^1$. The previous procedure yields the
  isotopy. 

  If $p\in \D, q\in S^1$ we use the same construction as before. Then
  we post-compose with the isotopy that maps the rays between $\phi(p,t)$
  and $\zeta\in S^1$ to the rays between $p$ and $\zeta\in S^1$. 

  Finally let $p,q\in \D$. By the above we can assume that
  $p=0$. Extend $\gamma,\sigma$ to curves
  $\tilde{\gamma},\tilde{\sigma}$ with common endpoints
  $\tilde{p},\tilde{q}$. As above we obtain an isotopy $\phi(x,t)$
  rel.\ $S^1\cup\{p\}$ deforming $\tilde{\gamma}$ to
  $\tilde{\sigma}$. We can assume that $\phi(q,1)=q$ (choose the
  homeomorphisms $h_i$ such that $h_i(q)=q$). This means that $\phi$
  deforms $\gamma$ to $\sigma$. 
  Let $r_t:= \abs{\phi(q,t)}$ and $\alpha_t:= \log
  r_0/\log r_t$. Then post-composition with the \emph{radial stretch}
  \begin{equation*}
    \psi(x,t):=\abs{x}^{\alpha_t}\frac{x}{\abs{x}}
  \end{equation*}
  yields an isotopy $\widetilde{\phi}$ rel.\ $S^1\cup\{p\}$ which
  keeps $\abs{q}$ constant. Let $\theta_t:= \arg
  \widetilde{\phi}(q,t)-\arg q$. Post-composing with
  \begin{equation*}
    \varphi\colon re^{i\theta}\mapsto re^{i(\theta - \frac{1-r}{1-\abs{q}}\theta_t)}
  \end{equation*}
  yields the desired isotopy. There is a tricky point hidden here:
  $\theta_1$ could be a multiple of $2\pi$. We can however always
  arrange that $\theta_1=0$ in the following way. Let
  $\tilde{\gamma}|[\tilde{q},q]$, $\tilde{\sigma}|[\tilde{q},q]$ be
  the paths of the extensions from $\tilde{q}$ to $q$. By choosing the
  extensions $\tilde{\gamma},\tilde{\sigma}$ in such a way that the
  change of argument along $\tilde{\gamma}|[\tilde{q},q]$ and
  $\tilde{\sigma}|[\tilde{q},q]$ is equal, it follows that
  $\theta_1=0$.
\end{proof}

The following is due to Epstein-Zieschang, see
\cite[Theorem A.5]{MR1183224}. 

\begin{theorem}[Isotopy rel.\ $\post$]
  \label{thm:Isotopy_rel_post}
  Let $\CC,\gamma\subset S^2$ be two Jordan curves going through the
  postcritical points $p_0,\dots, p_{k-1}$ in the same cyclical
  order. 
  Let $\CC_j$ and $\gamma_j$ be the arcs on $\CC$ and
  $\gamma$ between $p_j$ and $p_{j+1}$ (indices are taken $\bmod\, k$ here). 
  Then the following conditions are equivalent:
  \begin{align*}
    &(1)\quad  \CC_j \text{ and } \gamma_j \text{ are isotopic
      rel.\ } \post \text{ for all } j=0,\dots, k-1;
    \\
    &(2)\quad \CC,\gamma \text{ are isotopic rel.\ }\post. 
  \end{align*}
\end{theorem}

Combining the previous with Theorem \ref{thm:Schoenflies} we obtain
the following. 
\begin{theorem}
  \label{thm:isotopy2}
  With notation as in the previous theorem assume that
  \begin{equation*}
    \CC_i \cap \gamma_j \neq \emptyset \quad
    \text{only for } j=i-1,i,i+1.
  \end{equation*}
  Then $\CC,\gamma$ are isotopic rel.\ $\post$. 
\end{theorem}

\section{Connections}
\label{sec:connections}

In this and the following section the initial pseudo-isotopy $H^0$ is
constructed. This was used to define the first approximation
$\gamma^1$ of the Peano curve. Recall that $\gamma^1$ is an Eulerian
circuit of $1$-edges. Thus $\gamma^1$ is given by the following. For
each $1$-edge $E$ ending at a
$1$-vertex $v$ we have to define a \defn{succeeding} $1$-edge $E'\ni
v$. Since $\gamma^1$ will be non-crossing, there will be an even number
of $1$-edges in the sector between $E,E'$ (as well as in the sector
between $E',E$). Let $E$ be contained in the 
white $1$-tile $X$, and $E'$ be contained in the white $1$-tile $X'$. 
From the above it follows that if $\gamma^1$ traverses $E$ positively
(as boundary of $X$) it traverses $X'$ positively (as boundary of
$X'$). 

Since $\gamma^1$ is non-crossing it is possible to ``distort the
picture'' in a neighborhood of $v$ slightly, so that the resulting
curves are simple. In this distorted picture the $1$-tiles $X,X'$ are
\defn{connected at $v$}. See Figure \ref{fig:connection} for an
illustration.   
  
\smallskip
Formally we will do the reverse of the description above. Namely
at each $1$-vertex we will define a \defn{connection}, which is an
assignment which $1$-tiles are connected. This will be done in a
\defn{non-crossing} manner. The approximation $\gamma^1$ and
the pseudo-isotopy $H^0$ are constructed from the connection of (all)
$1$-tiles. 

\subsection{Non-crossing partitions}
\label{sec:non-cross-part}

Recall that
a \defn{partition} of the set $[n]:=\{0,\dots, n-1\}$ is a set
$\pi=\{b_1,\dots ,b_N\}$ of 
pairwise disjoint subsets (called \defn{blocks}) of $[n]$, whose union
is $[n]$. It  
is \defn{crossing} if and only if 
it contains distinct blocks $b_i,b_j$
with $a,c\in b_i$, $b,d\in b_j$
such that 
\begin{equation*}
  0\leq a<b<c<d\leq n-1;
\end{equation*}
otherwise non-crossing.

 It is easy to see that the partition 
$\pi=\{b_1,\dots,b_N\}$ of $[n]$ is non-crossing if and only if the
sets $B_i:=\{e_m \mid m\in b_i\}$, where $e_m:=e^{2\pi i
  \frac{m}{n}}$, have the property that each $B_i$ lies in one
component of $S^1\setminus B_j$ (for $i\ne j$). 

With this description in mind let (for $i,j\in[n]$)
\begin{align}
  \label{eq:notation_pi1}
  [i,j]&:= 
  \begin{cases}
    \{i,\dots j\},                        & \text{if $i\leq j$,}
    \\
    \{i,\dots, n-1\} \cup \{0,\dots ,j\}, & \text{if $i> j$;} 
  \end{cases}  
  \\
  \notag
  (i,j)&:=[i,j]\setminus\{i,j\}.
\end{align}
Let $b=\{j_0,\dots, j_m\}\subset[n]$, where $j_0< \dots < j_m$, then     
a \defn{component} of $[n]\setminus b$ is defined to be one of the
sets 
\begin{equation*}
  (j_0,j_1),\dots, (j_{m-1},j_m), (j_m, j_0).
\end{equation*}
The partition $\pi=\{b_1,\dots,b_N\}$ is non-crossing if and only if
each $b_i$ lies in one component of $[n]\setminus b_j$ for all $i\ne
j$.  



\medskip
The set of non-crossing partitions (or \defn{nc-partitions}) of $[n]$
is partially ordered by 
refinement. Namely for two partitions $\pi,\sigma$ one defines 
$\sigma\leq \pi$ if and only if every block in $\pi$ is the union of
blocks in $\sigma$. Equipped with this partial ordering the
nc-partitions (of $[n]$) form a 
\defn{lattice}, i.e., \defn{meet} and \defn{join} are well
defined. 
The meet of
(non-crossing) partitions $\pi_1,\dots, \pi_m$ is 
\begin{equation}
  \label{eq:def_meet}
  \bigwedge_{i=1}^m \pi_i:=\{b_1\cap \dots \cap b_m\mid b_i\in
  \pi_i\}.    
\end{equation}
It is the biggest (non-crossing) partition smaller than any
$\pi_i$. The join is the smallest nc-partition bigger than any $\pi_i$
(the description is slightly more difficult). 

Non-crossing partitions were introduced in \cite{kreweras},
see \cite{SimionNC} for a recent survey. The number of 
nc-partitions of $[n]$ is equal to the \defn{$n$-th Catalan number}
$C_n:=\frac{1}{n+1}\binom{2n}{n}$.





\medskip
Consider now $\even=\even_n=\{2m\mid m=0,\dots, n-1\},
\odd=\odd_n=\{{2m+1}\mid m=0,\dots, n-1\}$, so that $[{2n}]=\even\cup
\odd$. 

Non-crossing partitions of $\even$/$\odd$ 
are defined as
before. 
We denote by 
$\pi_w$ a nc-partition
of $\even$, 
by $\pi_b$ a nc-partition of $\odd$. They will describe how white
(black) tiles are connected at a vertex $v$; see again Figure
\ref{fig:connection} for an illustration, Figure \ref{fig:ind_conn}
for a more complicated example.

\begin{lemma}
  \label{lem:comppart}
  Let $\pi_w$ be a partition of $\even_n$. 
  Then there is a unique maximal non-crossing partition
  $\pi_b=\pi_b(\pi_w)$ of $\odd_n$ such that $\pi_w\cup \pi_b$ is a
  non-crossing partition of $[2n]$. 
\end{lemma}

\begin{proof}
  Fix a block $b_i\in \pi_w$. Let $c_1,\dots, c_M$ be the components
  of $[2n]\setminus b_i$. Let 
  \begin{equation*}
    a_j:= \odd \cap c_j, \quad j=1,\dots M. 
  \end{equation*}
  Then
  $\pi_b(b_i):=\{a_1,\dots , a_M\}$. This is a nc-partition
  of $\odd$. 
  We now define (see (\ref{eq:def_meet}))
  \begin{equation*}
    \pi_b:=\bigwedge_i \pi_b(b_i),
  \end{equation*}
  this is a non-crossing partition of $\odd$. Also $\pi_w\cup\pi_b$ is
  a non-crossing partition of $[2n]$.

  \medskip
  Let $\sigma_b$ be any non-crossing partition of $\odd$ such that
  $\pi_w\cup 
  \sigma_b$ is a nc-partition of $[2n]$. Then $\sigma_b\leq
  \pi_b(b_i)$ for all $i$. Thus $\sigma_b\leq \pi_b$.   
\end{proof}

The partition $\pi_b=\pi_b(\pi_w)$ is called the \defn{partition
  complementary} to $\pi_w$. We mention some more facts which can be
found in \cite[Section 3]{kreweras}. 

\begin{lemma}[Properties of complementary partitions]
  \label{lem:prop_comp}
  Complementary partitions have the following properties.

  \begin{itemize}
  \item 
    Two blocks $a,b$ are called
    \defn{adjacent} if there are $i\in a$, $j\in b$ such that
    $i+1\in b$, $j+1\in a$. The partition $\pi_w\cup\pi_b$ has the
    property that the two blocks containing $i$ and $i+1$ are
    adjacent for all $i$. This characterizes $\pi_b$, meaning it is the
    unique nc-partition of $\odd$, such that $\pi_w\cup\pi_b$ is
    non-crossing, with this property.
  \item 
    One may define $\pi_w=\pi_w(\pi_b)$, the partition (of $\even$)
    complementary to the partition $\pi_b$ (of $\odd$) as before. Then the
    previous characterization shows that $\pi_w(\pi_b(\pi_w))=\pi_w$. Thus
    we simply say that 
    the partitions $\pi_w,\pi_b$ are \defn{complementary}.
  \item 
    It is possible to define a \defn{graph}, where the vertices are the blocks of
    $\pi_w\cup\pi_b$, connected by edges if and only if they are
    adjacent. It is not very hard to show that this is a \defn{tree} with $n$
    edges. Thus $\pi_w\cup\pi_b$ contains exactly $n+1$ blocks.   
  \end{itemize}
\end{lemma}

From now on we write \defn{cnc-partition} for complementary
non-crossing partitions $\pi_w\cup\pi_b$ as above.   

\smallskip
We next proceed to construct a \defn{geometric realization} of a given
cnc-partition; see again Figure \ref{fig:connection}. 



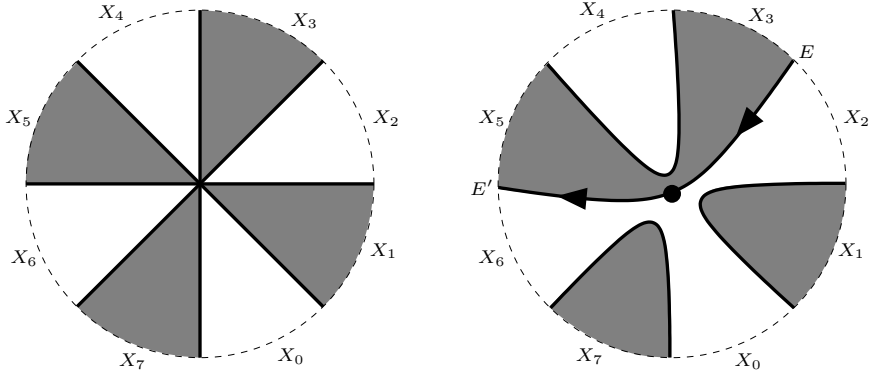
\begin{figure}
  \centering
  
  \includegraphics[width=11cm]{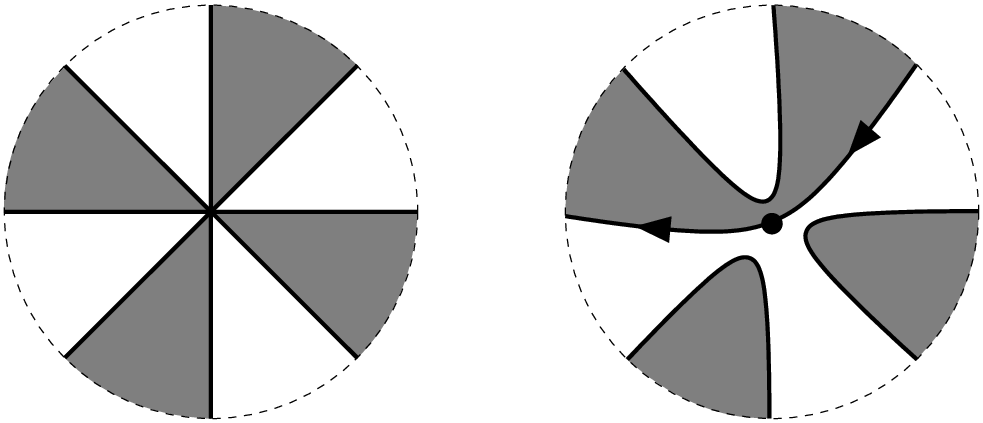}
  \begin{picture}(10,10)
    %
    \put(-220,0){$\scriptstyle{X_0}$}
    \put(-185,40){$\scriptstyle{X_1}$}
    \put(-184,90){$\scriptstyle{X_2}$}
    \put(-215,127){$\scriptstyle{X_3}$}
    \put(-288,130){$\scriptstyle{X_4}$}
    \put(-323,90){$\scriptstyle{X_5}$}
    \put(-321,37){$\scriptstyle{X_6}$}
    \put(-280,-2){$\scriptstyle{X_7}$}
    %
    \put(-47,-2){$\scriptstyle{X_0}$}
    \put(-8,40){$\scriptstyle{X_1}$}
    \put(-6,90){$\scriptstyle{X_2}$}
    \put(-42,129){$\scriptstyle{X_3}$}
    \put(-106,131){$\scriptstyle{X_4}$}
    \put(-144,90){$\scriptstyle{X_5}$}
    \put(-144,37){$\scriptstyle{X_6}$}
    \put(-107,0){$\scriptstyle{X_7}$}    
    \put(-23,115){$\scriptstyle{E}$}
    \put(-147,63){$\scriptstyle{E'}$}        
  \end{picture}
  \caption{Connection at a vertex.}
  \label{fig:connection}
\end{figure}

\medskip
Divide the unit disk into $n+1$ (simply connected) domains $D_1,\dots,
D_{n+1}$ by $g_1,\dots,g_n\subset \Dbar$
disjoint Jordan arcs. More precisely, the (distinct) endpoints of each
$g_j$ lie in $S^1=\partial \D$, the interior of $g_j$ in
$\D$. The arcs $g_m$ divide $S^1$ into $2n$ circular arcs $a_0,\dots
a_{2n-1}\subset S^1$ (labeled mathematically positively on $S^1$). A
partition $\pi(\{g_m\})$ of
$[2n]$ is obtained as follows.
\begin{align}
  \label{eq:defpart_geo}
  &i,j\in [2n] \text{ are in the same block of } \pi(\{g_m\})
  \intertext{if and only if }
  \notag
  &a_i,a_j
  \text{ are in the boundary of the same component } D_l.
\end{align}
So for each component $D_l$ of $\D\setminus \bigcup g_j$ there is
exactly one block $b_l\in \pi(\{g_m\})$. 
\begin{lemma}
  \label{lem:geom_real_comp_part}
  The partition $\pi(\{g_m\})$ is a \emph{cnc-partition}. Conversely
  each cnc-partition of $[2n]$ is obtained in this way. 
  
  \smallskip
  Furthermore
  $\overline{D}_k, \overline{D}_l$ are \emph{not disjoint} if
    and only if the 
    (corresponding) blocks $b_k,b_l$ are \emph{adjacent}. In this case
    the intersection of $\overline{D}_k,\overline{D}_l$ is one arc
    $g_m$. Conversely each $g_m$ is the intersection of the closure of
    two components $\overline{D}_k,\overline{D}_l$.    
\end{lemma}

\begin{proof}
  We first show that $\pi(\{g_m\})$ is non-crossing. 
  Consider distinct components ${D}_k,{D}_l$. Then there is a Jordan
  arc $g_m\subset \partial D_k$ that separates $D_k$ from $D_l$. Let
  $\alpha, \beta\in S^1$ be the endpoints of $g_m$. Let
  $a_i,a_{i+1}\subset S^1$ and $a_j,a_{j+1}\subset S^1$ be the
  circular arcs containing $\alpha,\beta$. We can assume that
  $a_i\subset \partial D_k$, then $a_{j+1}\subset \partial D_k$. Then
  all arcs in the boundary of $D_l$ are contained in $a_{i+1}, \dots,
  a_j$. This means that $b_l\subset [i+1,j]$, which is one component
  of $[2n]\setminus b_k$ (recall that $b_k$ is the block corresponding
  to $D_k$, $b_l$ the block corresponding to $D_l$, see
  (\ref{eq:notation_pi1}) for notation). This shows that
  $\pi(\{g_m\})$ is non-crossing. 

  If $\partial D_l\supset a_{i+1}$ ($\Leftrightarrow i+1\in b_l$) it
  follows that $g_m\subset \partial D_l$. Thus $a_j\subset \partial D_l$
  ($\Leftrightarrow j\in b_l$). Thus $i, j+1\in b_k$ and $i+1,j\in
  b_l$, meaning that $b_k,b_l$ are adjacent. This shows that the
  partition $\pi(\{g_m\})$  is a cnc-partition. 

  Furthermore it is clear that $b_k,b_l$ are adjacent if and only if
  $\overline{D}_k,\overline{D}_l$ intersect.



  \smallskip
  It remains to show that each cnc-partition is obtained in this
  geometric fashion.
  Identify each $j\in[2n]$ with the circular arc
  $a_j=[e_j,e_{j+1}]\subset S^1$ ($e_j=e^{2\pi i \frac{j}{2n}}$). For each block
  $b_l\in \pi_w\cup \pi_b$ the domain $D_l$ is the hyperbolic polygon
  whose boundary intersects $S^1$ in $\bigcup_{i\in b} a_i$. 

  To be more precise, for each two adjacent blocks $b\ni i,j+1$, $b'\ni
  i+1,j$ we connect $e_{i+1},e_{j+1}$ by a hyperbolic geodesic. Since
  every block distinct from $b$ is contained in one component of
  $[2n]\setminus \{i,j+1\}$ the Jordan arcs $g_m$ thus obtained are
  disjoint. 
\end{proof}

How $1$-tiles are
connected at a $1$-vertex $v$ will be described by complementary
non-crossing partitions. Additional data is needed however, to make the
construction well defined. 
Namely if $v=p$ is a postcritical point we need to
declare where $p$ lies in the ``distorted picture'' (in the \emph{geometric
representation} of the complementary connections, see below). 

\begin{definition}[Marking]
  \label{def:mark_partition}
  A cnc-partition $\pi_w\cup \pi_b$ is \defn{marked} by singling out
  a pair of \defn{adjacent blocks} $b,c\in \pi_w\cup
  \pi_b$. Equivalently this 
  means that if the cnc-partition $\pi_w\cup \pi_b$ is given
  geometrically as above in Lemma~\ref{lem:geom_real_comp_part}, we
  mark one of the arcs $g_m$.   
  In Figure \ref{fig:connection} the marked arc
  $g_m$ is indicated by the big dot.
  
  Given a marked cnc-partition we always assume that the geometric
  realization from Lemma \ref{lem:geom_real_comp_part} was chosen such
  that \defn{the marked arc $g_m$ contains the origin}.
  
  A third equivalent way to mark a connection is given in
  Corollary~\ref{cor:marking}. 
\end{definition}

Assume now that the circular arcs from Lemma
\ref{lem:geom_real_comp_part} are 
of the form $a_j=[e_j,e_{j+1}]\subset S^1$ ($e_j=e^{2\pi i
  \frac{j}{2n}}$).  
Color the set $D_l$ white if the corresponding block $b_l\in \pi_w$,
otherwise 
black. Thus we obtain a ``checkerboard tiling'' of the unit disk,
where sets which share a side $g_m$
have different color. 
\begin{definition}[Geometric representation of cnc-partition]
  \label{def:geom_representation}
  The decomposition of the closed unit disk 
into black and white sets as above is called a \defn{geometric
  representation} of 
the cnc-partition $\pi_w\cup\pi_b$, it is denoted by
$\Dbar(\pi_w\cup\pi_b)$. The union of white sets $\overline{D}_l$ is
denoted 
by $\Dbar_w=\Dbar_w(\pi_w\cup \pi_b)$, the union of black sets
$\overline{D}_l$ by $\Dbar_b=\Dbar_b(\pi_w\cup \pi_b)$.
\end{definition}
Denote by $S_j$ a sector in $\Dbar$ ($j=0,\dots, 2n-1$), 
\begin{equation}
  \label{eq:def_sector}
  S_j:=\left\{ r e^{2\pi i\theta} \bigm|
  \frac{j}{2n}\leq\theta\leq \frac{j+1}{2n}, 0\leq r\leq 1\right\}.      
\end{equation}

\begin{lemma}[Deforming $\Dbar(\pi_w\cup \pi_b)$]
  \label{lem:deform_geom_real}
  Let the geometric representation $\Dbar(\pi_w\cup \pi_b)$ be as
  above. Then there is a pseudo-isotopy $H$ of $\Dbar$ rel.\ $\partial
  \Dbar \cup \{0\}$ satisfying the following.
  \begin{itemize}
  \item $H$ deforms $\Dbar(\pi_w\cup\pi_b)$ to sectors. More precisely
    \begin{align*}
      H_1(\Dbar_w)&=\bigcup_{j \text{ even}} S_j, 
      &&
      H_1(\Dbar_b)= \bigcup_{j \text{ odd}} S_j.
    \end{align*}
  \item The pseudo-isotopy $H$ ``freezes'' outside of a neighborhood
    of $0$.
    By this we mean that for $\epsilon< 1/2$ 
    \begin{equation*}
      H \colon \Dbar \times [1-\epsilon,1]\to \Dbar \text{ is a
        pseudo-isotopy 
        rel.\ } \Dbar\setminus B_{\epsilon},
    \end{equation*}
    where $B_{\epsilon}=\{\abs{z}< \epsilon\}$.
  \item Only one point on each arc $g_m$ is deformed to $0$ by $H$.
  \end{itemize}
\end{lemma}

\begin{proof}
  This follows from the Sch\"{o}nflies Theorem \ref{thm:Schoenflies}.  
\end{proof}

\subsection{Connections}
\label{sec:connections-1}

Let $v$ be a $1$-vertex. A \defn{connection} at $v$
consists of an assignment which black/white $1$-tiles are connected at
$v$. The objective is to ``cut'' tiles at vertices, so that the boundary
of the ``white (or black) component'' is a Jordan curve. 

\medskip
Let $n=\deg_F v$ be the degree of $F$ at $v$,
let $X_0,\dots ,X_{2n-1}$ be the $1$-tiles containing $v$, labeled
mathematically positively around $v$, such that white $1$-tiles have even
index and black $1$-tiles have odd index.   

\begin{definition}[Connection at a vertex]
  \label{def:connection}
  A \defn{connection} at a $1$-vertex $v$ consists of a labeling of
  $1$-tiles 
  containing $v$ as above and cnc-partitions
  $\pi_w=\pi_w(v),\pi_b=\pi_b(v)$ of $\even_n$ (representing white
  $1$-tiles) and $\odd_n$ (representing black $1$-tiles). The
  $1$-tiles $X_i,X_j$ 
  (of the same color) are 
  said to be \defn{connected} at $v$ if $i,j$ are contained in the same block
  of $\pi_w\cup\pi_b$, $1$-tiles of different
  color are never connected. The $1$-tile $X_i$ is \defn{incident} (at
  $v$) to
  the block $b\in \pi_w\cup \pi_b$ containing $i$. 
  By Lemma \ref{lem:comppart} it is enough to define $\pi_w(v)$, then
  $\pi_b(v)$ will always be the complementary partition.

  If $v=p$ is a postcritical point the connection at $p$ is
  \emph{marked} in addition (see Definition
  \ref{def:mark_partition}). 
  Recall that the \emph{marked arc} of a geometric representation
  $\Dbar(\pi_w\cup \pi_b)$ (of the 
  connection at the postcritical point $p$, Definition
  \ref{def:geom_representation}) is assumed to \emph{contain the
    origin}.  
\end{definition}
The connection illustrated in Figure \ref{fig:connection} is given by
$\pi_w=\{\{0,2,6\}, \{4\}\}$, $\pi_b=\{\{1\}, \{3,5\}, \{7\}\}$. The
marked arc is indicated by the dot.

\smallskip
When talking about $1$-tiles $X_j$ and cnc-partitions at the same
time, it is always assumed without mention that the indices of the
$X_j$ are as above. 

\medskip
Let $v$ be a $1$-vertex, and $n=\deg_v F$. Let $X_0,\dots
,X_{2n-1}$ be the $1$-tiles containing $v$, labeled positively around
$v$ (white tiles have even index, black ones odd index as before). 
Every such $1$-vertex $v$ has arbitrarily small neighborhoods $U=U(v)$, that
are 
closed and homeomorphic to the closed disk $\Dbar$, such that
there is a homeomorphism  
\begin{equation}
  \label{eq:defhv}
  h=h_v\colon U\to \Dbar,
\end{equation}
that maps tiles to sectors (see (\ref{eq:def_sector})), 
\begin{equation*}
  h(X_j\cap U)= S_j,
\end{equation*}
for $j=0,\dots, 2n-1$. In particular $h(v)=0$. 
We require that the neighborhoods $U(v), U(v')$ have disjoint closures 
for distinct $1$-vertices $v,v'$. The reader should think of the
neighborhood $U$ as a ``blowup'' of the point $v$. 

\begin{definition}[Geometric representation of a connection]
  \label{def:conn_geom_repres}
  Let a connection at $v$ be given, with cnc-partition
  $\pi_w\cup\pi_b$, geometrically represented by 
  $\Dbar(\pi_w\cup\pi_b)$ as in Definition
  \ref{def:geom_representation};
  and $h$, $U=U(v)$ be as above. 
  A \defn{geometric representation of
    the connection at $v$} is given by replacing $U$ by
  $h^{-1}\left(\Dbar(\pi_w\cup\pi_b)\right)$.   

  More precisely, the white $1$-tiles in $U$, $\left(X_0\cup X_2\cup
    \dots X_{2n-2}\right) \cap U$ are replaced by $h^{-1}(\Dbar_w)$
  (see Definition \ref{def:geom_representation}). Note that this set
  is colored white.
  Similarly we replace the black
  $1$-tiles in $U$, $\left(X_1\cup X_3\cup
    \dots X_{2n-1}\right) \cap U$ by $h^{-1}(\Dbar_b)$. This set is 
  colored black.

  Let $v=p$ be a postcritical point and the connection at $p$ be
  marked by the arc $g_m$. More precisely, in the geometric
  representation $\overline{\D}(\pi_w\cup \pi_b)$ of the connection
  $\pi_w\cup \pi_b$ at $p$, the marking corresponds to the arc $g_m
  \subset \overline{\D}(\pi_w\cup \pi_b)$. Since the marked arc was
  chosen to contain 
  $0$, it follows that in this case $p\in h^{-1}(g_m)$, thus
  \emph{the geometric representation of the marked arc contains
    $p$}. This is the purpose of the marking, namely to keep track of
  where in the geometric representation of the connection the
  postcritical point is located. 
\end{definition}

\begin{definition}[Connection]
  \label{def:connection_tiles}
  A \defn{connection of $1$-tiles} is an assignment of a connection
  at every $1$-vertex.  
  Representing the connection at
  each $1$-vertex geometrically as above gives a \defn{geometric
    representation} of this connection of $1$-tiles. Objects arising from
  a geometric representation will be denoted with an $\epsilon$-subscript. 
\end{definition}

Assume a geometric representation of a connection of $1$-tiles is
given. From the construction it follows that each boundary component
of some black/white component is a Jordan curve. 
Let $X$ be a $1$-tile with $1$-vertices $v_0, \dots,
v_{k-1}$. Then the geometric representation of $X$ is $X_\epsilon:=
X\setminus \bigcup_j U(v_j)$, where the neighborhood $U(v_j)$ of $v_j$
is as in (\ref{eq:defhv}).  
Note that by construction two $1$-tiles $X,Y$ (of the same color) are
connected at a $1$-vertex $v$ if and only if their geometric
representations $X_\epsilon,Y_\epsilon$ are connected in $U(v)$. This
means $X_\epsilon,Y_\epsilon$ can be joined by a path in $U(v)$ that
does not intersect any boundary of some black or white
component. 

\subsection{The connection graph}
\label{sec:tile-graphs}

Given a connection of $1$-tiles we construct the
\defn{white (black) connection graph}. 



\begin{definition}[Connection graph]
  \label{def:connection_graph}
  The \defn{white connection graph} is constructed as follows. 
  For each white $1$-tile $X$ there is a vertex $c(X)$ (thought of as
  the \emph{center} of the $1$-tile $X$). For each $1$-vertex $v$ and
  block $b\in \pi_w(v)$ there is a vertex $c(v,b)$. The vertex $c(X)$
  is connected to $c(v,b)$ by an edge if and only if $X$ is incident
  to $b$ at $v$.    

  
  \smallskip
  The \defn{black connection graph} is constructed in the same manner
  from black $1$-tiles and their connections.  
\end{definition}

We will identify a $1$-tile $X$ with (the vertex of the white
connection graph) $c(X)$. 
For example we will say that two white $1$-tiles $X,Y$ are
connected (given a connection of $1$-tiles) if $c(X)$ and $c(Y)$ lie in
the same component of the white connection graph. 

\begin{definition}[Cluster]
  \label{def:cluster}
  A white/black
  \defn{cluster} $K$ is one component of the white/black connection
  graph. Using the previous identification we say that $K$ contains a
  $1$-tile $X$ (and write $X\subset K$), if $c(X)\in K$. This means we
  identify 
  $K$ with the union of $1$-tiles ``contained'' in it. Similarly a
  $1$-edge $E$, $1$-vertex $v$ is said to be contained in $K$ if $E\subset
  X\subset K$, $v\in X\subset K$ (for some $1$-tile $X$)
  respectively. Each $1$-tile is contained in exactly one cluster (of
  the same color), each $1$-edge is contained in exactly two clusters
  (one black and one white). A $1$-vertex $v$ may be contained in several
  clusters (in fact at most $n+1$, where $n=\deg_F v$).   

  \smallskip
  Assume a geometric representation of the connection has been
  given. Let $X$ be a $1$-tile contained in the cluster $K$. Then
  there is a unique component $K_\epsilon$ (of the same color as $X$)
  containing (the geometric representation) $X_\epsilon$. Recall
  that some $1$-tile $Y$ is connected to $X$ at a $1$-vertex $v$ if and
  only if they are connected at $v$ in a geometric representation of
  the connection. Thus one obtains inductively that any $1$-tile $Z$ is
  contained in $K$ if and only if $Z_\epsilon\subset K_\epsilon$.  
  Thus each white/black cluster $K$ corresponds to one white/black
  component $K_\epsilon$ (of a geometric representation of the
  connection) and vice versa. We call $K_\epsilon$ a
  \defn{geometric representation} of the cluster $K$. 

  \smallskip
  A cluster $K$ is a \defn{tree}
  if the underlying component of the connection
  graph is a tree, i.e., contains no cycles. The white cluster $K$ is
  a \defn{spanning tree}, if it is a tree and contains all white
  $1$-tiles.    
\end{definition}

\medskip
In the next section 
the connection of $1$-tiles will be constructed such that the white
$1$-tiles form a spanning tree in ``the right homotopy class''. 


\begin{remark}  Assume all white $1$-tiles are connected at each
  $1$-vertex. Of course we can extract a spanning tree (in the standard
  sense) from the resulting
  white connection graph. This spanning tree however will have only
  one vertex for each $1$-vertex $v$. Thus not all spanning trees in the
  sense of the previous definition can be obtained in this way. See
  Corollary \ref{cor:spanning_tree_ind} for an inductive way to
  construct trees in the connection graph. 
\end{remark}
 
The first approximation of the Peano curve $\gamma^1$ will be
constructed as ``the outline'' of the spanning tree. One should think
of the construction as follows. 
A geometric representation of this (white) spanning tree will be a Jordan
domain. 
The positively oriented boundary
of this domain ``is'' the first approximation $\gamma^1$. 

\subsection{Succeeding edges}
\label{sec:succeeding-edges}

Let a connection of $1$-tiles be given. Let $E$ be a $1$-edge
contained in the white $1$-tile $X_i$, positively oriented (as
boundary of $X_i$) with terminal point $v$. 
  
Since $1$-tiles are cyclically
ordered around $v$, the $1$-tiles that are connected at $v$
with 
$X_i$ are cyclically ordered as well. 

Let $X_j$ be the cyclical successor (in mathematically
positive order around $v$) of $X_i$ among $1$-tiles connected to
$X_i$ at $v$. If no other $1$-tile is connected to $X_i$ at $v$,
we let $X_j=X_i$. 

Formally $i,j$ are contained in the same block of $\pi_w$, and none of the
numbers in $[i+1, j-1]$ are contained in this block.

Note that $X_j$ is a \emph{white} $1$-tile. Thus an
oriented $1$-edge $E'\subset X_j$ is positively oriented if and only if it
is positively oriented as boundary of $X_j$.

\begin{definition}[Successor]
  \label{def:successor}
  Let $v,E$ as well as $X_i,X_j$ be as above.
  The \defn{successor} to $E$ (at $v$) is the positively oriented
  $1$-edge 
  $E'\subset X_j$ with initial point $v$. Note that each $1$-edge
  $E'$ is the successor to exactly one $1$-edge $E$. 
\end{definition}

See Figure \ref{fig:connection} for an illustration.
For each $1$-edge $E$ with initial/terminal point $v,w$, let
$E_\epsilon:= E\setminus (U(v)\cup U(w))$. Here $U(v),U(,w)$ are the
neighborhoods of $v,w$ from (\ref{eq:defhv}).  
Recall from Lemma \ref{lem:geom_real_comp_part} how a cnc-partition
was geometrically represented by dividing the disk by arcs $g_m$. We
call 
such an arc $g_m$ \defn{positively oriented} if it is positively
oriented 
as boundary arc of a \emph{white} set $D_l$.  

\begin{lemma}[Equivalent formulations for succeeding edges]
  \label{lem:arcs_successors}
  Consider white $1$-tiles $X_i\supset E$, $X_j\supset E'$, where $E,E'$
  are positively oriented $1$-edges containing a $1$-vertex $v$.
  The following are equivalent.
  \begin{itemize}
  \item $E'$ is the successor to $E$ at $v$.
  \item $E'_\epsilon$ is succeeding $E_\epsilon$ on $\partial
    K_\epsilon$, where $K_\epsilon$ is a geometric representation of
    the white cluster $K$ containing $E$. This means that when
    $\partial K_\epsilon$ is 
    positively oriented (as boundary of $K_\epsilon$) there is no
    (geometric representation of a $1$-edge $\widetilde{E}$) 
    $\widetilde{E}_\epsilon\subset \partial K_\epsilon$ on the
    positively oriented arc from 
    $E_\epsilon$ to $E'_\epsilon$. 
  \item Represent the connection at $v$ geometrically as in
    Lemma~\ref{lem:geom_real_comp_part}. Using the notation from this
    lemma, there is a (positively oriented) arc $g_m$ that
    connects the right endpoint of the arc $a_i\subset S^1$ to the
    left endpoint of the arc $a_j\subset S^1$.
  \item There are adjacent blocks $b\in \pi_w(v), c\in \pi_b(v)$ such
    that
    \begin{equation*}
      i,j\in b,\quad i+1,j-1\in c. 
    \end{equation*}
  \end{itemize}
\end{lemma}
The proof is clear from the proof of Lemma
\ref{lem:geom_real_comp_part}.  

\begin{cor}[Marked connection]
  \label{cor:marking}
  A marking of a connection at a postcritical point $p$ may be given 
  \begin{itemize}
  \item by \emph{marking an arc} $g_m$ from a geometric representation
    of the connection at $p$. 
  \item or equivalently by \emph{marking} a pair of \emph{succeeding
      $1$-edges} $E,E'$ at $p$;
  \item or equivalently by \emph{marking} a pair of \emph{adjacent
      blocks} $b\in \pi_w(p)$, $c\in \pi_b(p)$. 
  \end{itemize}
\end{cor}
The precise correspondences (i.e., which marked arc corresponds to
which marked pair of succeeding edges, corresponds to which marked
pair of adjacent blocks) is given by Lemma~\ref{lem:arcs_successors}. 



The $1$-tiles containing successors $E,E'$ are connected at $v$. If on
the other 
hand $1$-tiles $X,Y$ are connected at $v$, we can find a chain of
succeeding $1$-edges.

\begin{lemma}
  \label{lem:prop_successor}
  Two $1$-tiles $X,Y$ (of the same color) are connected at the
  $1$-vertex $v$ if and only if there is a chain
  \begin{equation*}
    X=X_1,E_1,E'_2,X_2, \dots, X_{m-1},E_{m-1}, E'_m,X_m=Y.
  \end{equation*}
  Here $X_j\ni v$ are $1$-tiles of the same color as $X,Y$; $E_j,E'_j\subset
  X_j$ are $1$-edges, and $E'_{j+1}$ succeeds $E_j$ at $v$. 
\end{lemma}

Note that in the above, the labelling of the white $1$-tiles is
\emph{not} the one used in the definition of the connection at $v$
(there are some white $1$-tiles with odd index). 

\begin{proof}
  If the $1$-tiles in the lemma are white, the cyclical order of
  $1$-tiles connected to $X$ at $v$ from $X=X_1$ to $Y=X_m$ is given
  by $X_1,\dots, X_m$. If the $1$-tiles are black this gives the
  anti-cyclical order. Clearly going (anti-)cyclically around $v$
  among $1$-tiles connected to $X$ gives all such $1$-tiles. 
\end{proof}

\subsection{Adding clusters}
\label{sec:adding-clusters}

The spanning tree will be built successively by adding more ``secondary
clusters'' 
to a ``main cluster''. 

\medskip
Let the connection at a $1$-vertex $v$ be given by the
cnc-partition $\pi_w\cup \pi_b$ (of $[2n]$, where $n=\deg_F(v)$) and  
$K,K'$ be two white clusters containing $v$. Let $b\in \pi_w$ be a
block with indices of $1$-tiles in $K$ ($j\in b \Rightarrow X_j\subset
K$), $b'\in \pi_w$ a block with indices of $1$-tiles in $K'$. We
\defn{add} the cluster $K'$ to $K$ at $v$ by replacing $b,b'$ in
$\pi_w$ by $\tilde{b}:=b\cup b'$. 
The resulting partition $\widetilde{\pi}_w$ however may not be
non-crossing anymore. 

\begin{lemma}[Adding clusters]
  \label{lem:add_cluster}
  The partition $\widetilde{\pi}_w$ is non-crossing if and only if
  there is a block $c\in \pi_b$ that is adjacent to both $b$ and $b'$
  (see Lemma \ref{lem:prop_comp}).  

  In this case,
  let $\widetilde{K}$ be the cluster in the new connection graph that
  contains $K,K'$. 
  If $K,K'$ are trees then $\widetilde{K}$ is a tree as well. 
\end{lemma}

The situation is illustrated in Figure~\ref{fig:add_cluster}. 

\begin{proof}



  We show the equivalence first.

  \medskip
  ($\Leftarrow$)
  Assume $\widetilde{\pi}_w$ is crossing. Then there is a block
  $\hat{b}\in \pi_w$, such that there are 
  \begin{align*}
    a,a'&\in \hat{b}, \,
    d\in b,\, d'\in b' 
    \quad  \text{satisfying}
    \\
    &a<d<a'< d'.
  \end{align*}
  This means that $b,b'$ have to be contained in different components
  of 
  $[2n]\setminus \{a,a'\}$. Thus every block $c\in \pi_b$ adjacent to
  $b$ has to be in a different component of $[2n]\setminus\{a,a'\}$
  than every block $c'\in\pi_b$ adjacent to $b'$. Thus there is no
  block $c\in \pi_b$ adjacent to both $b,b'$. 

  \medskip
  ($\Rightarrow$)
  Assume now that there is no $c\in \pi_b$ adjacent to both $b,b'$. 
  Let $b=\{b_1,\dots,b_N\},
  b'=\{b'_1,\dots, b'_M\}$, where $b_1<\dots < b_N$, and $b'_1 <
  \dots < b'_M$. Since $\pi_w$ is non-crossing $b,b'$ are in
  disjoint intervals, meaning we can assume that for some $j$
  \begin{equation*}
    b_j< b'_1< b'_M< b_{j+1}. 
  \end{equation*}
  Since $\pi_b$ is complementary to $\pi_w$ there are blocks $c,c'\in
  \pi_b$ such that
  \begin{equation*}
    b_j+1,b_{j+1}-1 \in c,  \quad  b'_1-1 , b'_M+1\in c',
  \end{equation*}
  by Lemma \ref{lem:prop_comp}. The blocks $c,c'$ are distinct by
  assumption. Let 
  $c'_1:=\min\{c'_j\in c'\}$, $c'_2:= \max\{c'_j\in c'\}$. The numbers
  $c'_1-1,c'_2+1$ are in the same block $\hat{b}\in\pi_w$ (since
  $\pi_w,\pi_b$ are complementary). Thus we have the following
  ordering
  \begin{equation*}
    \underbrace{b_j}_{\in b}
    < \underbrace{b_j+1}_{\in c}
    < \underbrace{c'_1-1}_{\in \hat{b}}
    < \underbrace{c'_1}_{\in c'}
    < \underbrace{b'_1<b'_M}_{\in b'} 
    < \underbrace{c'_2}_{\in c'}
    <
    \underbrace{c'_2+1}_{\in \hat{b}} < \underbrace{b_{j+1}-1}_{\in
      c}< \underbrace{b_{j+1}}_{\in b}.   
  \end{equation*}
  Clearly $b\cup b'$ and $\hat{b}$ are crossing, which finishes this
  implication.





  \medskip
  We now show the second statement. Recall that in the white
  connection graph the block $b\in \pi_w$ is represented by a vertex
  $c(v,b)$
  and $b'\in \pi_w$ is represented by a
  (different) vertex $c(v,b')$. 
  The new white
  connection graph (where the connection at $v$ is given by
  $\widetilde{\pi}_w$) 
  is obtained by identifying $c(v,b)$ and $c(v,b')$; this yields
  the vertex $c(v,\tilde{b})$.  
  Then $\widetilde{K}$ is
  the component (of the new white connection graph) containing
  $c(v,\tilde{b})$. If $K,K'$ are trees, then clearly $\widetilde{K}$
  is a tree as well. 
\end{proof}

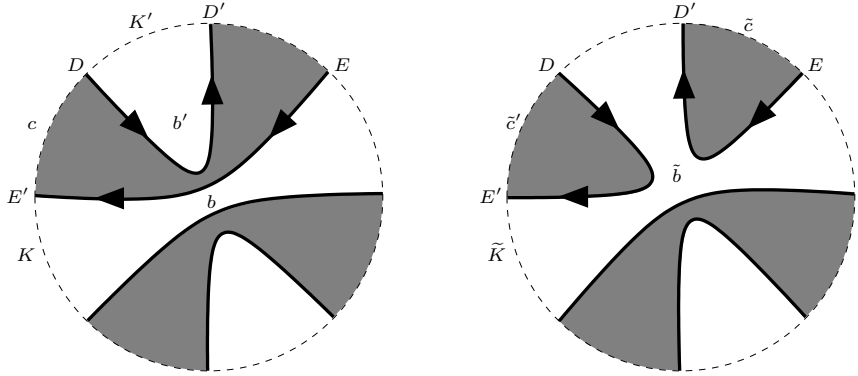
\begin{figure}
  \centering
  \includegraphics[width=11cm]{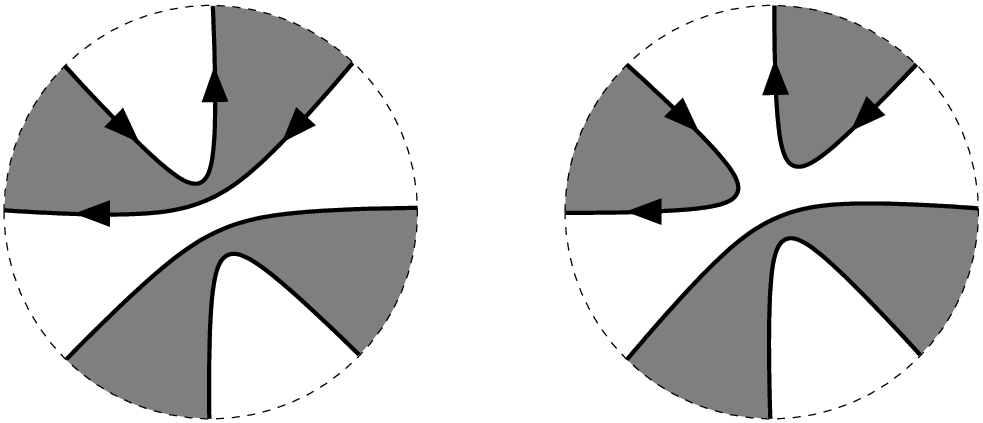}
  \begin{picture}(10,10)
    %
    \put(-250,63){$\scriptstyle{b}$}
    \put(-263,93){$\scriptstyle{b'}$}
    \put(-318,93){$\scriptstyle{c}$}
    \put(-322,43){$\scriptstyle{K}$}
    \put(-280,131){$\scriptstyle{K'}$}
    \put(-202,115){$\scriptstyle{E}$}
    \put(-326,65){$\scriptstyle{E'}$}
    \put(-303,115){$\scriptstyle{D}$}
    \put(-252,135){$\scriptstyle{D'}$}
    %
    \put(-74,73){$\scriptstyle{\tilde{b}}$}
    \put(-144,43){$\scriptstyle{\widetilde{K}}$}
    \put(-137,93){$\scriptstyle{\tilde{c}'}$}
    \put(-23,115){$\scriptstyle{E}$}
    \put(-147,65){$\scriptstyle{E'}$}
    \put(-125,115){$\scriptstyle{D}$}
    \put(-74,135){$\scriptstyle{D'}$} 
    \put(-47,130){$\scriptstyle{\tilde{c}}$} 
  \end{picture}
  \caption{Adding clusters.}
  \label{fig:add_cluster}
\end{figure}

Assume that $c$ is adjacent to both $b,b'$, i.e., that we can add $K'$
to $K$ at $v$ in this fashion.
Let the notation be as in the previous proof, i.e.,   $b=\{b_1,\dots,b_N\},\,
  b'=\{b'_1,\dots, b'_M\}$, where
\begin{equation}
  \label{eq:defbbprime}
  b_1< \dots < b_N, \; b_1'< \dots < b'_M\; \text{ and } \;  b_j< b'_1< b'_M< b_{j+1}. 
\end{equation}
Then the complementary
partition $\widetilde{\pi}_b$ to $\widetilde{\pi}_w$ is given by
replacing $c\in \pi_b$ by 
the two blocks 
\begin{equation}
  \label{eq:def_tildec}
  \tilde{c}=c\cap [b_j,b_1'], \quad
  \tilde{c}'=c\cap [b'_M,b_{j+1}].   
\end{equation}
These
two blocks are both adjacent to $\tilde{b}=b\cup b'\in
\widetilde{\pi}_b$.   

\medskip
If we add a cluster $K'$ to a cluster $K$ as above at a postcritical
point $p$, we need to specify the marking (see Definition
\ref{def:mark_partition})  of the new connection at $p$. 

\begin{definition}[Marking of new connection]
  \label{def:add_marking}
  Let $\pi_w\cup\pi_b$ be a marked cnc-partition, i.e., a connection
  at a postcritical point $p$. Then the marking of the cnc-partition
  $\widetilde{\pi}_w\cup\widetilde{\pi}_b$ from the previous lemma is
  given as follows (notation is as before). 
  Let the marked adjacent blocks in $\pi_w\cup \pi_b$ be 
  \begin{itemize}
  \item $b,c$, or $b',c$; 
    \newline
    then (in both cases) we can pick $\tilde{b},\tilde{c}$ or
    $\tilde{b},\tilde{c}'$ as the marked adjacent blocks in
    $\widetilde{\pi}_w\cup\widetilde{\pi}_b$.
  \item $d,c$, where $d\in \pi_w\setminus\{b,b'\}$;
    \newline
    then $d$ is adjacent to either $\tilde{c}$ or $\tilde{c}'$, which  
    are the  
    marked adjacent blocks in
    $\widetilde{\pi}_w\cup\widetilde{\pi}_b$. 
  \item $b,e$ or $b',e$, where $e\in \pi_b\setminus\{c\}$;
    \newline
    then $\tilde{b},e$ are the marked adjacent blocks in
    $\widetilde{\pi}_w\cup\widetilde{\pi}_b$. 
  \item $d,e$, where $d\in \pi_c\setminus\{b,b'\}$, $e\in
    \pi_b\setminus \{c\}$; 
    then $d,e$ are the marked adjacent blocks in
    $\widetilde{\pi}_w\cup\widetilde{\pi}_b$. 
  \end{itemize}
\end{definition}

\begin{lemma}
  \label{lem:add_successor}
  Assume a white cluster $K'$ can be added to a white cluster $K$ at a
  $1$-vertex 
  $v$ as in Lemma \ref{lem:add_cluster} 
  to form a cluster $\widetilde{K}$. Then
  there exist (uniquely) succeeding $1$-edges at $v$
  \begin{equation*}
    E,E'\subset K \quad \text{as well as }\;\;  D,D'\subset K',
  \end{equation*}
  such that
  \begin{equation*}
    E,D' \quad \text{as well as }\;\; D,E'
  \end{equation*}
  are succeeding in $\widetilde{K}$.
\end{lemma}
The situation is again illustrated in Figure~\ref{fig:add_cluster}.

\begin{proof}
  Consider the blocks $b,b'\in \pi_w(v)$ which are both adjacent 
  to the block $c\in \pi_b(v)$ as in Lemma \ref{lem:add_cluster} (here
  $b$ contains indices of $1$-tiles in $K$, $b'$ contains indices of
  $1$-tiles in $K'$). The succeeding
  $1$-edges $E,E'\subset K$, and $D,D'\subset K'$, are the ones
  corresponding to these adjacencies according to Lemma
  \ref{lem:arcs_successors}.    
  Using the notation from (\ref{eq:defbbprime}), we
  obtain that these $1$-edges are
  contained in the following (white) $1$-tiles. In $K,K'$
  \begin{align*}
    & E\subset X_{b_j}
    && E' \subset X_{b_{j+1}}
    \\
    & D\subset X_{b'_M}
    && D'\subset X_{b'_1}.
  \end{align*}
  Recall the description of the blocks $\tilde{c},\tilde{c}'\in
  \widetilde{\pi}_b$ from (\ref{eq:def_tildec}). They  
  are both adjacent to $\tilde{b}=b\cup b'\in \widetilde{\pi}_w$.
  Then $b_j+1,b'_1-1\in \tilde{c}$, $b'_M+1,
  b_{j+1}-1\in \tilde{c}'$. Thus (using Lemma
  \ref{lem:arcs_successors} again) we obtain that $E,D'$ and $D,E'$ are
  succeeding in $\widetilde{K}$.
\end{proof}

We will often be in the following specific situation. Consider a  
white cluster $K$. Assume that the only white $1$-tiles that are
possibly connected at a $1$-vertex $v$ are in $K$. Put differently, this means
that all distinct white $1$-tiles $Y,Y'\ni v$ not in $K$ are
disconnected at $v$. Let $X_{i}\ni v$ be a white $1$-tile not contained in
$K$. The following lemma means that we can \defn{add} $X_i$, or the
cluster containing $X_i$, to $K$ at $v$.    

\begin{lemma}
  \label{lem:change_conn}
  In the situation as above, 
  there is a block $b\in\pi_w$ containing indices of white $1$-tiles
  in $K$ ($j\in b\Rightarrow X_j\subset K$), such that the partition
  $\widetilde{\pi}_w$ obtained by replacing $b,\{i\}\in \pi_w$ by
  $\tilde{b}=b\cup \{i\}$ is non-crossing. 

  Furthermore if $K$ and the cluster containing $X_i$ are trees, the
  resulting cluster $\widetilde{K}$ ($\supset K\cup X$) is a tree as well.
\end{lemma}

\begin{proof}
  Consider the graph $\Gamma$ representing
  $\pi_w\cup\pi_b$ 
  from Lemma \ref{lem:prop_comp} (this is \emph{neither} the white
  connection graph \emph{nor} the graph $\bigcup \E^1$).

  Let $X_j\ni v$ be a white $1$-tile not contained
  in $K$. Since $X_j$ is not connected to any other $1$-tile at $v$
  the singleton $\{j\}$ is a block of $\pi_w$. This block is adjacent
  to a single block (in $\pi_b$), thus $\{j\}$ is a leaf of $\Gamma$
  (incident to a single edge). 

  \smallskip
  Consider the block $c\in \pi_b$ adjacent to $\{i\}\in \pi_w$. 
  Since $\Gamma$ is connected, $c$ has to be connected to a
  block $b\in\pi_w$ containing indices corresponding to $1$-tiles in
  $K$. This means that $b,c$ are adjacent blocks. 
  The result now follows from Lemma \ref{lem:add_cluster}. 
\end{proof}



 



We record the following corollary (see also Lemma
\ref{lem:whiteXconn}). 
\begin{cor}[Trees in connection graphs]
  \label{cor:spanning_tree_ind}
  A (cluster that is a) tree in the white (black) connection graph may
  be constructed 
  inductively by adding one $1$-tile to a cluster at a time. Every
  tree in the white (black) connection graph (in a cluster) is
  obtained in such a way. 
\end{cor}

\subsection{Boundary circuits}
\label{sec:boundary-circuits}
The first approximation of the Peano curve $\gamma^1$ will be given as
the \defn{boundary circuit} of a (cluster that is a) spanning tree (in
the white connection graph). 

\begin{definition}[Boundary circuit of a cluster]
  \label{def:gamma1_connection}
  Consider a cluster $K$. A \defn{boundary circuit} $\EC$ of $K$ is a 
  circuit of positively oriented $1$-edges in $K$ 
  \begin{equation*}
    E_0,\dots, E_{M-1},
  \end{equation*}
  such that $E_{j+1}$ is the successor of $E_j$ for each $j$ (indices
  are taken $\bmod M$, in particular $E_0$ succeeds $E_{M-1}$);
  furthermore no $1$-edge appears twice in $\EC$.  




    
\end{definition}

Recall that every $1$-edge has exactly one successor and one
predecessor. Thus it is clear that starting from any $1$-edge
$E_0\subset K$ and following succeeding $1$-edges will yield a
boundary circuit.

\medskip
We note the following, which is an immediate consequence of Lemma
\ref{lem:arcs_successors} and Corollary \ref{cor:marking}, see also
the discussion after Definition \ref{def:conn_geom_repres}. 

\begin{lemma}[$K_\epsilon$ contains $p$]
  \label{lem:p_on_boundary}
  Let $K$ be a cluster, $p$ a postcritical point. A boundary
  circuit of $K$ contains the marked succeeding $1$-edges at $p$ if
  and only if $p \in K_\epsilon$ for any geometric representation
  $K_\epsilon$ of $K$.  
\end{lemma}

\begin{lemma}
  \label{lem:cluster_sc}
  Consider a cluster $K$.
  The following are equivalent.
  \begin{enumerate}
  \item 
    \label{item:cluster_sc1}
    The cluster $K$ is a tree. 
  \item 
    \label{item:cluster_sc2}
    $K$ has only a single boundary circuit.
  \item 
    \label{item:cluster_sc3}
    Each geometric representation $K{_\epsilon}$ of $K$ is a
    Jordan domain. 
  \end{enumerate}
  In this case the single boundary circuit $\EC$ of $K$ is an Eulerian
  circuit in $K$. This means each of the $km$ $1$-edges in $K$ appears 
  exactly once in $\EC$. Here $m$ is the number of $1$-tiles in $K$
  ($k=\#\post=\# 0$-edges).  
\end{lemma}

\begin{proof}
  Assume without loss of generality that the cluster $K$ is white.

  $(\ref{item:cluster_sc1}) \Rightarrow (\ref{item:cluster_sc2})$
  Recall from Corollary \ref{cor:spanning_tree_ind} that every
  tree can be obtained inductively by adding more $1$-tiles
  to one cluster in the connection graph. Start with a white tile
  graph that is totally disconnected, meaning no two white $1$-tiles
  are connected (at any $1$-vertex). Consider one white $1$-tile $X_0$ 
  and a $1$-edge $E_0\subset X_0$. Clearly $E_0$ is contained in an
  Eulerian circuit in $X_0$ of length $k$ (containing
  all $1$-edges in $\partial X_0$). 
  
  \smallskip
  Let the white connection graph be given such that all clusters except
  one cluster $K_{j-1}$ contain a single $1$-tile, i.e., as in Lemma
  \ref{lem:change_conn}. Assume $E_0\subset K_{j-1}$. Furthermore we assume
  that $E_0, \dots E_{kj-1}$ is an Eulerian circuit in $K_{j-1}$,
  containing all $1$-edges in $K_{j-1}$, where $j$ is the number of
  $1$-tiles in $K_{j-1}$. 

  Add a $1$-tile $X$ to $K_{j-1}$ at a $1$-vertex $v\in K_{j-1}$ as in
  Lemma 
  \ref{lem:change_conn} to form a new component $K_j$. The above
  procedure then yields as a path 
  $$E_0,\dots, E_i, E^X_1,\dots, E^X_k,E_{i+1},\dots ,E_{kj-1},$$ 
  see Lemma \ref{lem:add_successor}.
  Here $E^X_1,\dots E^X_k$ are the $1$-edges in $X$, positively
  oriented, starting at $v$. 

  This is an Eulerian circuit in $K_j$. The
  construction ends 
  when $K=K_j$. Since the constructed circuit contains all $1$-edges
  in $K$ there is only a single boundary circuit.

  \medskip
  $(\ref{item:cluster_sc2})\Rightarrow (\ref{item:cluster_sc3})$ 
  Consider a neighborhood $U$ of a $1$-vertex $v\in K$ as in
  Definition \ref{def:conn_geom_repres}.
  The boundary of $K_{\epsilon}$ is constructed from boundary circuits
  by replacing $E_j,E_{j+1}\cap U$ by $h^{-1}(g_m)$. Thus $\partial
  K_\epsilon$ is a single Jordan curve.


  \medskip
  $(\ref{item:cluster_sc3}) \Rightarrow (\ref{item:cluster_sc1})$
  Assume $K$ is not a tree. Then there exists a circuit in $K$. This
  means there are $1$-tiles $X_0, \dots, X_{N-1}$
  in $K$ such that $X_j$ is connected to $X_{j+1}$ at a
  $1$-vertex $v_j$ (indices $\bmod N$), where all $1$-vertices $v_j$ are
  distinct. Then in the interior of any geometric representation
  $K_\epsilon$ we can 
  find a Jordan curve following this circuit (connecting
  $X_{0,\epsilon}$ to $X_{1,\epsilon}$ at $v_{0,\epsilon}$ and so
  on). This Jordan curve divides $K_\epsilon$ into two
  components. Note that both components contain boundary of
  $K_\epsilon$, namely the (geometric representations of the) two arcs
  on $\partial X_j$ between $v_{j-1},v_j$ lie in different components. Thus
  $K_\epsilon$ is not a Jordan domain.  

\end{proof}

  

  


We record the following, which is an easy corollary.
\begin{lemma}[Boundary circuit of added trees]
  \label{lem:add_tree_boundary}
  Consider trees $K,K'$ with boundary circuits
  $\EC = E_0,\dots,E_{N-1}$, $\EC'=D_0,\dots,D_{M-1}$. 
  Assume we
  can add them at a $1$-vertex $v$ as in Section
  \ref{sec:adding-clusters} to form a tree $\widetilde{K}$.
  Then the boundary
  circuit $\widetilde{\EC}$ of $\widetilde{K}$ is 
  \begin{equation*}
    E_0,\dots, E_i,
    D_{j+1},\dots,D_{M-1},D_0,\dots,D_j,E_{i+1},\dots,E_{N-1}.  
  \end{equation*}
\end{lemma}

\begin{proof}
  This is clear from Lemma \ref{lem:add_successor}, where
  $E_i,E_i+1\subset K$ and $D_j,D_{j+1}\subset K'$ are the succeeding
  $1$-edges associated with adding $K$ to $K'$.   
\end{proof}

We next show that adding a tree $K'$ that ``does not contain a
postcritical point'' to another tree $K$ does not change the
``homotopy type'' of $\partial K_{\epsilon}$. 
\begin{definition}[Trivial tree]
  \label{def:trivial_tree}
  A cluster $K'$ that is a tree is called \defn{trivial} if a (and
  thus any) geometric representation $K'_\epsilon$ does not contain a
  postcritical point. Equivalently the boundary circuit of $K'$ does
  not contain the \defn{marked successors} $E=E(p),E'=E'(p)$ at $p$
  for any postcritical point $p$ (see Corollary \ref{cor:marking}).     
\end{definition}

\begin{lemma}[Adding a trivial tree does not change homotopy type] 
  \label{lem:add_trivial_tree}
  Consider a cluster $K$ that is a tree, and a trivial tree $K'$ as
  above. Assume it is possible to add $K'$ to $K$ at some $1$-vertex
  $v$ as in Lemma \ref{lem:add_cluster}, to obtain the tree
  $\widetilde{K}$. 

  Then if $\partial K_\epsilon$ is isotopic to a Jordan
  curve $\CC$ rel.\ $\post$,
  then $\partial\widetilde{K}_\epsilon$ is isotopic to $\CC$
  rel.\ $\post$ as well  (for any geometric representations $K_\epsilon,
  \widetilde{K}_\epsilon$ of $K, \widetilde{K}$). 
\end{lemma}

\begin{proof}
  Let $U=U(v)$ be as in Definition \ref{def:conn_geom_repres}. We
  consider a neighborhood $V$ of ``$K'_\epsilon\subset
  \widetilde{K}_\epsilon$''. More precisely, $V$ satisfies the
  following. 
  \begin{itemize}
  \item $V$ is a Jordan domain.
  \item $V$ contains no postcritical point.
  \item $V$ is a neighborhood of $K'_\epsilon\setminus U$.
  \item $\partial V$ intersects $\partial \widetilde{K}_\epsilon$
    exactly twice, where $\partial V\cap\partial
    \widetilde{K}_\epsilon =\{w_1,w_2\} \subset U$.  
  \end{itemize}
  The arc $\partial \widetilde{K}_\epsilon\setminus \{w_1,w_2\}$
  contained in $V$ 
  is now deformed to one contained in $U$ by an 
  isotopy rel.\ $\partial V$ as in Theorem \ref{thm:Schoenflies}. This
  isotopy deforms $\widetilde{K}_\epsilon$ to $K_\epsilon$.  


\end{proof}

\section{Construction of $H^0$}
\label{sec:construction-h0}

The $0$-th pseudo-isotopy $H^0$ as required in Section
\ref{sec:appr-gn} is
constructed here, thus the first approximation $\gamma^1$ of the Peano
curve. 

\smallskip
Consider two oriented Jordan curves $\CC,\CC'\subset S^2$. We say that $\CC,\CC'$ are
\defn{orientation preserving isotopic rel.\ $A$} if there is an
isotopy $H\colon S^2\times [0,1] \to S^2$
rel.\ $A$, with $H_0=\id_{S^2}$, such that $H_1$ maps
$\CC$ orientation preserving to $\CC'$. 

\smallskip
We construct a connection of $1$-tiles 
with the following properties. 
\begin{definition}(Properties of connections)
  \label{def:prop_conn}
  \phantom{xblabl}

  \begin{enumerate}[(C 1)]
  \item 
    \label{item:prop_conn_1}
    The associated white connection graph (Section
    \ref{sec:tile-graphs}) is a spanning tree $K$. 
  \item 
    \label{item:prop_conn_2}
    The Jordan curve $\partial K_\epsilon$ is orientation preserving
    isotopic to $\CC=\gamma^0$ rel.\ 
    $\post$. Here $K_\epsilon$ is a geometric representation of $K$,
    see Lemma \ref{lem:cluster_sc}. 
  \end{enumerate}
\end{definition}
Here $\partial K_\epsilon$ is positively oriented as boundary of
$K_\epsilon$, recall that $\CC$ is positively oriented as boundary of
the white $0$-tile $X^0_w$. 

\begin{lemma}
  \label{lem:H0epsH0}
  A connection of $1$-tiles satisfies properties
  {\upshape (C \ref{item:prop_conn_1})}, {\upshape (C
    \ref{item:prop_conn_2})} if and only if there exists a
  pseudo-isotopy $H^0$ as in Definition \ref{def:pseudo-isotopy-h0}. 
\end{lemma}

\begin{proof}
  ($\Rightarrow$)
  Concatenate an isotopy $\widetilde{H}$ rel.\ $\post$
  that deforms $\CC$ to $\partial 
  K_\epsilon$ (orientation preserving) with a pseudo-isotopy rel.\
  $\post$ that deforms $\partial K_\epsilon$ in a
  neighborhood $U(v)$ (as in (\ref{eq:defhv})) of each $1$-vertex as in
  Lemma~\ref{lem:deform_geom_real}. This yields a pseudo-isotopy rel.\
  $\post$ that clearly satisfies ($H^0$~\ref{item:H0_1}),
  ($H^0$~\ref{item:H0_2}), ($H^0$~\ref{item:H0_3}), and
  ($H^0$~\ref{item:H0_4}). Since $\widetilde{H}_1$ maps $\CC$ orientation
  preserving to $\partial K_\epsilon$, it follows that every $1$-edge
  in the first 
  approximation $\gamma^1$ (constructed via $H^0$ as in Section
  \ref{sec:euler-circ-gamm}) is positively oriented. It follows from
  Lemma~\ref{lem:orientation-d-fold-cover} that
  ($H^0$~\ref{item:H0_5}) is satisfied. 

  \smallskip
  ($\Leftarrow$)
  Let $\gamma^1=H^0_1(\gamma^0)$ be the Eulerian circuit constructed
  from $H^0$ as in Section~\ref{sec:euler-circ-gamm}. By
  Lemma~\ref{lem:prop_successor} we can reconstruct the connection at
  each $1$-vertex from
  $\gamma^1$. It is a cnc-partition by Lemma
  \ref{lem:geom_real_comp_part}. Since $\gamma^1$ contains all
  $1$-edges, all white $1$-tiles are connected. Furthermore
  $\gamma^1_\epsilon:=H^0_{1-\epsilon}(\gamma^0)$ is a Jordan curve,
  thus it follows from Lemma \ref{lem:cluster_sc} that the white
  connection graph is a spanning tree, 
  i.e., {\upshape (C \ref{item:prop_conn_1})}. Finally
  $\gamma^1_\epsilon$ is clearly isotopic to $\gamma^0$ rel.\ $\post$,
  from ($H^0$~\ref{item:H0_5}) and
  Lemma~\ref{lem:orientation-d-fold-cover} it follows that the
  orientation on $\gamma^1_\epsilon$ induced by $\CC$ and
  $H^0_{1-\epsilon}$ agrees with the orientation of
  $\gamma^1_\epsilon$ as boundary of (a geometric representation of
  the white spanning tree) $K_\epsilon$. Thus {\upshape (C
    \ref{item:prop_conn_2})} holds.   
\end{proof}

Let us note the following immediate consequence.
\begin{theorem}
  \label{prop:sufficient_Peano}
  Let $F\colon S^2\to S^2$ be an expanding Thurston map.
  The following two equivalent conditions are \emph{sufficient} for
  the existence of an invariant Peano curve $\gamma\colon S^1\to S^2$
  (onto) as in Theorem~\ref{thm:main}. 
  \begin{enumerate}
  \item There is a Jordan curve $\CC\supset \post$ and a
    pseudo-isotopy $H^0$ in Definition~\ref{def:pseudo-isotopy-h0}.
  \item There is a Jordan curve $\CC\supset \post$ and a connection of
    $1$-tiles satisfying the properties from
    Definition~\ref{def:prop_conn}. 
  \end{enumerate}
\end{theorem}
In \cite{exp_quotients} it will be shown that the same conditions are
sufficient to ensure that $F$ arises as a mating. Furthermore the
poynomials $p_1, p_2$ into which $F$ unmates, may then be obtained by
an explcit algorithm. More precisely the \emph{critical portraits} of
$p_1, p_2$ may be obtained from the vector $l$ considered in
Section~\ref{sec:length-edges-s1}, see \cite{unmating}.

\medskip
The proof of Theorem~\ref{thm:main} will be finished by 
constructing the white connection as in Definition~\ref{def:prop_conn}.

\smallskip
Let us first note the following, which is
an immediate consequence of the proof of the previous lemma. Assume a
connection of $1$-tiles satisfying {\upshape (C
  \ref{item:prop_conn_1})}, {\upshape (C \ref{item:prop_conn_2})} is  
given. Let $H^0$ be a corresponding pseudo-isotopy from Lemma
\ref{lem:H0epsH0}. 

\begin{lemma}
  \label{lem:g1bdK}
  The first approximation $\gamma^1$ (viewed as an Eulerian circuit)
  constructed from $H^0$ as in Section \ref{sec:euler-circ-gamm} is
  equal to the boundary circuit of the (white) spanning tree $K$ 
  (see Lemma \ref{lem:cluster_sc}).   
\end{lemma}

The main work in constructing the connection as desired lies in
ensuring property (C~\ref{item:prop_conn_2}). 

The starting point is to take a sufficiently high iterate $F=f^n$ such
that there is an $F$-invariant Jordan curve $\CC\supset \post$ and $1$-tiles
defined in terms of $(F,\CC)$ (i.e., closures of components of
$S^2\setminus F^{-1}(\CC)$) are sufficiently small. We require two
separate conditions, since they are needed in distinct 
parts of the construction; they could be expressed as a
single one. In fact, the second condition is only given later,
when the suitable description becomes available.

\begin{lemma}
  \label{lem:exFC}
  For each sufficiently high $n\in \N$ there is a Jordan curve $\CC$
  with $\post \subset \CC$ satisfying the following.
  \begin{itemize}
  \item $\CC$ is \emph{invariant} for the iterate $F=f^n$. This
    means that $F(\CC)\subset \CC$.
  \end{itemize}
  The $1$-tiles for $(F,\CC)$ satisfy the following.
  \begin{itemize}
  \item There is no $1$-tile $X$ that \emph{joins opposite sides} of
    $\CC$. This means no $1$-tile $X$ meets disjoint $0$-edges in the
    case $\#\post\geq 4$, and no $1$-tile $X$ intersects all three
    $0$-edges in the case $\#\post =3$.  
  \item The $1$-tiles do not form a \emph{link} the sense of
    Definition~\ref{def:link}. 
  \end{itemize}
\end{lemma}

This is essentially \cite[Theorem 13.2]{expThurMarkov}, see also
\cite{CFPsubdiv_rat}. A proof of this lemma is given in Section
\ref{sec:connecting-trees}, here we show how the arguments in
\cite{expThurMarkov} are slightly adjusted to obtain the
statement in the above form. 

The iterate $F=f^n$ as well as the $F$-invariant Jordan curve $\CC$ as
above will be fixed from now on, tiles are defined in terms of
$(F,\CC)$.

\smallskip
Let us first give a slightly incomplete outline of the
construction. 
Recall that $X^0_w,X^0_b$ are the white, black $0$-tiles; they are
both bounded by the invariant curve $\CC$. 
We consider a spanning tree of white $1$-tiles in
$X^0_w$. Then we consider a spanning tree of black $1$-tiles in $X^0_b$, the
complementary white $1$-tiles in $X^0_b$ form (``homotopically'') trivial
trees in the sense of Definition \ref{def:trivial_tree}. These (white)
trivial trees (in $X^0_b$) are then attached to the 
white spanning tree in $X^0_w$. 

This construction has to be adjusted slightly for the following
reason: the white $1$-tiles in $X^0_w$ (as well as the black $1$-tiles
in $X^0_b$) need not be connected. So there are no \emph{spanning}
trees as described before.  


\subsection{Decomposing $X^0_w$}
\label{sec:constr-conn}
Here we decompose the \emph{white} $0$-tile $X^0_w$ into white trees.

\medskip
Consider the white $1$-tiles in $X^0_w$. We assume in the next lemma
that they are all 
\emph{connected} at all $1$-vertices $v$ in the \emph{interior} of
$X^0_w$, and
\emph{disconnected} at all $1$-vertices on $\CC$. 
The resulting white connection graph may not be
connected. 

\begin{lemma}
  \label{lem:1main_cluster}
  The white connection graph in $X^0_w$ as above has exactly one
  (white) cluster
  that intersects all sides ($0$-edges). 
\end{lemma}

\begin{proof}
  Let $K$ be a (white) cluster in $X^0_w$ as above. Consider one
  component $B$ (in the standard topological sense) of $X^0_w\setminus
  K$. We call the set $a:= \partial B\cap K$ a \emph{boundary arc} of
  $K$. 

  \smallskip
  {\it Claim 1.} Every boundary arc $a$ as above is contained in a
  single black $1$-tile. 

  Clearly $a$ is a union of $1$-edges. Either $a$ starts and ends at
  two distinct $1$-vertices $v,w\in \CC$, or $a$ is a closed curve. 
  Let $E,E'\ni v$ be two
  $1$-edges in $a$ consecutive in $a\subset \partial B$; where
  $v\notin \CC$ is a 
  $1$-vertex. Note that by construction all white $1$-tiles $X_j\ni v$
  are connected at $v$. Thus $E,E'$ are contained in the same black
  $1$-tile. The claim follows.

  \smallskip
  Assume now that $a$ is not an arc having as two distinct endpoints
  the $1$-vertices $v,w\in \CC$. Then $a$ is a Jordan curve in the
  boundary of a single black $1$-tile. Thus the corresponding component
  $B$ is the interior of a single black $1$-tile. Thus $a$ does not
  separate $K$ from any other distinct white cluster $K'$ in $X^0_w$. 

  \smallskip
  We call a black $1$-tile $Y\subset X^0_w$ \emph{non-trivial} if
  $Y\cap \CC$ contains at least two $1$-vertices. A
  \emph{complementary component} of $Y$ is the closure of a component
  $X^0_w\setminus Y$. 
 
  \smallskip
  {\it Claim 2.} Let $X,X'\subset X^0_w$ be two distinct white
  $1$-tiles. Then $X,X'$ are contained in distinct white clusters
  $K,K'\subset X^0_w$ if
  and only if there is a black $1$-tile $Y\subset X^0_w$ such that
  $X,X'$ are contained in complementary components of $Y$. 

  The implication $(\Leftarrow)$ is clear. To see the other
  implication we note that if $X'$ is contained in a cluster distinct
  from the cluster $K\supset X$, then $X'$ has to be contained in the
  closure of one component of $X^0_w\setminus K$. Such a component is
  separated from $K$ by a boundary arc $a$. However, if $a$ does not
  contain two $1$-vertices $v,w\in \CC$ this component is a single
  black $1$-tile, meaning it does not contain $X'$. Otherwise $X'$ is
  separated from $X$ by the black $1$-tile $Y$ containing $a$, proving
  the claim.

  \smallskip
  Recall that we assumed that no $1$-tile joins opposite sides of
  $\CC$ (see Lemma~\ref{lem:exFC}). Thus for every non-trivial black
  $1$-tile $Y$ there is a complementary component of $Y$, denoted by
  $K_Y$, that intersects all $0$-edges. 

  We now define $\overline{K}:= \bigcap K_Y$, where the intersection
  is taken over all non-trivial black $1$-tiles $Y\subset
  X^0_w$. Since two non-trivial black $1$-tiles $Y,Y'\subset X^0_w$ do
  not cross, it follows that $\overline{K}$ intersects all $0$-edges. 

  By Claim 2 it follows that all white $1$-tiles contained in
  $\overline{K}$ are connected, i.e., belong to the same cluster
  denoted by $K$. 

  \smallskip
  Assume
  $\overline{K}$ intersects a given $0$-edge $E^0$ in a $1$-edge $E$. This
  cannot happen if $E$ is contained in a black $1$-tile $Y\subset
  X^0_w$, since $Y$ would be non-trivial, and the corresponding set
  $K_Y$ does not contain $E$. Thus $E$ is contained in a white
  $1$-tile, which is in $\overline{K}$.

  If $\overline{K}$ intersects $E^0$ only in a $1$-vertex $v$, there
  is a boundary arc $a\subset \partial \overline{K}$ containing
  $v$. Let $Y\subset X^0_w$ be the corresponding non-trivial black
  $1$-tile containing $a$. Let $E\subset a$ be the $1$-edge containing
  $v$. Since $E$ is not in $\CC$ the white $1$-tile containing $E$ is
  in $\overline{K}$. 

  This means there is a white $1$-tile in $K$ that
  intersects $E^0$. 
\end{proof}

In each white cluster in $X^0_w$ define a spanning tree (see
Definition \ref{def:cluster}). The spanning tree in the cluster
from Lemma \ref{lem:1main_cluster} is called the \defn{main tree}
$K_M$, the spanning trees in the other clusters are called the
\defn{secondary trees} in $X^0_w$. The 
\emph{connections} at all $1$-vertices $v\in X^0_w\setminus \CC$ are thus
\emph{defined}, they will not be changed any more in the construction. 

\smallskip
Let $\mathcal{E}$ be the boundary circuit of the main
tree $K_M$ (see Definition \ref{def:gamma1_connection} and Lemma
\ref{lem:cluster_sc}).  
Let $v_0,\dots ,v_{N-1}$ be the $1$-vertices on $\CC$ that $\mathcal{E}$
visits (in this order). Note that a $1$-vertex $v$ may appear
several times in this list. 

\begin{notation}
  Given points $v,w\in \CC$ denote by
  \begin{align}
    [v,w], (v,w),
  \end{align}
  the closed/open positively oriented arc on $\CC$ from $v$ to
  $w$. Note that $(v,v)=\emptyset$. 
\end{notation}

\begin{lemma}
  \label{lem:viKM}
  The points $\{v_i\}$ satisfy the following. Indices are taken $\bmod
  N$ here.

  \begin{enumerate}
  \item 
    \label{item:viKM_1}
    Each (open) arc $(v_i,v_{i+1})$ 
    contains no point $v_l$.
    
    This means the points $\{v_i\}$ are positively oriented on $\CC$. 
  \item 
    \label{item:viKM_2}
    The points $v_i,v_{i+1}$ are not contained in disjoint $0$-edges,
    in particular each $0$-edge contains at least one point $v_i$.
  \item 
    \label{item:viKM_3}
    For all $v_i,v_{i+1}$ there is a black $1$-tile $Y\ni
    v_i,v_{i+1}$.
  \item 
    \label{item:viKM_4}
    Let $K$ be a secondary tree in $X^0_w$. Then there is an arc
    $[v_i,v_{i+1}]$ such that
    \begin{equation*}
      K\cap \CC\subset [v_i,v_{i+1}].
    \end{equation*}
  \end{enumerate}
\end{lemma}

\begin{proof}

  (\ref{item:viKM_1})
  Let $K_{M,\epsilon}$ be a geometric representation of $K_M$ as in
  Lemma \ref{lem:cluster_sc} (\ref{item:cluster_sc3}).
  The path $\gamma_i$ on $\mathcal{E}$ between
  $v_i$ and 
  $v_{i+1}$ is then represented by a Jordan arc $\gamma_{i,\epsilon}$  
  with endpoints $v_{i,\epsilon},v_{i+1,\epsilon}$, such that
  $\abs{v_i-v_{i,\epsilon}}, \abs{v_{i+1}-v_{i+1,\epsilon}}$ are
  arbitrarily small. Since all white 
  $1$-tiles are disconnected at every $1$-vertex $v\in \CC$ we can
  assume that $v_{i,\epsilon}\in \CC$ 
  and
  $\gamma_{i,\epsilon}\subset X^0_w$ for all $i$ .

  The arcs $\gamma_{i,\epsilon}$ are
  non-crossing, thus the points $\{v_{i,\epsilon}\}$ are ordered
  cyclically or anti-cyclically on $\CC$. Hence the points $\{v_i\}$
  are ordered cyclically or anti-cyclically on $\CC$.

  \smallskip
  The winding number of $\mathcal{E}$ around $x\notin \mathcal{E}$ is
  $1$ if and only if $x$ is in the interior of a white $1$-tile of the
  main tree. This follows from an inductive argument as in
  Corollary \ref{cor:spanning_tree_ind}. 

  Assume the points  $\{v_i\}$ are ordered  anti-cyclically on $\CC$. Let
  $\CC_i$ be the (positively oriented) arc on $\CC$ between
  $v_i,v_{i+1}$. Then $\gamma_i+\CC_i$ has winding number $0$ around
  any point $x$ in the  interior of a $1$-tile of the main tree. Thus
  $\mathcal{E}+\CC$ has winding number $0$ around such an $x$. This is
  a contradiction. 

  \medskip
  (\ref{item:viKM_3})
  Consider $v_i, v_{i+1}$. Then either
  \begin{itemize}
  \item $v_i=v_{i+1}$ in which case the statement is trivial;
  \item or $[v_i,v_{i+1}]$ is a $1$-edge, property
    (\ref{item:viKM_3}) is then clear again;
  \item or $v_i,v_{i+1}$ are the boundary points of a boundary arc $a$
    of $K_M$, as in Claim~1 from the proof of Lemma
    \ref{lem:1main_cluster}. In this case there is a black $1$-tile
    $Y\supset a$.   
  \end{itemize}

  \medskip
  (\ref{item:viKM_2})
  This follows immediately from (\ref{item:viKM_3}) and the assumption
  that no $1$-tile intersects disjoint $0$-edges. Furthermore $K_M$
  intersects a $0$-edge $E$ if and only if it intersects it in some
  $1$-vertex. The set of all $1$-vertices in which $K_M$ intersects
  $\CC$ is equal to the set $\{v_i\}$. Thus, since $K_M$ intersects
  each $0$-edge, it follows that each $0$-edge contains one point
  $v_i$. 

  \medskip
  (\ref{item:viKM_4})
  The reader is reminded of Claim~1 and Claim~2 in the proof of
  Lemma~\ref{lem:1main_cluster}.   
  For every secondary component $K$ there is an arc $a$ contained in a
  (non-trivial) black $1$-tile $Y$ such that $\inte K$ is in the
  component of $X^0_w\setminus a$ not intersecting all $0$-edges. Let
  $v_i,v_{i+1}$ be the endpoints of $a$ (see the discussion from
  (\ref{item:viKM_3})), then
  \begin{equation*}
    K\cap \CC\subset [v_i,v_{i+1}].
  \end{equation*}

\end{proof}

\subsection{Decomposing $X^0_b$}
\label{sec:decomposing-x0_b}

We now decompose the \emph{black} $0$-tile $X^0_b$. 
Consider the \emph{black} $1$-tiles in $X^0_b$. Construct clusters
of black $1$-tiles as before. Namely assume that all black $1$-tiles
are connected at each $1$-vertex $v\in X^0_b\setminus\CC$. All (black
and white) $1$-tiles in $X^0_b$  are
disconnected at each $1$-vertex $v\in \CC$. Pick a spanning tree in
each cluster (of black $1$-tiles in $X^0_b$). 
This \emph{defines the connections} at all $1$-vertices $v \in
X^0_b\setminus \CC$, they will not be changed anymore in the 
construction.    
As in Lemma~\ref{lem:1main_cluster}, there is exactly one such tree
(of black $1$-tiles in $X^0_b$) that intersects all $0$-edges. 

\smallskip
Consider now the \emph{white}
$1$-tiles in $X^0_b$. The connections at $1$-vertices $v\in
X^0_b\setminus \CC$ are already given (they are all disconnected at
each $1$-vertex $v\in \CC$).    

\begin{lemma}
  \label{lem:white_X0b}
  Every white cluster $K$ in $X^0_b$ as above 
  \begin{itemize}
  \item is a tree;
  \item furthermore
    \begin{equation*}
      K\cap\CC\subset [v,w],
    \end{equation*}
    where $v,w\in \CC$ are $1$-vertices contained in a \emph{single}
    white $1$-tile.  
  \end{itemize}
\end{lemma}

\begin{proof}
  Assume $K$ is not a tree. Then $K$ has at least two distinct
  boundary circuits (see Lemma \ref{lem:cluster_sc}). 

  \begin{claim}
    There is a (white) $1$-tile $X\subset K$ and a $1$-vertex $v\in X$
    at which $1$-edges $E,E'\subset X$ from distinct boundary circuits
    intersect.  
  \end{claim}
  If the claim were not true we could partition $K$ into $1$-tiles
  containing $1$-edges from distinct boundary circuits. These
  partitions, and therefore $K$, would not be connected by Lemma
  \ref{lem:prop_successor}.  

  \smallskip
  Let $v, E,E'$ be as in the claim. Note that $v\notin \CC$, since
  all $1$-tiles are disconnected at $\CC$.  

  Consider the \emph{black} $1$-tiles $Y,Y'\subset X^0_b$ that
  contain $E,E'$. Let $K_b,K_b'\subset X^0_b$ be the black clusters
  containing $Y,Y'$. Since they are by assumption trees, they are
  distinct (again by Lemma \ref{lem:cluster_sc}). 

  \smallskip
  On the other hand the (black) $1$-tiles $Y,Y'$ were connected at
  $v$, before spanning trees were picked. This means they are in the
  same tree ($K_b=K_b'$), which is a contradiction.

  \medskip
  The arguments from Lemma \ref{lem:1main_cluster} and Lemma
  \ref{lem:viKM} apply verbatim to
  $X^0_b$. Thus there is a unique black tree $K_{M,b}\subset X^0_b$ that
  intersects each $0$-edge. Let $w_0,\dots, w_{\widetilde{N}}$ be the
  $1$-vertices that the boundary circuit of $K_{M,b}$ visits (in this
  order); note that these points are ordered positively
  on $\CC$ (recall that $1$-edges in a boundary circuit of a cluster
  were always \emph{positively oriented} as boundary of \emph{white}
  $1$-tiles they are contained in, regardless of the color of the
  cluster). As in Lemma \ref{lem:viKM} one obtains that the endpoints 
  $w_i,w_{i+1}$ of each arc
  $[w_i,w_{i+1}]$ are contained in a single white $1$-tile. Each set
  $K\cap\CC$ is contained in one such  arc $[w_i,w_{i+1}]$. 
  

  
\end{proof}

We call the (white) trees from the previous lemma the \defn{secondary 
  trees} in $X^0_b$.  Let us record the following immediate
consequence of Lemma \ref{lem:viKM} and Lemma \ref{lem:white_X0b}. 

\begin{lemma}
  \label{lem:secondary_tree}
  No secondary tree (in $X^0_w$ or $X^0_b$) intersects disjoint
  $0$-edges. 
\end{lemma}

We will need to break up boundary circuits.
\begin{definition}[Subpaths of boundary circuits]
  \label{def:path_pq_connection}
  Let $\EC$ be a boundary circuit, $D,E\subset \EC$ two
  $1$-edges. Then $\EC(D,E)$ is the positively oriented subpath
  (of $1$-edges) of $\EC$ with initial $1$-edge $D$, terminal
  $1$-edge $E$. Note that $\EC(E,E)=E$. 
\end{definition}

 



In the next lemma we consider a secondary tree $K\subset X^0_b$ with
boundary circuit $\EC$. Consider two distinct $1$-vertices $v,w\in
(\EC\cap\CC)$. Let $E_v,E_v'\subset \EC$ and $E_w,E_w'\subset \EC$ be
succeeding $1$-edges at $v,w$. 

Let $x,y\in \CC$, in the following we write $[x,y]_b$ for the boundary
arc on $\CC=\partial X^0_b$ between $x,y$ that is \emph{positively
  oriented} with respect to $X^0_b$ (thus negatively oriented on $\CC$).

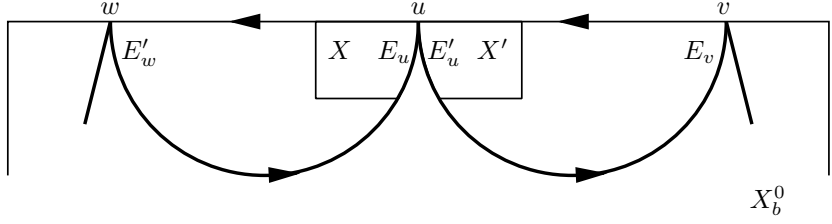
\begin{figure}
  \centering
   \begin{overpic}
      [width=11cm, 
      tics =20]{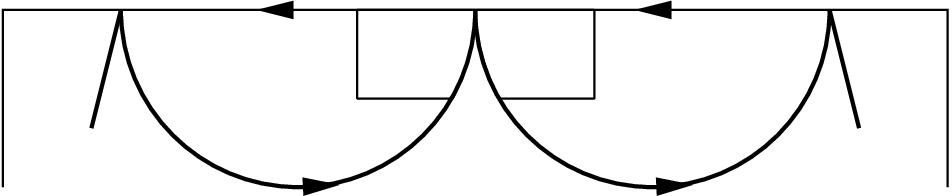}
      \put(90,-3){$X^0_b$}
      \put(86,20.5){$v$}
      \put(49,20.5){$u$}
      \put(11.5,20.5){$w$}
      \put(14,15){$E_w'$}
      \put(82,15){$E_v$}
      \put(39,15){$X$}
      \put(45,15){$E_u$}
      \put(51,15){$E_u'$} 
      \put(57,15){$X'$}
  \end{overpic}
  \caption{Illustration to Lemma~\ref{lem:boundary_K2nd}.}
  \label{fig:lemma79}
\end{figure}
 
\begin{lemma}
  \label{lem:boundary_K2nd}
  The subpath $\EC(E_w',E_v)$ does not intersect
  $[v,w]_b\setminus\{v,w\}$.  
\end{lemma}

\begin{proof}
  The situation is illustrated in Figure~\ref{fig:lemma79}. 
  Assume the statement is false, meaning that $\EC(E_w',E_v)$ intersects
  $[v,w]_b\setminus\{v,w\}$ in 
  a $1$-vertex $u$ ($\in \CC$). Let $E_u,E_u'\subset \EC(E_w',E_v)$ be the
  succeeding $1$-vertices at $u$. Then $\inte K$ is divided into
  points bounded by (having winding number $1$) $\EC(E_w',E_u)\cup
  [u,w]_b$ and $\EC(E_u',E_v)\cup [v,u]_b$. 

  Thus $E_u,E_u'$ are contained in different white $1$-tiles
  $X,X'\subset K$. Thus $X,X'$ are connected at $u$. This contradicts
  the construction of $K$, where no $1$-tiles are connected at any
  $1$-vertex in $\CC$.  
\end{proof}

\subsection{Connecting the trees}
\label{sec:connecting-trees}

The secondary trees are attached to the main tree at the $1$-vertices
on $\CC$. 

Initially all white $1$-tiles are disconnected at each
$1$-vertex $v\in \CC$. To use the results from
Section~\ref{sec:adding-clusters} we want the connections at all
$1$-vertices $v\in \CC$ to be cnc-partitions. Thus we now assume that
all black $1$-tiles are all connected at each $1$-vertex 
$v\in \CC$, thus the connections form cnc-partitions as desired.

\medskip
We first add secondary trees to ensure that all points of $\post$ are
contained in the main tree. 
Consider the main tree $K_M$ (in $X^0_w$) from Section
\ref{sec:constr-conn}. Let $v_0,\dots, v_{N-1}$ be the $1$-vertices on
$\CC$ along the boundary circuit $\EC$ of $K_M$, see
Lemma~\ref{lem:viKM}.  

Consider one (positively oriented) $0$-edge $E^0$ with terminal point
$p\in \post$, let $v_i$ be the
last of the $1$-vertices as above on $E^0$. Then either
\begin{itemize}
\item $v_i=p$. Let $E_j\subset \EC$ be last $1$-edge with terminal
  point $v_i$, $E_{j+1}\subset \EC$ be the succeeding $1$-edge. 
  The connection at $p$ is now marked by $E_j,E_{j+1}$, see Corollary
  \ref{cor:marking}. 
\item $v_i\notin \post$. Consider the $1$-edge $E=[v_i,w]\subset E^0$
  succeeding $v_i$ in $\CC$. Let $K$ be the secondary cluster
  containing $E$. This means $K$ contains the (unique) white $1$-tile
  containing $E$. Add $K$ to the main tree $K_M$ at $v_i$. Note that
  no white $1$-tile is connected at $v_i$, so this is possible by
  Lemma~\ref{lem:change_conn}. We obtain a new main tree, still
  denoted by $K_M$. 
\item Repeat the above procedure till the main tree contains $p$. 
\end{itemize}

The added secondary components will only intersect the $0$-edges
preceding and succeeding $E^0$. 
Then we want to use the
same procedure on the other $0$-edges. There is one problem however:
we may encounter a $1$-edge $E$ as above that belongs to a secondary
component already added before (when the above procedure was applied
to a \emph{different} $0$-edge $\widetilde{E}^0$). This may lead to a
boundary circuit of $K_M$ in which the postcritical points are
traversed not in the same order as in $\CC$, violating (C~\ref{item:prop_conn_2}).  

To elaborate, let $E^0_1=E^0$, and $E^0_2,E^0_3$ be the 
$0$-edges succeeding $E^0_1$. Let $q$ be the terminal point of $E^0_2$,
and $v_j$ be the last of the points $\{v_i\}$ on 
$E^0_2$.  The described problem occurs if there is a secondary
component $K$ containing a $1$-edge 
in $[v_i,p]\subset E^0_1$ and a $1$-edge in $[v_j,q]\subset E^0_2$. By
Lemma \ref{lem:viKM} (\ref{item:viKM_3}) and (\ref{item:viKM_4}) 
as well as Lemma \ref{lem:white_X0b} this can only happen if there
are white/black $1$-tiles \defn{linked} in a certain way, see Figure
\ref{fig:link}.   

\begin{figure}
  \centering
  \includegraphics[width=5cm]{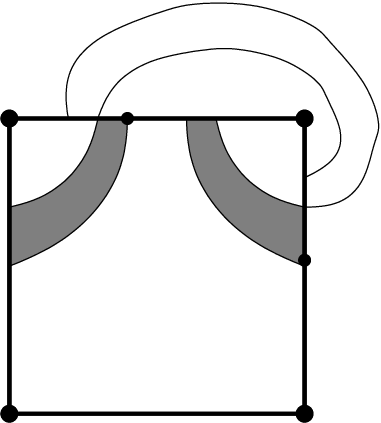}
  \begin{picture}(10,10)
    \put(-25,30){$E^0_1$}
    \put(-60,55){$X_1$}
    \put(-125,55){$X_2$}
    \put(-10,135){$Y$}
    \put(-80,120){$E^0_2$}
    \put(-165,30){$E^0_3$}
    \put(-27,58){$v_i$}
    \put(-95,105){$v_j$}
    \put(-42,103){$p$}
    \put(-155,110){$q$}
  \end{picture}
  \caption{A link.}
  \label{fig:link}
\end{figure}

\begin{definition}[Link]
  \label{def:link}
  A \defn{link} means that there exists the following.
  \begin{itemize}
  \item A (black) $1$-tile $X_1$ containing $v_i\in E^0_1$ and intersecting 
    $E^0_2$.
  \item A (black) $1$-tile $X_2$ containing $v_j\in E^0_2$ and intersecting
    $E^0_3$.
  \item A (white) $1$-tile $Y$ intersecting $[v_i,p]\subset E^0_1$ and
    $[v_j,q]\subset E^0_2$. 
  \end{itemize}
\end{definition}

Thus we have given the description of the last property in
Lemma~\ref{lem:exFC}. 

\begin{proof}
  [Proof of Lemma \ref{lem:exFC}]
  We essentially recall the proof of
  \cite[Theorem~13.2]{expThurMarkov}, see also
  \cite[Theorem~13.3]{expThurMarkov} and its proof.

  More precisely, we break up each $0$-edge into two $0$-arcs and use
  the same arguments as in \cite{expThurMarkov} to show that there is
  an $f^n$-invariant curve $\widetilde{\CC}$, such that no $n$-tile
  connects disjoint $0$-arcs.

  \smallskip
  Let $k_0$ be a fixed integer such that there are at least twice as
  many $k_0$-vertices as postcritical points (recall that the number
  of $n$-vertices grows exponentially). Fix a Jordan curve $\CC\subset
  S^2$ such that $\post\subset \CC$; additionally $\CC$ has the
  property that each arc on $\CC$
  between two consecutive postcritical points $p,q$ contains a
  $k_0$-vertex distinct from $p,q$. Let $P$ be the set of all such
  $k_0$-vertices and postcritical points. The points in $P$ divide
  divide $\CC$ into \emph{$0$-arcs}. Each $0$-edge on $\CC$ is divided
  into two $0$-arcs. 

  \smallskip
  Consider the $n$-tiles given in terms of $(f,\CC)$ where $n\geq
  k_0$. Since $f$ is expanding $n$-tiles get arbitrarily small,
  meaning that $\max_{X\in \X^n} (\diam X) \to 0$ as $n\to \infty$. 
  This implies by \cite[Lemma~10.17]{expThurMarkov} that there is an
  $n_0\geq k_0$ such for all $n\geq n_0$ there is a Jordan curve
  $\CC'\subset f^{-n}(\CC)$ isotopic to $\CC$ rel.\ $P$ (thus $P\subset
  \CC'$). Furthermore no 
  $n$-tile  joins opposite sides of $(\CC',P)$. This means there is no
  $n$-tile that intersects disjoint closed \emph{$0$-arcs} into which
  $P$ divides the curve $\CC'$. 

  \smallskip
  Let $H\colon S^2\times [0,1] \to S^2$ be an isotopy rel.\ $P$ that
  deforms $\CC$ to $\CC'$, i.e., $H_1(\CC)= \CC'$. Then $\widehat{F}:=
  H_1\circ f^{n}$ is a Thurston map, such that $\CC'$ is
  $\widehat{F}$-invariant, since
  $\widehat{F}(\CC')=H_1(f^n(\CC'))\subset H_1(\CC)= \CC'$. The
  $1$-tiles for $(\widehat{F},\CC')$ are exactly the $n$-tiles for
  $(f,\CC)$. Since no $1$-tile for $(\widehat{F},\CC')$ joins opposite
  sides of $\CC'$, we can choose $\widehat{F}$ to be expanding, see
  \cite[Corollary~12.18]{expThurMarkov}. Furthermore no $1$-tile for
  $(\widehat{F}, \CC')$ intersects disjoint $0$-arcs of $\CC'$.

  \smallskip
  The map $\widehat{F}$ is Thurston equivalent to $f^n$. Since they
  are both expanding, they are actually topologically conjugate, i.e.,
  there is a homeomorphism $h\colon S^2\to S^2$, such that
  $h\circ \widehat{F}\circ h^{-1} = f^n$ (see
  \cite[Theorem~10.4]{expThurMarkov}). Let $\widetilde{\CC}:= 
  h(\CC')$. Note that $\widetilde{\CC}$ is $f^n$-invariant, since
  $f^n(\widetilde{\CC})= h\circ \widehat{F}\circ
  h^{-1}(\widetilde{\CC}) = h\circ \widehat{F}(\CC')\subset h(\CC')=
  \widetilde{\CC}$. 

  We call the images of $0$-arcs on $\CC'$ by $h$ the $0$-arcs of
  $\widetilde{\CC}$. The images of $1$-tiles for $(\widehat{F}, \CC')$
  by $h$ are the $n$-tiles for $(f,\widetilde{\CC})$. It follows that
  no $n$-tile (for $(f,\widetilde{\CC})$) intersects disjoint $0$-arcs
  of $\widetilde{\CC}$. Recall that each $0$-edge of $\widetilde{\CC}$
  contains exactly two $0$-arcs. 

  \medskip
  With this choice of $F=f^n$ and $\widetilde{\CC}$ we will show that
  a link as in Definition~\ref{def:link} cannot
  occur. Let $A_j^-,A_j^+$ be the two $0$-arcs in $E^0_j$, where 
  $A_j^+$ succeeds $A_j^-$ in $\CC$. Then the white $1$-tile $Y$ has
  to intersect $E^0_2$ in $\inte A^-_2$, while the black $1$-tile $X_2$ has
  to intersect $E^0_2$ in $\inte A^+_2$. The claim follows. 

\end{proof}

  


Since we assumed that $F=f^n$ and $\CC$ were chosen to satisfy the
properties from Lemma~\ref{lem:exFC}, there are no links. Thus the
following holds. 
Let $K$ be a secondary cluster added (to the
main tree) when considering the $0$-edge $E^0$; $\widetilde{K}$ a
secondary cluster added when considering a distinct $0$-edge
$\widetilde{E}^0$.

\begin{cor}
  \label{cor:disjoint_sec_cluster}
  The secondary clusters $K, \widetilde{K}$, given as in the setting
  as above, are distinct.
\end{cor}

Thus we can apply the above procedure to each $0$-edge. This yields
the (new) 
main tree (still denoted by $K_M$). Note that $K_M\supset \post$ by
construction. More precisely $K_M$ contains the marked succeeding
$1$-edges $E(p),E'(p)$ at each postcritical point $p$. 
%
This means that $K_{M,\epsilon}\supset \post$
(for any geometric representation $K_{M,\epsilon}$ of
$K_M$), see Lemma~\ref{lem:p_on_boundary}.  

\subsection{Main tree is in the right homotopy class}
\label{sec:proof-main-theorem}

Recall from Definition \ref{def:path_pq_connection} how a boundary
circuit $\EC$ was broken up into subpaths. 
Assume $\EC$ contains the marked succeeding $1$-edges $E(p),E'(p)$ at
$p\in \post$, as well as the marked succeeding $1$-edges $E(q),E'(q)$
at $q\in
\post$. Then 
\begin{align*}
  & \EC(p,q):=\EC(E'(p), E(q));
  &&
  \text{and for any $1$-edge } E\subset \EC
  \\ 
  &\EC(p,E):=\EC(E'(p), E),
  &&
  \EC(E,q):= \EC(E, E(q)).
  \intertext{Furthermore if $E,E'$ are succeeding in $\EC$ we define}
  & \EC(E',E)= \emptyset.
\end{align*}

We are now ready to finish the proof of Theorem~\ref{thm:main}. $K_M$
is the main tree as constructed in Section~\ref{sec:connecting-trees}.  

\begin{lemma}
  \label{lem:main_tree_homotopy}
  The main tree $K_M$ is in the right homotopy class, i.e.,
  satisfies \emph{(C~\ref{item:prop_conn_2})}.
\end{lemma}

\begin{proof}
  Let $\EC$ be the boundary circuit of $K_M$. 
  Consider a $0$-edge $E^0$ with initial/terminal
  points $p,q\in\post$; and the subpath $\EC(p,q)\subset\EC$ as
  defined above. We will prove the following.

  \begin{claim1}
    $\EC(p,q)$ does not intersect any $0$-edge disjoint with
    $E^0$.  
  \end{claim1}

  The statement of the lemma follows quickly from this claim. 
  Namely consider a geometric representation $K_{M,\epsilon}$ of $K_M$,
  where 
  the neighborhoods $U(v)$ from (\ref{eq:defhv}) were
  chosen such that $U(v)\cap \CC=\emptyset$ whenever $v\notin \CC$. 
  It follows from Claim 1 that the (positively oriented) arc on
  $\partial K_{M,\epsilon}$ from $p$ to $q$ does not intersect
  $0$-edges disjoint from $E^0$. Theorem \ref{thm:isotopy2} now
  finishes the proof.  

  \medskip
  To prove Claim 1 we go through the construction of $K_M$. 
  Consider $K_{M,0}$, the main tree from
  Section \ref{sec:constr-conn} (before any secondary tree was
  added), with boundary circuit $\EC_0$. 
  Let $w_0,w_1\in E^0$ be the first/last $1$-vertices on $E^0$ that
  $\EC_0$ visits; and $E_0,E_0'\subset \EC_0$ as well as
  $E_1,E_1'\subset \EC_0$ be the 
  first/last succeeding $1$-edges at $w_0,w_1$. Consider
  $\EC_0(E_0',E_1)$, note that $\EC(E'_0,E_1)=\EC(E'_0, E_0)=\emptyset$ in the case
  that $\EC_0$ intersects $E^0$ only once. This subpath does not intersect any $0$-edge
  disjoint from $E^0$ by Lemma \ref{lem:viKM} (in fact it may only
  intersect adjacent $0$-edges if $w_0=p$ or $w_1 =q$). 

  Note that $\EC_0(E'_0, E_1)$ is a subpath of $\EC(p,q)$, or
  $\EC(E'_0,E_1)= \EC_0(E'_0,E_1)$, which we call the \emph{middle}
  subpath of $\EC(p,q)$. The remaining subpaths of $\EC(p,q)$ are
  given as follows. Let $D_0$ be the $1$-edge preceding $E'_0$ in
  $\EC$ and $D_1$ be the $1$-edge succeeding $E_1$ in $\EC$. Then 
  the \emph{initial} subpath of $\EC(p,q)$ is $\EC(p,D_0)$ (connecting $p$ to
  $\EC(E'_0,E_1)$), and the \emph{terminal} subpath of $\EC(p,q)$ is
  $\EC(D_1',q)$ (connecting $\EC(E'_0, E_1)$ to $q$). Note that the
  initial and/or the terminal subpath may be empty. We focus our
  attention for now on the terminal subpath.


  \smallskip
  Let $K_1,\dots, K_m$ be the secondary trees that were added in
  Section \ref{sec:connecting-trees} to
  ``reach'' the postcritical point $q$. The last secondary tree $K_m$
  contains the postcritical point $q$ by construction. 
  
  Let $K_{M,j}$ be the main tree obtained when the
  secondary tree $K_{j}$ was added to $K_{M,j-1}$ at the $1$-vertex
  $w_j\in E^0$. Let $E_j,E_j'\subset K_{M,j-1}$, and $D_j,D_j'\subset
  K_j$ be the succeeding $1$-edges associated to adding $K_j$ to
  $K_{M,j-1}$ by Lemma \ref{lem:add_successor}. Note that by
  construction the $1$-vertices of $K_{M,j}$ closest to $q$ on the
  $0$-edge $E^0$ are contained in $K_j\subset K_{M,j}$. 
  Thus $K_{j+1}$ is attached to $K_{M,j}$ at
  $1$-edges contained in $K_j$.

  Thus if we denote by $\EC_j$ the boundary circuit of the secondary tree
  $K_j$, then $D_j,D_j',E_{j+1}$, $E_{j+1}'\in \EC_j$ and
  \begin{equation*}
    \EC_j \text{ consists of the two (non-empty) subpaths } \EC_j(D_j', E_{j+1}),
    \EC_j(E_{j+1}',D_j), 
  \end{equation*}
  for $j=1,\dots, m-1$, we break $\EC_m$ up
  into the (non-empty) subpaths $\EC_m(D'_m,q)$, $\EC_m(q,D_m)$. 
 
  \begin{figure}
    \centering
    \includegraphics[width=12cm]{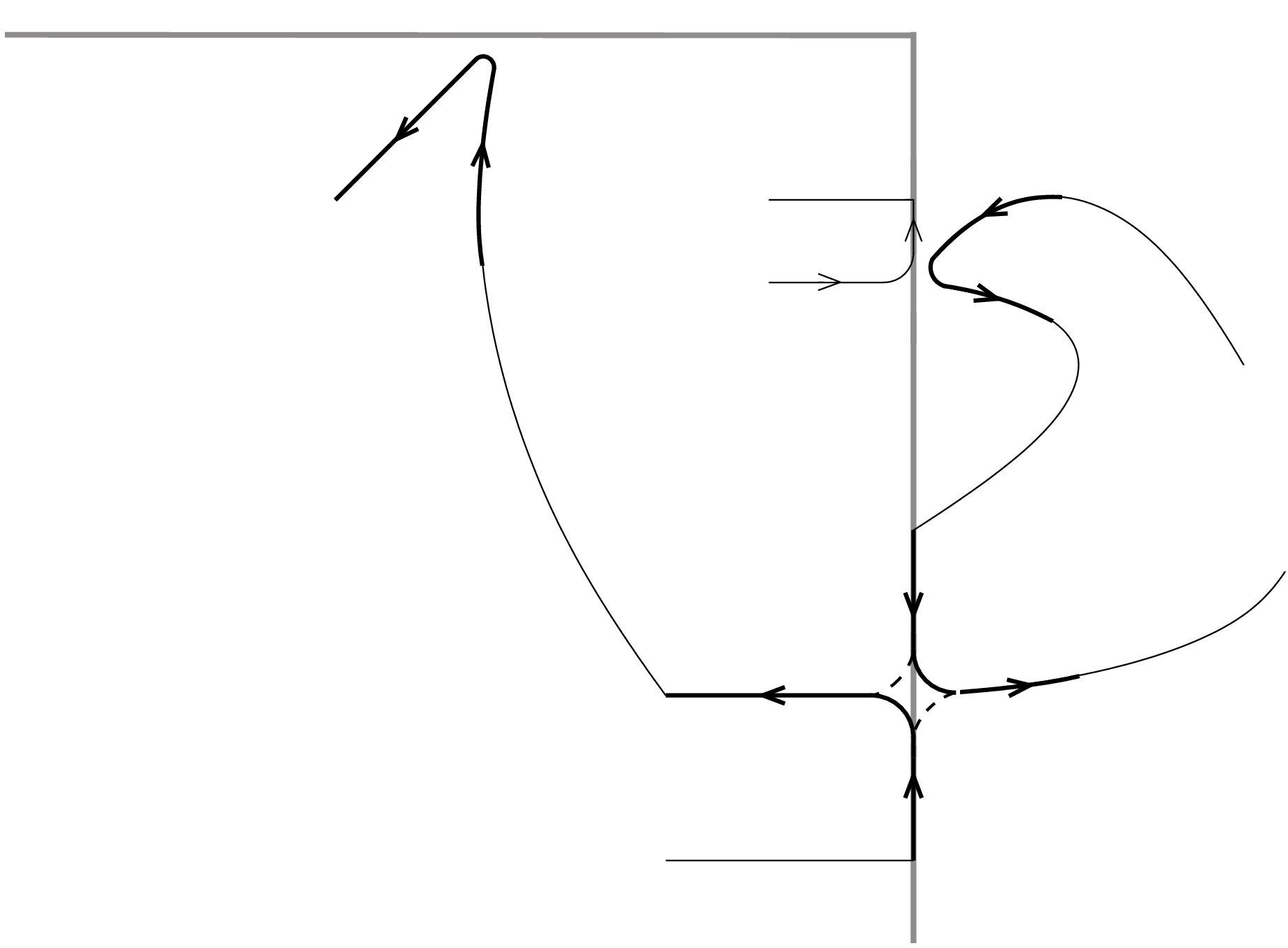}
    \begin{picture}(10,10)
      \put(-100,250){${q}$}
      \put(-120,214){${E^0}$}
      \put(-118,187){$\scriptstyle{D_2'}$}
      \put(-150,187){$\scriptstyle{K_2}$}
      \put(-130,165){$\scriptstyle{D_2}$}
      \put(-50,130){${K_1}$}
      \put(-93,200){$\scriptstyle{E_2}$}
      \put(-93,162){$\scriptstyle{E_2'}$}
      \put(-97,92){$D_1$}
      \put(-90,55){$D_1'$}
      \put(-120,40){$E_1$}
      \put(-145,73){$E_1'$}
      \put(-180,50){$K_{M,0}$}
      \put(-213,205){$\scriptstyle{E_N}$}
      \put(-257,215){$\scriptstyle{E_N'}$}
      \put(-350,247){${r}$}
      \put(-330,225){${E^0}'$}
      \put(-96,5){$p$}
      \put(-220,130){$\EC_0$}
      \put(-40,65){$\EC_1$}
    \end{picture}
    \caption{Adding $K_j$ to $K_{M,j-1}$.}
    \label{fig:addKj2K}
  \end{figure}

  Lemma
  \ref{lem:add_successor} implies that the terminal
    subpath $\EC(D_1',q)$ is given as the concatenation of (subpaths
  from the boundary circuits from the secondary trees $K_j$)  
  \begin{equation}
    \label{eq:EC_tail}
    \EC_1(D_1',E_2),\EC_2(D_2',E_3),\dots, \EC_m(D_m',q),
  \end{equation}
  see Figure \ref{fig:addKj2K}. It
  follows from Lemma \ref{lem:secondary_tree} that $\EC(D_1',q)$ does
  not intersect any $0$-edge disjoint from $E^0$.

  \medskip
  It remains to show that the initial subpath does not
  intersect a $0$-edge disjoint from $E^0$. 

  Instead of looking at the initial subpath of $\EC(p,q)$ we consider
  the initial subpath of $\EC(q,r)$. Here $r$ is the terminal point of
  the $0$-edge ${E^0}'$ succeeding $E^0$. Let $E_N\subset \EC_0$ be
  the first $1$-edge intersecting ${E^0}'$ in a $1$-vertex $w_N$. The
  initial subpath of  
  $\EC(q,r)$ is $\EC(q,E_N)$; it is given as the concatenation of
  \begin{equation*}
    \EC_m(q,D_m),\EC_{m-1}(E_m',D_{m-1}),\dots, \EC_1(E_2',D_1),
    \EC_0(E_1',E_N); 
  \end{equation*}
  where $D_j,E_j'$ are as above. These are the ``complementary
  subpaths'' to the ones in (\ref{eq:EC_tail}) (of the boundary
  circuits of the secondary trees $K_j$). See again
  Figure~\ref{fig:addKj2K}.

  It remains to show that this path does not intersect a $0$-edge
  disjoint from ${E^0}'$. Clearly $\EC_0(E_1',E_N)$ intersects
  $\CC$ only at the endpoints, which are in $E^0$ and ${E^0}'$. 

  \smallskip
  Recall that $\EC_j(E_{j+1}',D_j)\subset K_j$, where $K_j$ does not
  intersect disjoint $0$-edges. Thus $\EC_j(E_{j+1}',D_j)$ may only
  intersect $E^0,{E^0}'$, or the $E^0$ preceding $0$-edge
  $\widetilde{E}^0$. 

  \begin{claim2}
    The subpath $\EC_j(E_{j+1}',D_j)$ does not intersect
    $\widetilde{E}^0$. 
  \end{claim2}
  This is clear if $K_j\subset X^0_w$, since then $K_j\cap \CC\subset
  [w_1,q]\cup [q,w_N]$ by Lemma \ref{lem:viKM} (\ref{item:viKM_4}). 
  
  Assume now that $K_j\subset X^0_b$. Let $w$ be the initial point of
  $\EC_j(E_{j+1}',D_j)$ and $v$ be its terminal point. Note that by
  construction $w\in E^0$ is closer to $q$ on $E^0$ than $v\in
  E^0$. From Lemma \ref{lem:boundary_K2nd} it follows that
  $\EC_j(E_{j+1}',D_j)\subset [v,w]\subset E^0\setminus\{p\}$.   Claim~2 follows. 
  
  \bigskip
  The argument that the initial subpath $\EC(p,D_0)$ does not
  intersect $0$-edges disjoint from $E^0$ is completely
  analogous. This finishes the proof of Claim 1, thus the proof of the lemma.

\end{proof}

We finish the construction of the main tree, i.e., of the connection
of $1$-tiles by adding the 
remaining secondary trees to the main tree arbitrarily, to
form the spanning tree $K_M$. The previous lemma, together with
Lemma~\ref{lem:add_trivial_tree} implies that $K_M$ satisfies
properties (C~\ref{item:prop_conn_1}) and (C~\ref{item:prop_conn_2}).  
Thus there is a pseudo-isotopy $H^0$ as required in
Definition~\ref{def:pseudo-isotopy-h0}, by
Lemma~\ref{lem:H0epsH0}. This yields 
the invariant Peano curve by Sections~\ref{sec:appr-gn},
\ref{sec:constr-g}. The proof of Theorem~\ref{thm:main} is thus
finished.

\section{Combinatorial construction of $\gamma^n$}
\label{sec:comb-constr-gamm}

The $(n+1)$-th approximation $\gamma^{n+1}$ of the invariant Peano
curve $\gamma$ was constructed as a deformation of $\gamma^{n}$ by
$H^n$. Here 
$H^n$ was the lift of the ``initial pseudo-isotopy'' $H^0$ by $F^n$. 
In this section we give an \emph{alternative} way to construct
$\gamma^{n+1}$ from $\gamma^n$, namely in a purely
\emph{combinatorial} fashion.  

Recall from Lemma \ref{lem:g1bdK} that the first approximation
$\gamma^1$ may be obtained as the \emph{boundary circuit} of the
white spanning tree, defined via the \emph{connection of $1$-tiles}.
Here we construct the \defn{connection of $n$-tiles} (which will again  
satisfy (C \ref{item:prop_conn_1}), (C \ref{item:prop_conn_2})), such
that $\gamma^n$ is the boundary circuit of the white tree of
$n$-tiles. See Figure \ref{fig:H0H1H2g} for an illustration of the
desired connections of $n$-tiles. 

\smallskip
The connections of $n$-tiles \emph{could} be constructed from the
approximations $\gamma^n$ (using Lemma \ref{lem:prop_successor}). We
do however take the opposite route here, namely we construct the
connections inductively and show that their boundary circuits are the 
approximations as defined before.

\subsection{Connection of $n$-tiles}
\label{sec:connection-n-tiles}

We give the (inductive) description of the connection of
$n$-tiles first, before showing that it has the desired properties. 

\smallskip
Fix $n\geq 1$.
Assume the connection of $n$-tiles is given. This means at each
$n$-vertex $v$ a cnc-partition $\pi^n_w(v)\cup \pi^n_b(v)$ is defined;
if $v=p\in\post$ it is marked (see Definition \ref{def:connection}). The
connection satisfies properties (C \ref{item:prop_conn_1}), (C
\ref{item:prop_conn_2}) and the (single) boundary circuit is equal to
the $n$-th approximation $\gamma^n$ (viewed as an Eulerian circuit).  

\smallskip
Consider now an $(n+1)$-vertex $v$. The connection of $(n+1)$-tiles at
$v$ is defined as follows. 

\begin{case}[1]
  $v$ is not an $n$-vertex.  

  Note, that this implies that $v$ is \emph{not a critical point}. Thus
  we can define the connection at $v$ as the ``pullback'' of the
  connection at $F(v)$. 
  
  More precisely let
  $w:= F(v)$ ($\in \V^n$). Let $X^n_0, \dots, X^n_{2m-1}$ be the
  $n$-tiles around $w$ 
  (labeled mathematically positively around $w$). Label the
  $(n+1)$-tiles around $v$, $X^{n+1}_0,\dots, X^{n+1}_{2m-1}$, such that
  $F(X^{n+1}_j)= X^n_j$ ($j=0,\dots, 2m-1$). Then
  \begin{equation}
    \label{eq:defpi_n_case1}
    X^{n+1}_i, X^{n+1}_j \text{ are connected at } v
    \;:\Leftrightarrow\; X^n_i,
    X^n_j  \text{ are connected at } w.
  \end{equation}
  In other words, the connection (of $(n+1)$-tiles) at $v$ is defined
  by  
  \begin{equation*}
    \pi^{n+1}_w(v)\cup \pi^{n+1}_b(v):= \pi^n_w(w)\cup \pi^n_b(w). 
  \end{equation*}
\end{case}

\smallskip
\begin{case}[2]
  $v$ is an $n$-vertex ($v\in \V^{n+1}\cap \V^n$).

  Then $p:= F^n(v)\in \post=\V^0$. 
  Consider two white $(n+1)$-tiles $X^{n+1}, Y^{n+1}\ni v$. They
  are connected (at $v$) if and only if they are 
  \begin{itemize}
  \item either contained in the image of the \emph{same} (white)
    $n$-tile $X^n$ by the 
    pseudo-isotopy $H^n$,
    \begin{equation*}
      X^{n+1},Y^{n+1} \subset H^n_1(X^n)
    \end{equation*}
    and their images by $F^n$ are connected, meaning the $1$-tiles
    \begin{equation*}
      F^n(X^{n+1}), F^n(Y^{n+1}) \text{ are connected at } p; 
    \end{equation*}
  \item or $X^{n+1},Y^{n+1}$ are contained in the images of
    \emph{connected} $n$-tiles $X^n, Y^n\ni v$, 
    \begin{align*}
      &X^{n+1}\subset H^n_1(X^n) , \,Y^{n+1}\subset H^n_1(Y^n) 
      \quad \text{and} 
      \\
      &X^n,Y^n \text{ are connected at } v,
    \end{align*}
    and $X^{n+1},Y^{n+1}$ both map to $1$-tiles that are ``connected
    to the marked succeeding $1$-edges'', meaning the $1$-tiles
    \begin{align*}
      &F^n(X^{n+1}), F^{n}(Y^{n+1}) \text{ are connected at $p$ to the
        white $1$-tiles $X^1,\widetilde{X}^1$}
      \\
      &\text{that \emph{contain} the \emph{marked} succeeding
        $1$-edges $E^1, \widetilde{E}^1$}.  
    \end{align*}
  \end{itemize}
  The connection of black $(n+1)$-tiles at $v$ is defined analogously
  to the above.

  \begin{figure}
    \centering
    \includegraphics[width=12cm]{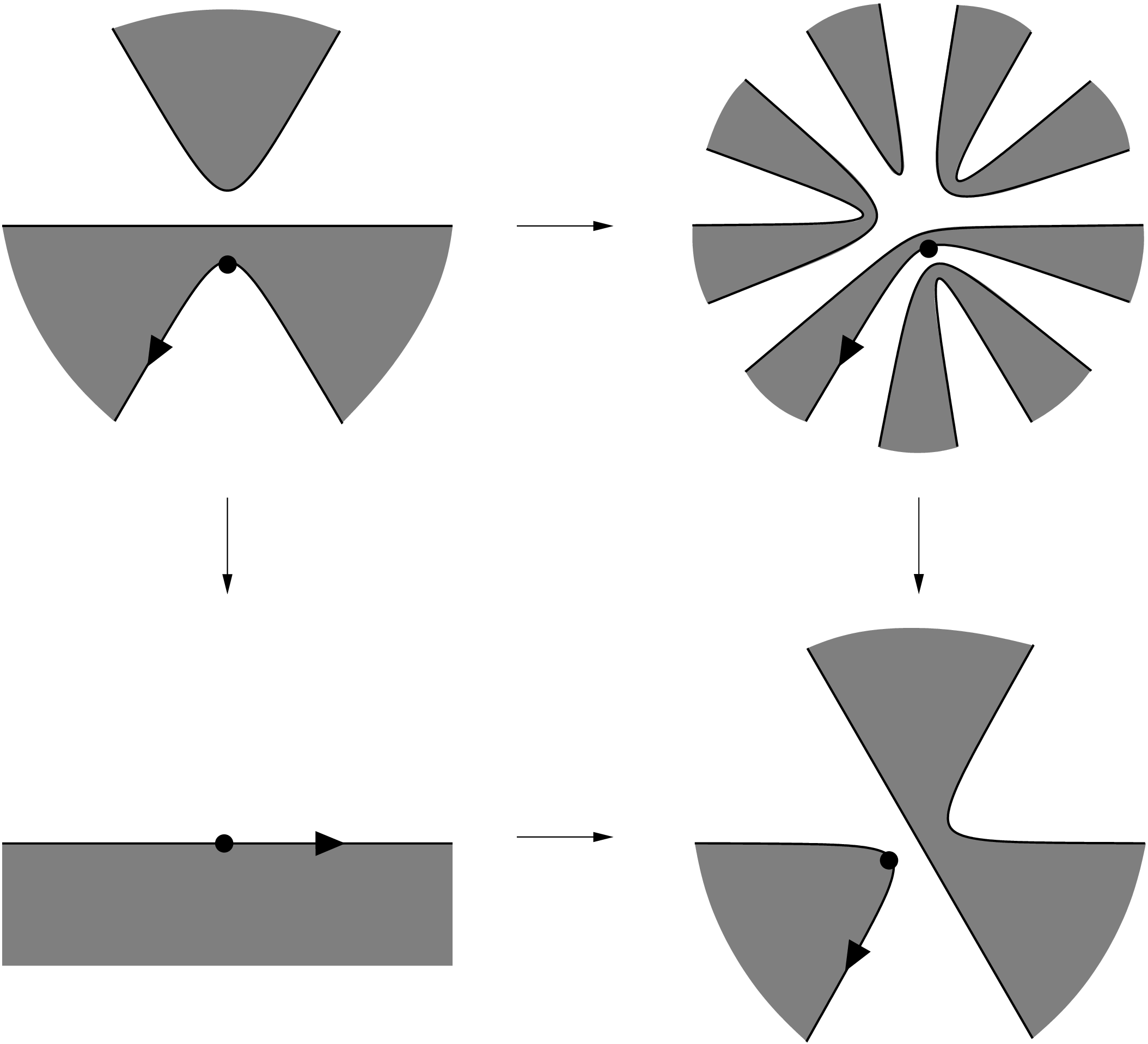}
    \begin{picture}(10,10)
      %
      \put(-88,140){${F^n}$}
      \put(-294,140){${F^n}$}
      \put(-182,67){$H^0$}
      \put(-182,249){$H^n$}
      %
      \put(-320,70){$\scriptstyle{X^{0}_{w}}$}
      \put(-320,10){$\scriptstyle{X^{0}_{b}}$}
      \put(-275,67){$\scriptstyle{p}$}     
      \put(-250,67){$\scriptstyle{\widetilde{E}^0}$}     
      %
      \put(-80,0){$\scriptstyle{X^{1}_{0}}$}
      \put(-110,-9){$\scriptstyle{\widetilde{E}^{1}}$}
      \put(-11,25){$\scriptstyle{X^{1}_{1}}$}
      \put(-17,92){$\scriptstyle{X^{1}_{2}}$}
      \put(-63,127){$\scriptstyle{X^{1}_{3}}$}
      \put(-140,92){$\scriptstyle{X^{1}_{4}}$}
      \put(-153,25){$\scriptstyle{X^{1}_{2k-1}}$}      
      \put(-82,61){$\scriptstyle{p}$}

      %
      \put(-280,180){$\scriptstyle{X^n_0}$}
      \put(-222,200){$\scriptstyle{X^n_1}$}
      \put(-222,270){$\scriptstyle{X^n_2}$}  
      \put(-287,311){$\scriptstyle{X^n_3}$}
      \put(-338,270){$\scriptstyle{X^n_4}$}
      \put(-354,200){$\scriptstyle{X^n_{2m-1}}$}
      \put(-317,175){$\scriptstyle{\widetilde{E}^n}$}
      \put(-280,223){$\scriptstyle{v}$}
      %
      \put(-105,170){$\scriptstyle{X^{n+1}_0}$}
      \put(-113,177){$\scriptstyle{\widetilde{E}^{n+1}}$}
      \put(-80,168){$\scriptstyle{X^{n+1}_1}$}      
      \put(-55,172){$\scriptstyle{X^{n+1}_2}$}
      \put(-30,184){$\scriptstyle{X^{n+1}_3}$}
      \put(-15,205){$\scriptstyle{X^{n+1}_4}$}
      \put(-4,230){$\scriptstyle{X^{n+1}_5}$}
      \put(-5,253){$\scriptstyle{X^{n+1}_6}$}
      \put(-10,277){$\scriptstyle{X^{n+1}_7}$}
      \put(-30,295){$\scriptstyle{X^{n+1}_8}$}
      \put(-50,310){$\scriptstyle{X^{n+1}_9}$}                                    
      \put(-80,312){$\scriptstyle{X^{n+1}_{10}}$}
      \put(-105,312){$\scriptstyle{X^{n+1}_{11}}$}
      \put(-128,297){$\scriptstyle{X^{n+1}_{12}}$}
      \put(-148,280){$\scriptstyle{X^{n+1}_{13}}$}
      \put(-155,254){$\scriptstyle{X^{n+1}_{14}}$}
      \put(-160,228){$\scriptstyle{X^{n+1}_{15}}$}
      \put(-152,207){$\scriptstyle{X^{n+1}_{16}}$}
      \put(-147,187){$\scriptstyle{X^{n+1}_{2km-1}}$}
      \put(-74,230){$\scriptstyle{v}$}
    \end{picture}
    \caption{Inductive construction of connections.}
    \label{fig:ind_conn}
  \end{figure}
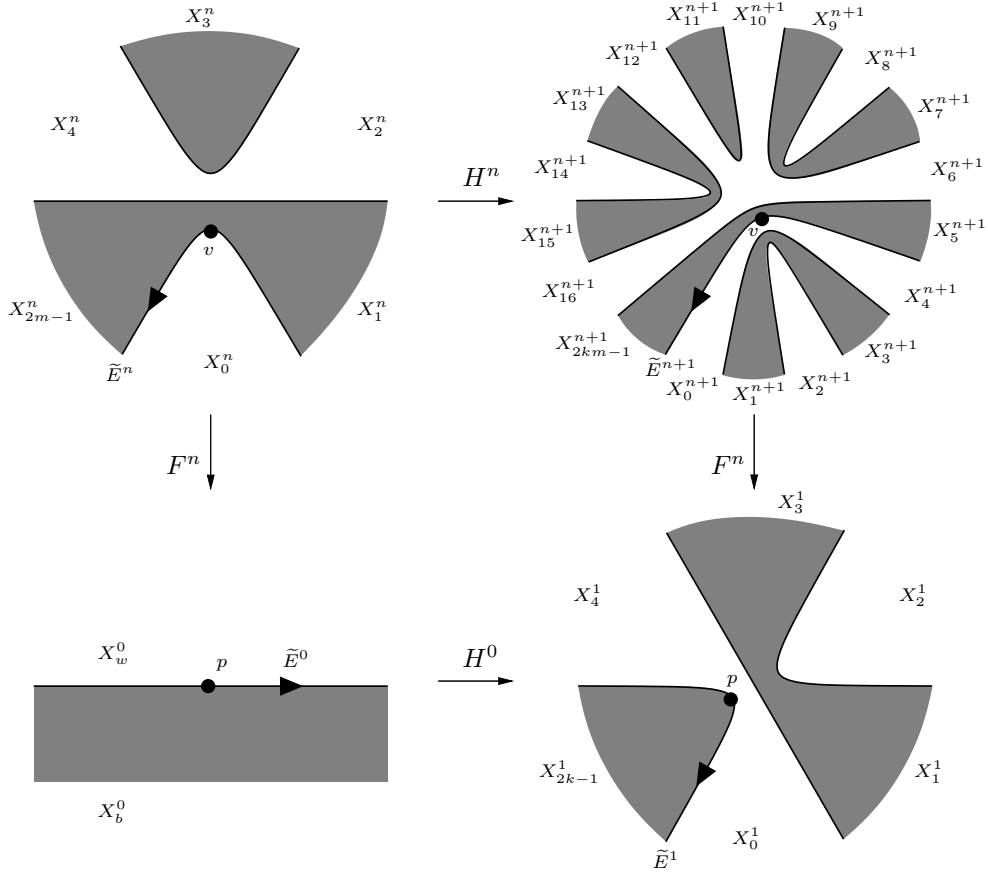

  \smallskip
  We will \emph{formalize} the description above.   To do this, we
  will first have to \emph{label} the involved $1$-tiles, $n$-tiles,
  and $(n+1)$-tiles in a \emph{consistent manner}. See Figure
  \ref{fig:ind_conn} for an illustration.  
  
  Recall from Lemma
  \ref{lem:Hn_maps_edges} that for each $(j+1)$-edge $E^{j+1}$ there
  is a unique arc $A^j$ contained in a $j$-edge $E^j$ that is deformed 
  by the pseudo-isotopy $H^j$ to $E^{j+1}$. Since we will often want
  to keep track of where such an $E^{j+1}$-edge ``comes from'', we use
  the \emph{notation} 
  \begin{equation*}
    H^j\colon A^j\subset E^j\to E^{j+1},
  \end{equation*}
  in this case.

  We will single out one $0$-, $1$-, $n$-, and $(n+1)$-edge. Let
  $\widetilde{E}^0$ be the $0$-edge with initial point $p$
  ($\widetilde{E}^0$ is positively oriented as boundary of the white
  $0$-tile $X^0_w$). The $1$-edge $\widetilde{E}^1$ is the
  \emph{marked} one with initial point $p$. Thus there is an arc
  $\widetilde{A}^0\ni p$, such that $H^0\colon \widetilde{A}^0\subset
  \widetilde{E}^0 \to \widetilde{E}^1$. We choose (arbitrarily) one
  $n$-edge $\widetilde{E}^n\ni v$ such that
  $F^n(\widetilde{E}^n)=\widetilde{E}^0$. Finally we choose the
  $(n+1)$-edge $\widetilde{E}^{n+1}\ni v$, such that there is an
  $n$-arc $\widetilde{A}^n\ni v$ satisfying $H^n\colon
  \widetilde{A}^n\subset \widetilde{E}^n \to \widetilde{E}^{n+1}$. 
 

  Let $2m$ be the number of $n$-tiles containing $v$ (this means that
  $m= \deg_{F^n}(v)$) and $2k$ the number of $1$-tiles containing
  $p$. Then the number of $(n+1)$-tiles containing $v$ is $2km$. 

  \smallskip
  The $1$-tiles $X^1_0,\dots, X^1_{2k-1}$ around $p$, the $n$-tiles $X^n_0,
  \dots, X^n_{2m-1}$ around $v$, and the $(n+1)$-tiles $X^{n+1}_0,\dots,
  X^{n+1}_{2km-1}$ around $v$ are labeled mathematically positively
  (around $p$, $v$ respectively) and such that $\widetilde{E}^1\subset
  X^1_0$, $\widetilde{E}^n\subset X^n_0$, $\widetilde{E}^{n+1}\subset
  X^{n+1}_0$. 

  Recall that white tiles are always labeled
  by even, black tiles by odd indices. Thus $X^1_0,X^n_0,X^{n+1}_0$
  are all \emph{white} tiles. This finishes the labelling.


  



  \smallskip
  The blocks $b^{n+1}$ of the cnc-partition $\pi^{n+1}_w(v)\cup
  \pi^{n+1}_b(v)$ are defined as follows. For each block $b^1\in
  \pi^1_w(v)\cup \pi^1_b(v)$ and each $j=0,\dots, m-1$ there is a
  block 
  \begin{equation}
    \label{eq:defbn+1}
    b^{n+1} = b^{n+1}_{j}(b^1) = b^1 + 2kj = \{ i + 2kj \mid i \in
    b^1\}.  
  \end{equation}
  This corresponds to the first part of the description above.



  \smallskip
  Now let $b^1_{\star}\in \pi^1_w(p)$ be the block
  containing $0$; it contains indices of white $1$-tiles that are
  connected to the marked succeeding $1$-edges at $p$. The sets
  $b^{n+1}_j(b^1_{\star})=b^1_{\star} +2kj$ are defined as in
  (\ref{eq:defbn+1}), they 
  contain indices of $(n+1)$-tiles that are mapped to ($1$-tiles with
  indices in) $b^1_{\star}$ by $F^n$. 
  For each block
  $b^n\in\pi^n_w(v)$ there is a block
  $b^{n+1}_{\star}\in\pi^{n+1}_w(v)$ given by
  \begin{equation}
    \label{eq:defbn+1star}
    b^{n+1}_{\star}= b^{n+1}_{\star}(b^n) := \bigcup
    \{ b^1_{\star} +2kj \mid 2j \in b^n\}.  
  \end{equation}
  This is the formal description of the second part described above.

  In the same fashion let $c^1_{\star}\in \pi^1_b(p)$ be the block
  containing $2k-1$. It contains indices of black $1$-tiles connected
  to the marked succeeding $1$-edges at $p$. For each block
  $c^n\in\pi^n_b(v)$ there is a block $c^{n+1}_{\star}\in
  \pi^{n+1}_b(v)$ given by
  \begin{equation}
    \label{eq:defcn+1star}
    c^{n+1}_{\star}= c^{n+1}_{\star}(c^n) := \bigcup \{
    c^1_{\star} + 2kj \mid 2j+1\in c^n\}.  
  \end{equation}
  
  The cnc-partition $\pi^{n+1}_w(v)\cup \pi^{n+1}_b(v)$ consists of
  all blocks $b^{n+1}_j(b^1)$ as in (\ref{eq:defbn+1}), where $b^1\neq
  b^1_{\star}, c^1_{\star}$; as well as all blocks
  $b^{n+1}_{\star}=b^{n+1}(b^n), c^{n+1}_{\star}=
  c^{n+1}_{\star}(c^n)$ as above.  


\end{case}

\begin{case}[3]
  $v\in \post$.

  Note that $\post= \V^0\subset \V^n$. This case is thus a subcase of
  Case (2). 
  The cnc-partition $\pi^{n+1}_w(v)\cup \pi^{n+1}_b(v)$ is thus
  already constructed in Case (2). 
  It remains to \emph{mark} it.   Recall that in Case (2) the $n$-edge 
  $\widetilde{E}^n$ with $F^n(\widetilde{E}^n)=\widetilde{E}^0$, was
  chosen \emph{arbitrarily}. Now however, we let $\widetilde{E}^n$ be
  the \emph{marked} $n$-edge with initial point $v$.
 
  The marked $(n+1)$-edge with initial point $v$ is
  $\widetilde{E}^{n+1}$ (recall that there is an arc
  $\widetilde{A}^n\ni v$ such that $H^n\colon \widetilde{A}^n\subset
  \widetilde{E}^n\to \widetilde{E}^{n+1}$). 



  Alternatively consider the blocks $b^{n+1}=b^{n+1}(0)\in
  \pi^{n+1}_w(v)$, $c^{n+1}=c^{n+1}(2km-1)\in \pi^{n+1}_b(v)$ such
  that $0\in b^{n+1}$ and $2km-1\in c^{n+1}$. These two adjacent
  blocks mark the connection of $(n+1)$-tiles at $p$ (see Corollary
  \ref{cor:marking}).    
\end{case}

\subsection{Properties of connections}
\label{sec:prop-conn}

Here we prove that the connections of $n$-tiles defined above have the
desired properties.

\begin{prop}
  \label{prop-conn-n}
  The connection  of $n$-tiles as defined in Section
  \ref{sec:connection-n-tiles} satisfies the following.
  \begin{enumerate}
  \item
    \label{item:conn_n_cnc}
    Each $\pi^n_w(v)\cup \pi^n_b(v)$ is a \emph{cnc-partition}.
  \item 
    \label{item:conn_n_C1C2}
    The connection of $n$-tiles satisfies properties {\upshape (C
      \ref{item:prop_conn_1}), (C \ref{item:prop_conn_2})} from
    Definition~\ref{def:prop_conn}.
  \item
    \label{item:conn_n_bdgn}
    The (single) \emph{boundary circuit} of the cluster of white
    $n$-tiles is \emph{equal} to the \emph{$n$-th approximation}
    $\gamma^n$ (viewed as an Eulerian circuit). 
  \end{enumerate}
\end{prop}

\begin{proof}

  To  be able to keep the notation from Section
  \ref{sec:connection-n-tiles} we will prove the statements for the
  connection of $(n+1)$-tiles. 

  \smallskip
  (\ref{item:conn_n_cnc})
  The statement will be proved by induction. Thus we assume that
  $\pi^n_w(w)\cup \pi^n_b(w)$ is a cnc-partition for each $n$-vertex
  $w$. Consider now an arbitrary $(n+1)$-vertex $v$. We want to show
  that $\pi^{n+1}_w(v)\cup \pi^{n+1}_b(v)$ is a cnc-partition. 
  This is trivial in Case (1) (i.e., if $v$ is not an
  $n$-vertex). Thus assume that we are in Case (2), 
  i.e., that $v\in \V^{n+1}\cap \V^n$.   
  
  \smallskip
  (\ref{item:conn_n_cnc}a)
  We first prove that $\pi^{n+1}_w(v)\cup \pi^{n+1}_b(v)$ is
  \emph{non-crossing}. Consider first two blocks
  \begin{equation*}
    b^{n+1}=b^{n+1}_i(b^1),\, c^{n+1}=b^{n+1}_j(c^1)\in
    \pi^{n+1}_w(v)\cup \pi^{n+1}_b(v)     
  \end{equation*}
  as in (\ref{eq:defbn+1}), where $i,j=0,\dots, m-1$ 
  and $b^1,c^1\in \pi^1_w(p)\cup \pi^1_b(p)\setminus \{b^1_{\star},
  c^1_{\star}\}$. If $i\ne j$ the blocks $b^{n+1},c^{n+1}$ are
  non-crossing, since $b^{n+1}, c^{n+1}$ are contained in disjoint
  intervals; namely $b^{n+1}\subset [2ki, 2k(i+1)-1]$,
  $c^{n+1}\subset [2kj, 2k(j+1)-1]$.

  If $i=j$ the blocks $b^{n+1},c^{n+1}$ are non-crossing, since the
  blocks $b^1,c^1$ are. 

  \smallskip
  (\ref{item:conn_n_cnc}b)
  Now let $b^{n+1}=b^{n+1}_i(b^1)$ be as before and $b^{n+1}_{\star}=
  b^{n+1}_{\star}(b^n)=\bigcup \{b^1_{\star} + 2kj \mid 2j\in
  b^n\}$ be as in (\ref{eq:defbn+1star}) (where $b^n\in 
  \pi^n_w(v)$). Assume without loss of generality that $i=0$. Then
  $b^{n+1}$ is contained in one component of $[0, 2k-1]\setminus
  b^1_{\star}$. Each set $b^1_{\star} +2kj$ distinct from
  $b^1_{\star}$ is contained in an interval distinct from $[0,
  2k-1]$. It follows that $b^{n+1}, b^{n+1}_{\star}$ are non-crossing.   

  That $b^{n+1}$ and $c^{n+1}_{\star}$ (as in (\ref{eq:defcn+1star}))
  are non-crossing is shown by the same argument.

  \smallskip
  (\ref{item:conn_n_cnc}c)
  Now let $b^{n+1}_{\star}= b^{n+1}_{\star}(b^n)$ be as before and
  $\tilde{b}^{n+1}_{\star}= b^{n+1}_{\star}(\tilde{b}^n)$ be a
  distinct set as in (\ref{eq:defbn+1star}), meaning that the block
  $\tilde{b}^n\in\pi^n_w(v)$ is distinct from $b^n$. Since
  $b^n,\tilde{b}^n$ are non-crossing it follows that $b^{n+1}_{\star},
  \tilde{b}^{n+1}_{\star}$ are non-crossing. The same argument shows
  that distinct $c^{n+1}_{\star},\tilde{c}^{n+1}_{\star}$ as in
  (\ref{eq:defcn+1star}) are non-crossing. 

  \smallskip
  (\ref{item:conn_n_cnc}d)
  Consider now two sets $b^{n+1}_{\star}= b^{n+1}_{\star}(b^n)$,
  $c^{n+1}_{\star}=c^{n+1}_{\star}(c^n)$ as in (\ref{eq:defbn+1star})
  and (\ref{eq:defcn+1star}) ($b^n\in \pi^n_w(v)$, $c^n\in
  \pi^n_b(v)$). Recall that $\pi^n_w(v)\cup \pi^n_b(v)$ 
  is a cnc-partition by inductive hypothesis. Assume first that
  $b^n,c^n$ are not adjacent (see Lemma \ref{lem:prop_comp}), i.e.,
  they do not contain indices $i$ and $i+1$ respectively. Then it
  follows from the fact that $b^n,c^n$ are non-crossing, that
  $b^{n+1}_{\star}, c^{n+1}_{\star}$ are non-crossing. 

  (\ref{item:conn_n_cnc}e)
  Now let $b^n,c^n$ be adjacent. Recall that $0\in b^1_{\star},
  2k-1\in c^1_{\star}$. Thus there is an index $i^1\in b^1_{\star}$
  such that $i^1+1 \in c^1_{\star}$, since $\pi^1_w(p)\cup \pi^1_b(p)$
  is a cnc-partition. This means that
  \begin{equation*}
    b^1_{\star} \subset [0,i^1], 
    \quad c^1_{\star} \subset [i^1+1, 2k-1]. 
  \end{equation*}
  Similarly, since $b^n,c^n$ are adjacent, there are indices $i^n,
  j^n\in b^n$, such that $i^n+1, j^n-1\in c^n$; meaning that
  \begin{equation*}
    b^n\subset [j^n,i^n],
    \quad
    c^n \subset [i^n+1, j^n-1].
  \end{equation*}
  Here we are using the notation from (\ref{eq:notation_pi1}). From
  this we obtain the smallest and biggest elements in
  $b^{n+1}_{\star}=b^{n+1}_{\star}(b^n),
  c^{n+1}_{\star}=c^{n+1}_{\star}(c^n)$ according to 
  (\ref{eq:defbn+1star}), (\ref{eq:defcn+1star}), namely
  \begin{equation*}
    b^{n+1}_{\star} \subset [j^nk, i^1+i^nk],
    \quad
    c^{n+1}_{\star} \subset [i^1 + i^nk+1, j^nk -1]. 
  \end{equation*}
  Thus $b^{n+1}_{\star}$, $c^{n+1}_{\star}$ are non-crossing. 

  \smallskip
  We now prove that $\pi^{n+1}_w(v), \pi^{n+1}_b(v)$ are
  \emph{complementary}. Let $i^{n+1}= 0,\dots, 2km-1$ be arbitrary. We have
  to show that the two blocks of $\pi^{n+1}_w(v)\cup \pi^{n+1}_b(v)$
  containing $i^{n+1},i^{n+1}+1$ are adjacent. 

  If we are in case
  (\ref{item:conn_n_cnc}a), i.e., if $i^{n+1}\in b^{n+1}=b^{n+1}_i(b^1)$,
  $i^{n+1}+1\in c^{n+1}=b^{n+1}_j(c^1)$, where $b^1,c^1 \in
  \pi^1_w(p)\cup \pi^1_b(p)\setminus \{b^1_{\star}, 
  c^1_{\star}\}$, it follows that $i=j$. Then $b^1,c^1$
  are adjacent, which implies that $b^{n+1}, c^{n+1}$ are adjacent.   

  When we are in case (\ref{item:conn_n_cnc}b) it follows that
  $b^1,b^1_{\star}$ are adjacent. This implies that
  $b^{n+1},b^{n+1}_{\star}$ are adjacent. 

  Cases (\ref{item:conn_n_cnc}c) and (\ref{item:conn_n_cnc}d) cannot
  happen.

  In case (\ref{item:conn_n_cnc}e) it is clear from the description
  that $j^nk, i^1+i^nk\in b^{n+1}_{\star}$ and $i^1 + i^nk+1, j^nk
  -1\in c^{n+1}_{\star}$. Thus $b^{n+1}_{\star},c^{n+1}_{\star}$ are
  adjacent.  

  \medskip
  (\ref{item:conn_n_bdgn})
  Let $D^{n+1},\widetilde{D}^{n+1}$ be two $(n+1)$-edges. We have to
  show that
  \begin{align*}
    &D^{n+1},\widetilde{D}^{n+1} \text{ are succeeding in }
    \gamma^{n+1}\quad \text{if and only if} 
    \\
    &\text{they are succeeding with respect to the
      connection of $(n+1)$-tiles}.
  \end{align*}
  We keep the notation from Section \ref{sec:connection-n-tiles}. Case
  (1) is again clear. Thus we assume that we are in Case (2), meaning
  that $v\in \V^{n+1}\cap \V^n$. Recall that $\widetilde{E}^0$ is the
  $0$-edge with initial point $p=F^n(v)$ and $\widetilde{E}^1\ni p$ the
  marked $1$-edge (some arc $\widetilde{A}^0\subset \widetilde{E}^0$
  containing $p$ is deformed by $H^0$ to $\widetilde{E}^1$). 

  Let
  $\widetilde{E}^0=\widetilde{E}^{n}_0,\dots,
  \widetilde{E}^{n}_{m-1}\ni v$
  be all $n$-edges such that $F^n(\widetilde{E}^n_j)=\widetilde{E}^0$
  (labeled mathematically positively around $v$). 

  Consider the $(n+1)$-edges $\widetilde{E}^{n+1}_{j}$ such that
  $H^n\colon \widetilde{A}^n_j \subset \widetilde{E}^n_j \to
  \widetilde{E}^{n+1}_j$, for some arc $\widetilde{A}^n_j \ni v$.
  These $(n+1)$-edges $\widetilde{E}^{n+1}_0,\dots,
  \widetilde{E}^{n+1}_{m-1}$ are again labeled mathematically
  positively around $v$. Note that these are not all of the
  $(n+1)$-edges containing $v$. 

  \smallskip
  {\it Claim.} $F^n(\widetilde{E}^{n+1}_j)= \widetilde{E}^1$
  for all $j=0,\dots, m-1$. 

  To prove the claim we first note that $F^n(\widetilde{E}^{n+1}_j)$
  is a $1$-edge which we denote by $\widetilde{D}^1$. Since
  $\widetilde{A}^n_j\subset \widetilde{E}^n_j$, the arc 
  $\widetilde{B}^0:= F^n(\widetilde{A}^n_j)$ is contained in
  $\widetilde{E}^0=F^n(\widetilde{E}^n_j)$, with initial point
  $p=F^n(v)$. Since $H^n$ is the lift of $H^0$ by $F^n$ it holds 
  \begin{equation*}
    \widetilde{D}^1= F^n(\widetilde{E}^{n+1}_j)=
    F^n(H^n_1(\widetilde{A}^n_j)) = H^0_1(F^n(\widetilde{A}^n_j))=
    H^0_1(\widetilde{B}^0).
  \end{equation*}
  The unique arc in $\widetilde{E}^0$ with initial point $p$ that is
  deformed to a $1$-edge is $\widetilde{A}^0$. Thus
  $\widetilde{B}^0=\widetilde{A}^0$, thus $\widetilde{D}^1=
  \widetilde{E}^1$, proving the claim.

  \smallskip
  Note that a sector of sufficiently small radius between
  $\widetilde{E}^{n+1}_j,\widetilde{E}^{n+1}_{j+1}$ is mapped
  \emph{bijectively} by $F^n$ to some neighborhood of $p$ with
  $\widetilde{E}^1$ removed. 

  \smallskip
  Assume now that the $(n+1)$-edges $D^{n+1},\widetilde{D}^{n+1}$
  are succeeding in $\gamma^{n+1}$ at the $(n+1)$-vertex $v$. This is
  the case if and only if there 
  are distinct arcs $A^n,\widetilde{A}^n\ni x$ such that $H^n\colon
  A^n \subset D^n \to D^{n+1}$, $H^n\colon \widetilde{A}^n \subset
  \widetilde{D}^n \to \widetilde{D}^{n+1}$ ($D^n,\widetilde{D}^n \in
  \E^n$).  
  Either
  \begin{itemize}
  \item $A^n,\widetilde{A}^n$ are contained in the \emph{same}
    $n$-edge, equivalently $x\notin \V^n$. Note that
    $\widetilde{D}^{n+1}\neq \widetilde{E}^{n+1}_j$ for all
    $j=0,\dots, m-1$. 

    Note that $H^n_1(x)=v$. If $D^{n+1}= \widetilde{E}^{n+1}_j$ (for
    a $j=0,\dots, m-1$) it would follow that both endpoints of
    $\widetilde{E}^{n+1}_j$ are equal to $v$, which is impossible. 

    It follows that $D^{n+1}, \widetilde{D}^{n+1}$ are
    contained in one sector between $\widetilde{E}^{n+1}_j,
    \widetilde{E}^{n+1}_{j+1}$, since $H^n$ is a pseudo-isotopy; 
  \item or $x=v$ and $A^n,\widetilde{A}^n$ are contained in $n$-edges
    that succeed at $v$. Then $\widetilde{D}^{n+1}=
    \widetilde{E}^{n+1}_j$ for some $j=0,\dots, m-1$ in this case. 
  \end{itemize}

  Consider two $(n+1)$-edges $D^{n+1},\widetilde{D}^{n+1}\ni v$, such
  that 
  $\widetilde{D}^{n+1}\neq \widetilde{E}^{n+1}_j$ (for all $j=0,
  \dots, m-1$).  
  They are succeeding in $\gamma^{n+1}$ at 
  $v$ if and only if they are contained in one sector between
  $\widetilde{E}^{n+1}_{j}, \widetilde{E}^{n+1}_{j+1}$ and the 
  $1$-edges $F^n(D^{n+1}),
  F^n(\widetilde{D}^{n+1})$ are succeeding in $\gamma^{1}$ (since
  $F^n$ is bijective on this sector). This
  happens if and only if $D^{n+1}, \widetilde{D}^{n+1}$ are succeeding
  with respect to $\pi^{n+1}_w(v)\cup \pi^{n+1}_b(v)$ by definition
  (see (\ref{eq:defbn+1})). 

  \smallskip
  Let $E^0\ni p$ be the $0$-edge with terminal point $p$, i.e., the
  one preceding $\widetilde{E}^0$. 
  Let $E^n_0,\dots, E^n_{m-1}$ be all
  $n$-edges such that $F^n(E^n_j)=E^0$, labeled such that $E^n_j$ lies
  between $\widetilde{E}^n_j, \widetilde{E}^n_{j+1}$. Then
  $\widetilde{E}^n_j,E^n_j$ are both contained in the same white
  $n$-tile $X^n_j$. Thus $E^n_i, \widetilde{E}^n_j$ are succeeding (at
  $v$) if and only if $i,j$ are succeeding indices of a block $b^n\in
  \pi^n_w(v)$. 

  Consider the $1$-edge $E^1$ such that $H^0\colon A^0\subset E^0 \to
  E^1$, for an arc $A^0\ni p$. Let $X^1_l$ be the white $1$-tile
  containing $E^1$.  
  Now consider the $(n+1)$-edge $E^{n+1}_j$  such that $H^n\colon A^n_j
  \subset E^n_j \to E^{n+1}_j$, for an arc $A^n_j\ni v$. 
  Since $H^n$ is a pseudo-isotopy it follows that $E^{n+1}_j$ is in
  the sector between
  $\widetilde{E}^{n+1}_j,\widetilde{E}^{n+1}_{j+1}$; indeed
    it follows that $E^{n+1}_j\subset X^{n+1}_{2kj+l}$, since the 
  diagram in Figure \ref{fig:ind_conn} commutes (recall that
  $\widetilde{E}^{n+1}_j\subset X^{n+1}_{2kj}$).

  Consider now two $(n+1)$-edges $D^{n+1},
  \widetilde{D}^{n+1}=\widetilde{E}^{n+1}_j\ni v$. They are succeeding
  in $\gamma^{n+1}$ if and only if $D^{n+1}= E^{n+1}_{i}\subset
  X^{n+1}_{2ki+l}$, where $i,j$ are succeeding indices of a block
  $b^n\in \pi^n_w(v)$. This happens if and only if they are succeeding
  with respect to $\pi^{n+1}_w(v)\cup \pi^{n+1}_b(v)$ by definition
  (see (\ref{eq:defbn+1star})) (in the notation from
  (\ref{item:conn_n_cnc}e) $i=i^n, j=j^n, l=i^1$). 

  \smallskip
  (\ref{item:conn_n_C1C2}) follows as in Section
  \ref{sec:glimit_jordan}. 
\end{proof}

\section{Invariant Peano curve implies Expansion}
\label{sec:invar-peano-curve}

In this section we prove Theorem~\ref{thm:peano_implies_exp}.
Thus we assume that for some iterate $F=f^n$ there is a Peano
curve $\gamma\colon S^1\to S^2$ (onto), such that $F(\gamma(z))=
\gamma(z^d)$ for all $z\in S^1$ (where $d=\deg F$). We want to show
that $f$ is expanding. 

The following is \cite[Lemma 6.3]{expThurMarkov}.
\begin{lemma}
  \label{lem:fFexpanding}
  Let $f$ be a Thurston map and $F=f^n$, where $n\in \N$. Then $f$
  is expanding if and only if $F$ is expanding.
\end{lemma}



We will use the following equivalent formulation of ``expanding'' due to
Ha\"{i}ssinsky-Pilgrim \cite{HaiPil}. For a proof of the following
lemma we refer the reader to \cite[Proposition~6.2]{expThurMarkov}. 

\begin{lemma}
  \label{lem:expaning}
  A Thurston map $F$ is \emph{expanding} if and only if there exists a
  finite open cover $\mathcal{U}^0$ of $S^2$ by connected sets such
  that the following holds.

  Denote by $\mathcal{U}^n$ the set of connected components of
  $F^{-n}(U)$, for all $U\in \U^0$. Then
  \begin{equation*}
    \mesh \U^n\to 0 \text{ as } n\to \infty.
  \end{equation*}
  Here $\mesh \U^n$ denotes the biggest diameter of a set in $\U^n$. 
\end{lemma}

\begin{proof}
  [Proof of Theorem~\ref{thm:peano_implies_exp}]
  Let $\gamma\colon S^1\to S^2$ be a Peano curve (onto), such that
  \begin{equation}
    \label{eq:semi_conjugacy}
    F(\gamma(z))= \gamma(z^d) \text{ for all } z\in S^1 \quad (\text{where }
    d=\deg F).  
  \end{equation}

  Fix a point $x^0\in S^2$. Let $W(x^0)\subset S^2$ be an open
  neighborhood 
  of $x^0$ that is a Jordan domain. Furthermore we assume that
  $W(x^0)$ is sufficiently small such that each component of
  $F^{-1}(W(x^0))$ contains exactly one point of $F^{-1}(x^0)$. 

  \smallskip
  Consider $\gamma^{-1}(W(x^0))=: \I(x^0)=
  \bigcup I_j\subset S^1$, this is a (countable) union of open arcs
  $I_j$. Let  
  \begin{align*}
    &\J(x^0):= \bigcup \{I_j \mid \gamma(I_j)\ni x^0\}\subset
    S^1,
    \\
    &V(x^0):= \gamma(\J(x^0))\subset S^2. 
  \end{align*}
  Note that $\gamma(S^1\setminus \J(x^0))$ is a compact set
  that does not contain $x^0$. Thus $V(x^0)$ is a neighborhood of
  $x^0$. 

  \smallskip
  Fix a $x^n\in F^{-n}(x^0)$. Let $V^n(x^n)\subset S^2$ be the path
  component of $F^{-n}(V(x^0))$ containing $x^n$.

  As before we view the circle as $\R/\Z$, the map $z\mapsto z^d$ is
  then 
  given as 
  $\phi_d\colon \R/\Z\to \R/\Z$,  $t\mapsto dt (\bmod 1)$. 
  Let $\J^n:=
  \phi_{d^{n}}^{-1}(\J(x^0))$. Note that $\J^n=\bigcup J^n_j$ is a 
  (countable) union of open intervals, each of which has length $\leq
  d^{-n}$. Thus uniform continuity of $\gamma$ implies that
  \begin{equation*}
    \diam \gamma(J^n_j) \leq \omega(d^{-n}) \to 0 \text{ as } n\to
    \infty,  
  \end{equation*}
  where $\omega$ is the modulus of continuity of $\gamma$. 

  From (\ref{eq:semi_conjugacy}) it follows that each set
  $\gamma(J^n_j)$ contains a point $x^n_j\in F^{-n}(x^0)$. If
  $x^n_j\neq x^n$ then $\gamma(J^n_j)$ is contained in a component of
  $F^{-n}(W(x^0))$ distinct from the one containing $x^n$, thus
  $\gamma(J^n_j)\cap V^n(x^n)=\emptyset$. It follows
  that 
  \begin{equation*}
    \gamma^{-1}(V^n(x^n))= \bigcup \{J^n_j \mid \gamma(J^n_j)\ni x^n\}
    =: \J^n(x^n). 
  \end{equation*}
  Since $\gamma(J^n_i)\cap \gamma(J^n_j)\ni x^n$ for $J^n_i,
  J^n_j\subset \J^n(x^n)$, it follows that
  \begin{equation*}
    \diam V^n(x^n)\leq 2 \omega(d^{-n}).
  \end{equation*}
  The sets $V^0(x^0)$ are not necessarily open, and $\inte V^0(x^0)$
  is not necessarily connected. Let $U(x^0)\subset V^0(x^0)$ be an
  open connected set containing $x^0$. Pick a finite subcover
  $\U^0$ of
  $\{U(x^0) | x^0\in S^2\}$. From the above it follows that $\mesh
  \U^n \to 0$ as $n\to \infty$. Thus $F$ is expanding by
  Lemma~\ref{lem:expaning}, hence $f$ is expanding by
  Lemma~\ref{lem:fFexpanding}.  
\end{proof}

\section{An Example}
\label{sec:an-example}
The obvious question to ask is whether an iterate $F=f^n$ is necessary
in Theorem~\ref{thm:main} (or whether one may choose $n=1$). None of
the assumptions in Section~\ref{sec:construction-h0} seem to be
necessary. It is possible to show (similarly as in
\cite[Example~13.12]{expThurMarkov}) that the map $f$ for which Milnor
constructs an invariant Peano curve in \cite{MilnorMating} does not
have an invariant Jordan curve 
$\CC\supset \post$; also the $1$-tiles do intersect
disjoint $0$-edges.    

In this section we consider an example of an expanding Thurston map
$h$, where no pseudo-isotopy $H^0$ as desired exists. 
This means that for any Jordan curve $\CC\supset \post$ (not
necessarily invariant) there is no pseudo-isotopy $H^0$ rel\
$\post(h)$ as in
Definition \ref{def:pseudo-isotopy-h0} such that
$H^0_1(\CC)=\bigcup \E^1= h^{-1}(\CC)$. 

Thus one
has to take an iterate (in fact $h^2$ will do) in our
construction. Of course there could be a Peano curve $\gamma$ which
semi-conjugates $z^d$ to $h$, but a substantially different proof would be
required. 

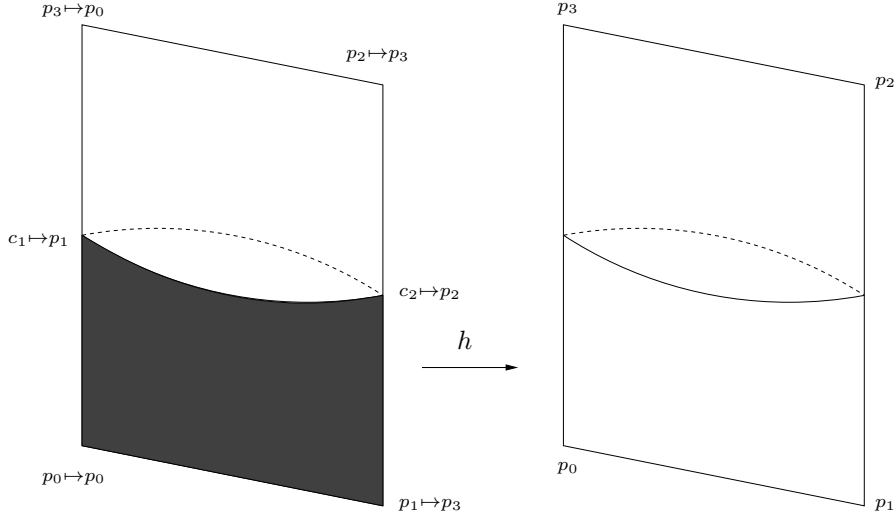
\begin{figure}
  \centering
  \includegraphics[scale=0.4]{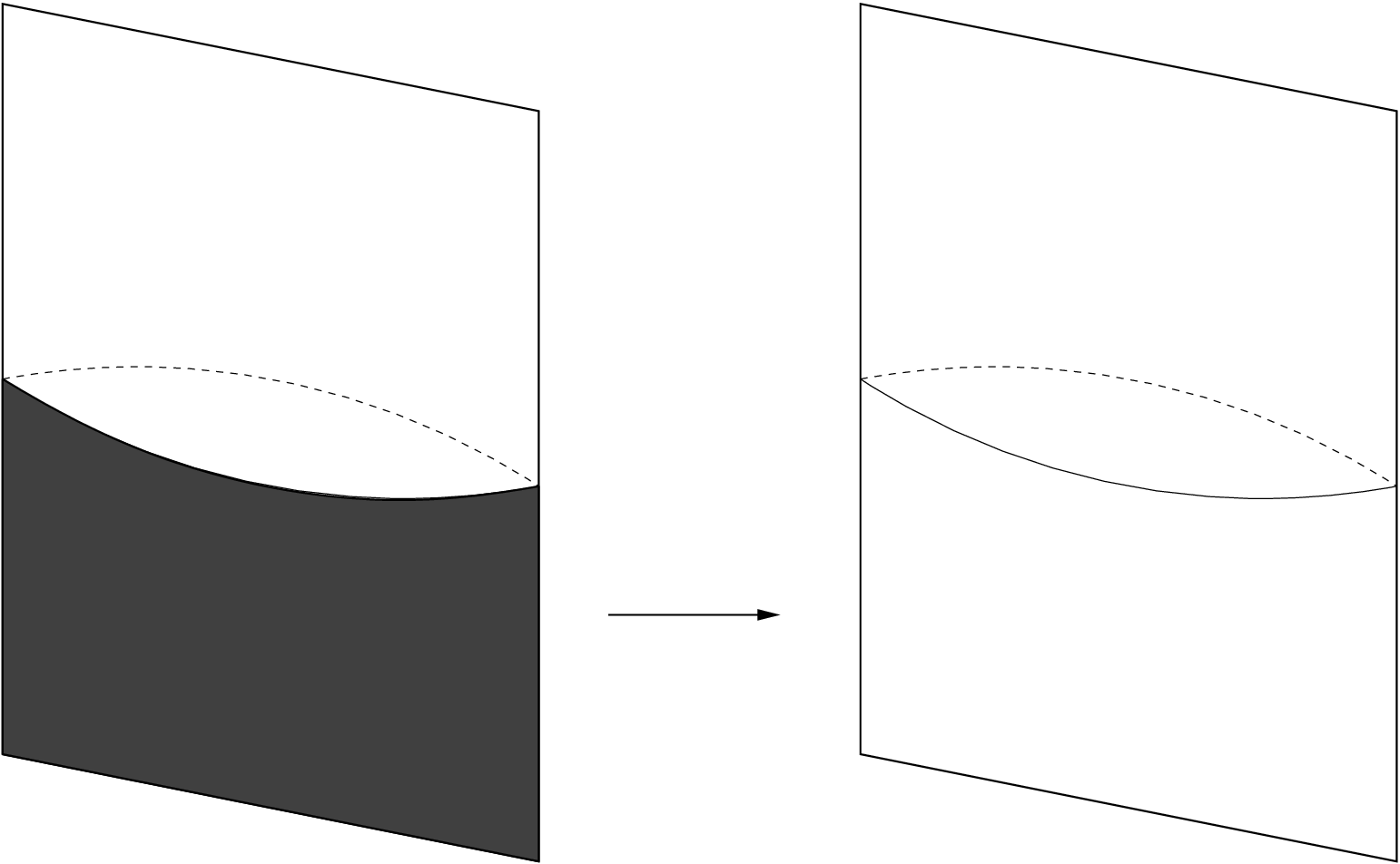}
  \begin{picture}(10,10)
    \put(-315,10){$\scriptstyle p_0\mapsto p_0$}
    \put(-328,100){$\scriptstyle c_1\mapsto p_1$}
    \put(-315,187){$\scriptstyle p_3\mapsto p_0$}
    \put(-200,170){$\scriptstyle p_2\mapsto p_3$}
    \put(-180,80){$\scriptstyle c_2 \mapsto p_2$}
    \put(-180,0){$\scriptstyle p_1\mapsto p_3$}
    \put(-158,60){$h$}
    \put(-120,14){$\scriptstyle p_0$}
    \put(0,0){$\scriptstyle p_1$}
    \put(0,160){$\scriptstyle p_2$}
    \put(-120,188){$\scriptstyle p_3$}
  \end{picture}
  \caption{The map $h$.}
  \label{fig:maph}
\end{figure}

\smallskip
The map $h$ is a \defn{Latt\`{e}s map} as the map $g$ from Section
\ref{sec:example}. 
Start with the square $[0,\sqrt{2}/2]\times [0,1]$, which is mapped by
a Riemann 
map to the upper half plane. This extends to a meromorphic map
$\wp=\wp_L\colon \C\to \CDach$, which is periodic with respect to the
lattice 
$L=\sqrt{2}\Z\times 2\Z$. Consider the map 
\begin{equation}
  \label{eq:defpsi_h}
  \psi\colon \C\to \C,
  \quad
  \psi(z)=\sqrt{2}iz.   
\end{equation}
Note that $\psi(L)\subset L$. 
The map $h$ is the one that makes the following
diagram commute.  
\begin{equation*}
  \xymatrix{
    \C \ar[r]^\psi \ar[d]_{\wp} &
    \C \ar[d]^{\wp}
    \\
    S^2 \ar[r]_h & S^2
  }
\end{equation*}
The degree of $h$ is $2$. 
Again one may use $\wp$ to push the Euclidean metric from $\C$ to the
sphere $S^2$. 
In this metric 
the upper and lower half plane are both isometric to the
rectangle $[0,\sqrt{2}/2]\times [0,1]$. Two such rectangles glued
together along their boundaries form a \defn{pillow} as before. 
Divide each rectangle
horizontally in two. The small rectangles are similar to the big
ones. The map $h$ is given by mapping each small rectangle (they are
the $1$-tiles) to big ones (the $0$-tiles)
as indicated in Figure~\ref{fig:maph}. The critical points are
$c_1,c_2$, the postcritical points are $p_0,p_1,p_2,p_3$; they are
mapped as follows (this is known as the \defn{ramification portrait}).
\begin{equation}
  \label{eq:ramification_h}
  \xymatrix @R=1pt{
    c_1 \ar[r]^{2 : 1} & p_1 \ar[dr] & &
    \\
    & & p_3 \ar[r] & p_0 \ar@(r,u)[]
    \\
    c_2 \ar[r]^{2 : 1} & p_2 \ar[ur] & &
  }
\end{equation}

\begin{lemma}
  \label{lem:noH0forh}
  Let 
  $\gamma^0=\CC\supset\post(h)$ be (any such) Jordan curve, and
  $\gamma^1$ be an Eulerian circuit in $h^{-1}(\CC)$ such that $h\colon
  \gamma^1\to \gamma^0$ is a $d$-fold cover. Then there is no
  pseudo-isotopy $H^0$ rel.\ $\post(h)$ as in Definition
  \ref{def:pseudo-isotopy-h0} that deforms $\gamma^0$ to $\gamma^1$.   
\end{lemma}

\begin{proof}[Sketch of Proof]
  The proof is a (rather tedious) case by case analysis. There are
  however only two cases that are essentially different. One of each
  is presented.  

  \begin{case}[1]
    The curve $\CC$ goes through $p_0,p_1,p_2,p_3$ (in this cyclic
    order). 
    
    \smallskip
    We fix an orientation of $\CC$. 
    Let $U_w,U_b$ be the two components of $S^2\setminus \CC$, where
    the positively oriented boundary of $U_w$ is $\CC$.
    The closures of $U_w, U_b$ are the white/black $0$-tiles
    $X^0_w=U_w\cup \CC$, $X_b=U_b\cup \CC$ as before. 
    Similarly  
    we define the (white) $1$-tiles as closures of components of
    $h^{-1}(U_w)$. 

    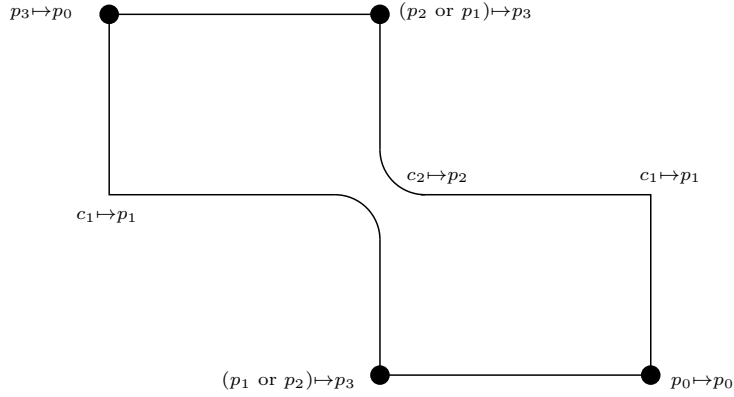
\begin{figure}
      \centering
      \includegraphics[scale=0.6]{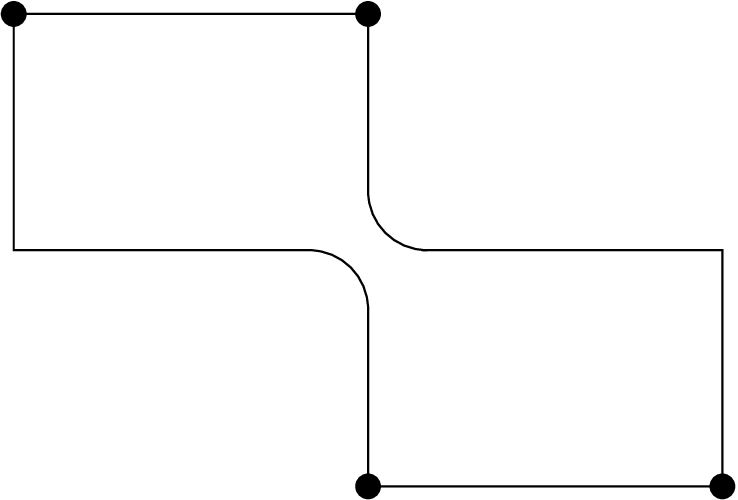}
      \begin{picture}(10,10)
        \put(0,0){$\scriptstyle p_0\mapsto p_0$}
        \put(-12,78){$\scriptstyle c_1\mapsto p_1$}
        \put(-100,78){$\scriptstyle c_2\mapsto p_2$}
         \put(-103,140){$\scriptstyle (p_2 \text{ or } p_1) \mapsto p_3$}
         \put(-250,140){$\scriptstyle p_3\mapsto p_0$}
         \put(-225,63){$\scriptstyle c_1\mapsto p_1$}
         \put(-170,0){$\scriptstyle (p_1\text{ or } p_2) \mapsto p_3$}
      \end{picture}
      \caption{An Eulerian circuit in $h^{-1}(\CC)$ (Case (1)).}
      \label{fig:Eulerh}
    \end{figure}

    \smallskip
    Since the degree of $h$ is $2$,
    there are two white $1$-tiles. They intersect at the critical
    points $c_1,c_2$. 
    The boundary of each $1$-tile contains $4$ points that are
    mapped to $p_0,p_1,p_2,p_3$ (in this cyclic order). 
    There are two
    different Eulerian circuits $\gamma^1$ in $h^{-1}(\CC)$ such that
    $h\colon \gamma^1\to \gamma^0$ is a $2$-fold cover. They
    correspond to 
    connecting the two $1$-tiles either at $c_1$ or at $c_2$. One
    situation (connection at $c_2$) is shown in Figure
    \ref{fig:Eulerh}. Note that the cyclic ordering of the
    postcritical points (shown as dots) is different from the one on
    $\CC$. Thus there is no pseudo-isotopy $H^0$ as desired
    that deforms $\CC=\gamma^0$ to $\gamma^1$.

  \end{case}
  
  When $\CC$ goes through the postcritical points in the order
  $(p_0,p_2,p_1,p_3)$, $(p_0,p_3,p_1,p_2)$, $(p_0,p_3,p_2,p_1)$ the same
  argument works.

  \begin{case}[2] The curve $\CC$ goes through $p_0,p_1,p_3,p_2$ (in this
    cyclic order). The $0$- and $1$-tiles are defined and colored as
    before (see Section \ref{sec:thurston-maps-as}).

    \smallskip
    As before there are two different Eulerian circuits $\gamma^1$ in
    $h^{-1}(\CC)$, such that $h\colon \gamma^1 \to \gamma^0$ is a
    $2$-fold cover. They correspond to whether the white $1$-tiles are
    connected at $c_1$ or $c_2$. Assume they are connected at
    $c_2$. The argument when they are connected at $c_1$ is again
    completely analog.  
    
    \smallskip
    Assume that the pseudo-isotopy $H^0$ is as in Definition
    \ref{def:pseudo-isotopy-h0}. 
    Then $H^0$ deforms (the white $0$-tile) $X^0_w$ to the two
    $1$-tiles.


    \smallskip
    In the following we work in the (orbifold) covering. Recall that
    $X^0_w,X^0_b\subset S^2$ are the white/black $0$-tiles (given by 
    $\CC$). 
    Pull this tiling back by $\wp$ to a tiling of $\C$. 
    More precisely, a $0$-tile $\widetilde{X}\subset\C$ is the closure of
    one component of $\wp^{-1}(U_{w,b})$. 
    Similarly as in the proof of (\ref{eq:fnXntoXhomeo}) one shows that $\wp\colon
    \widetilde{X}\to X_{w,b}$ is a homeomorphism. 
    We color one such $0$-tile $\widetilde{X}\subset \C$ white/black
    if it is the preimage of $X^0_w,X^0_b$.  
    This gives
    a tiling of the plane $\C$ into white/black $0$-tiles. 
    
    Recall that the
    ramification points of $\wp$ are the points in
    $\sqrt{2}/2\Z\times \Z$. 
    At each such ramified point $c\in\sqrt{2}/2\Z\times\Z$ two white and two
    black tiles intersect. Furthermore the
    map $\wp$ is symmetric with respect to each such point. This means
    that $\wp(c+z)=\wp(c-z)$ for all $z\in \C$. Thus
    the tiling of $\C$ is pointwise symmetric with respect to each such
    point $c$.

    We now define the $1$-tiles in $\C$. They
    may be obtained in two different ways; either as preimages
    of $1$-tiles in $S^2$ by $\wp$, or as preimages of $0$-tiles
    $\widetilde{X}\subset\C$ by $\psi$ (\ref{eq:defpsi_h}).

    Fix one white $0$-tile $\widetilde{X}\subset \C$. 
    Note 
    that $\widetilde{X}$ has $4$ vertices
    $\tilde{p}_0,\tilde{p}_1,\tilde{p}_2,\tilde{p}_3\in
    \sqrt{2}/2\Z\times \Z$, they are mapped by $\wp$ to
    $p_0,p_1,p_2,p_3$. We can assume that $\tilde{p}_0=0$.


    As in Lemma \ref{lem:lift_degenerate_isotopies} the pseudo-isotopy
    $H^0$ lifts to a 
    pseudo-isotopy (rel.\ $\sqrt{2}/2\Z\times \Z$)
    $\widetilde{H}^0\colon \C\times [0,1]\to \C$. 
    Note that $\widetilde{H}^0$
    deforms $\widetilde{X}$ to
    two $1$-tiles (in $\C$) connected at a point $\tilde{c}_2$. Here
    $\wp(\tilde{c}_2)=c_2$.    

    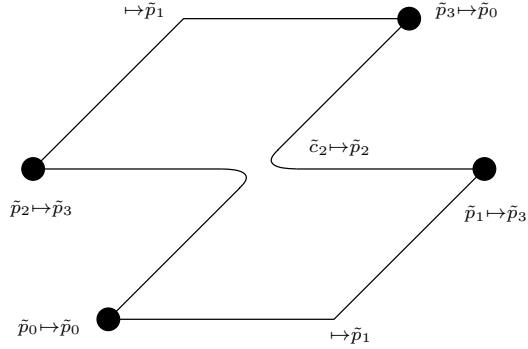
\begin{figure}
      \centering
      \includegraphics[scale=0.5]{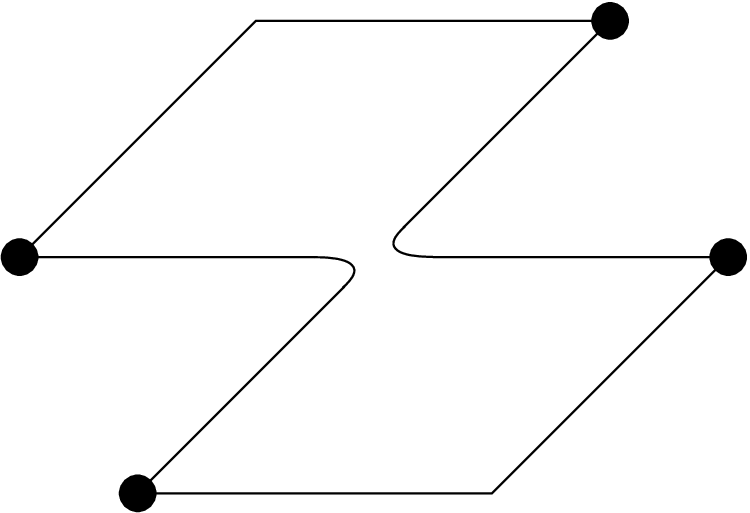}
      \begin{picture}(10,10)
        \put(-185,0){$\scriptstyle\tilde{p}_0\mapsto
          \tilde{p}_0$}
        \put(-67,-3){$\scriptstyle\mapsto \tilde{p}_1$}
        \put(-16,43){$\scriptstyle
          \tilde{p}_1\mapsto\tilde{p}_3$}
        \put(-75,68){$\scriptstyle\tilde{c}_2\mapsto
          \tilde{p}_2$} 
        \put(-27,120){$\scriptstyle\tilde{p}_3\mapsto
          \tilde{p}_0$}
        \put(-145,120){$\scriptstyle \mapsto\tilde{p}_1 $}
        \put(-188,45){$\scriptstyle\tilde{p}_2\mapsto\tilde{p}_3$}
      \end{picture}
      \caption{Eulerian circuit (Case (2)).}
      \label{fig:Eulerh2}
    \end{figure}

    \smallskip
    The ordering of the postcritical points along $\CC$ together with
    (\ref{eq:ramification_h}) implies that the situation looks as in
    Figure \ref{fig:Eulerh2}. Here ``$\mapsto \tilde{p}_j$'' labels a
    point $\tilde{z}$ that satisfies $h(\wp(\tilde{z}))=p_j$.   

    The symmetry of the $1$-tiles with respect to the point
    $\tilde{c}_2$ implies that
    \begin{equation*}
      \label{eq:hsymm}
      2\tilde{c}_2 = \tilde{p}_3=\tilde{p}_1 + \tilde{p}_2. 
    \end{equation*}

    Note that $\tilde{c}_2,\tilde{p}_1$ are contained in the same
    $1$-tile $\widetilde{X}^1$, which contains $\tilde{p}_0=0$. There
    are two $0$-tiles containing $\tilde{p}_0$, symmetric with respect
    to the origin. Thus $\pm
    \psi(\widetilde{X}^1)=\pm
    \sqrt{2}i\widetilde{X}^1=\widetilde{X}$. Therefore
    \begin{align*}
      \label{eq:hsymm2}
      &\pm \sqrt{2}i \tilde{c}_2=\tilde{p}_2
      \\
      &\pm \sqrt{2} i \tilde{p}_1=\tilde{p}_3.
      \intertext{Combining these three equations yields}
      \tilde{p}_2=\pm \sqrt{2}i \tilde{c}_2
      &=
      \pm \frac{\sqrt{2}}{2}i \tilde{p}_3 
      =
      \pm \frac{\sqrt{2}}{2} i \left(\pm \sqrt{2} i \tilde{p}_1\right)
      =
      -\tilde{p}_1.
      \intertext{Thus}
      &\tilde{p}_3=\tilde{p}_1 + \tilde{p}_2=0.
    \end{align*}
    This is a contradiction.
 
    \smallskip
    If $\CC$ goes through the postcritical points in the cyclical order
    $p_0,p_2,p_3,p_1$ the argument is completely analog to the one above. 
  \end{case}
\end{proof}

\section{Open Problems and concluding remarks}
\label{sec:open-probl-concl}

A rational map of degree $d$ can naturally be viewed as a point in
$\C^{2d+1}$ via its coefficients.
Consider a postcritically finite rational map $f$ without periodic
critical points. This is an expanding Thurston map in our sense, the
Julia set is all of $S^2$. 
M. Rees has shown that such a map can be disturbed in a set of
positive measure (in $\C^{2d+1}$) such that the Julia set stays
$S^2$ \cite{0611.58038}. 

\begin{open}
  Let $f$ be a rational map with Julia set $S^2$. Does Theorem
  \ref{thm:main} hold in this case?
\end{open}
   
On the other hand one may ask if the theorem continues to hold if the
Julia set is not the whole sphere. This however is false. Namely
Kameyama gives an example of a postcritically finite rational map where
no such semi-conjugacy exists (see Section 4 in \cite{MR1961296}). 

\smallskip
Finally one can ask if a corresponding result holds in the group case.
\begin{open}
  Let $\Gamma$ be a Gromov-hyperbolic group whose boundary at infinity
  is $S^2$. Is there a Peano curve $\gamma\colon S^1\to S^2$
  invariant under a non-trivial normal subgroup of $\Gamma$? 
\end{open}
A positive answer might conceivably open another line of attack
on \emph{Cannon's conjecture}.

\bibliographystyle{alpha}
\bibliography{main}

\end{document}